\documentclass[10pt,a4paper]{article}
\usepackage[utf8]{inputenc}
\usepackage{amsmath}
\usepackage{amsfonts}
\usepackage{amssymb}
\usepackage{amsthm}
\usepackage{enumitem}
\usepackage{mathtools}
\usepackage{hyperref}

\title{Distribution of Ruelle resonances for real-analytic Anosov diffeomorphisms.}
\author{Malo Jézéquel\footnote{Malo Jézéquel, CNRS, Univ. Brest, UMR 6205, Laboratoire de Mathématiques de Bretagne Atlantique, France. Email: malo.jezequel@math.cnrs.fr}}
\date{}

\newcommand{\set}[1]{\left\{ #1 \right\}}
\newcommand{\p}[1]{\left(#1\right)}
\newcommand{\n}[1]{\left\|#1\right\|}
\newcommand{\va}[1]{\left|#1\right|}
\newcommand{\brac}[1]{\langle #1 \rangle}

\DeclareMathOperator{\Diff}{Diff}
\DeclareMathOperator{\Anos}{Anos}
\DeclareMathOperator{\supp}{supp}
\DeclareMathOperator{\tr}{tr}
\DeclareMathOperator{\im}{Im}
\DeclareMathOperator{\re}{Re}
\DeclareMathOperator{\SL}{SL}
\DeclareMathOperator{\Id}{Id}

\newtheorem{thm}{Theorem}
\newtheorem{lemma}{Lemma}[section]
\newtheorem{proposition}[lemma]{Proposition}
\newtheorem{corollary}[thm]{Corollary}
\theoremstyle{remark}
\newtheorem{remark}[lemma]{Remark}
\theoremstyle{definition}
\newtheorem{definition}[lemma]{Definition}

\begin{document}
\maketitle
\begin{abstract}
We prove an upper bound for the number of Ruelle resonances for Koopman operators associated to real-analytic Anosov diffeomorphisms: in dimension $d$, the number of resonances larger than $r$ is a $\mathcal{O}(|\log r|^d)$ when $r$ goes to $0$. For each connected component of the space of real-analytic Anosov diffeomorphisms on a real-analytic manifold, we prove a dichotomy: either the exponent $d$ in our bound is never optimal, or it is optimal on a dense subset. Using examples constructed by Bandtlow, Just and Slipantschuk, we see that we are always in the latter situation for connected components of the space of real-analytic Anosov diffeomorphisms on the $2$-dimensional torus.
\end{abstract}
\section{Introduction}

Anosov diffeomorphisms are extensively studied hyperbolic dynamical systems. Due to their chaotic properties, one often studies the global or statistical behavior of these maps (rather than the pointwise dynamics). For instance, it is well-established that Anosov diffeomorphisms have many invariant measures with rich ergodic properties (see for instance Bowen's textbook \cite{bowen_book}). A fruitful approach to the study of the statistical properties of dynamical systems is the so-called \emph{functional approach}: instead of studying the dynamics itself, one considers associated composition operators (sometimes referred as the Koopman and (Ruelle) transfer operators). This field has been very active in the last two decades. Indeed, the appearance of the notion of \emph{spaces of anisotropic distributions} adapted to hyperbolic dynamics made possible to sharpen the understanding of the spectral properties of the associated composition operators. This line of work has been developed by many authors, see for instance \cite{blank_keller_liverani, gouezel_liverani_1, gouezel_liverani_2, baladi_tsujii, baladi_tsujii_determinant, faure_roy_sjostrand} in the context of Anosov diffeomorphisms. The textbook \cite{baladi_book} gives an introduction to the functional approach of hyperbolic dynamical systems, discussing modern techniques and including recent references. The interested reader may also refer to \cite{baladi_quest} for a discussion of the different approach of the construction of spaces of anisotropic distributions adapted to hyperbolic dynamics.

The notion of \emph{Ruelle resonances} is central in the functional approach to Anosov diffeomorphisms. Ruelle resonances are eigenvalues of composition operators that are relevant to understand statistical properties of Anosov diffeomorphisms. The goal of this paper is to discuss upper and lower bounds on the number of Ruelle resonances for \emph{real-analytic} Anosov diffeomorphisms. Before stating our main results, we will need to recall some basic facts about the spectral properties of composition operators associated to Anosov diffeomorphisms.

\subsection{Functional approach to Anosov diffeomorphisms}

Let us give more details on the functional approach of statistical properties of hyperbolic diffeomorphism. If $F$ is a $C^\infty$ Anosov diffeomorphism (see \S \ref{subsection:Anosov_diffeomorphisms} for the definition) on a compact manifold $M$, one may understand the statistical properties of $F$ by studying the associated \emph{weighted Koopman operators} defined by 
\begin{equation}\label{eq:def_Koopman}
\mathcal{L}_{F,w} : u \mapsto w. u \circ F,
\end{equation} 
where $w$ is a $C^\infty$ function on $M$, and $\mathcal{L}_{F,w}$ acts on functions or distributions on $M$ for instance. By letting the operator $\mathcal{L}_{F,w}$ act on suitable spaces of anisotropic distributions (adapted to the geometry of the hyperbolic splitting of $F$), one can define a notion of spectrum for $\mathcal{L}_{F,w}$: the \emph{Ruelle spectrum} (whose elements are called the Ruelle resonances). In the context of $C^\infty$ Anosov diffeomorphisms and weights, we can use the following result to define the Ruelle spectrum.

\begin{thm}[Theorem 2.3 in \cite{gouezel_liverani_1}]\label{theorem:resonances}
Let $M$ be a closed $C^\infty$ manifold and $F$ be a $C^\infty$ Anosov diffeomorphism on $M$. Let $w$ be a $C^\infty$ function on $M$. Then the family $z \mapsto (z - \mathcal{L}_{F,w})^{-1}$ of operators from $C^\infty(M)$ to $\mathcal{D}'(M)$, defined for $z \in \mathbb{C}$ large by
\begin{equation*}
(z - \mathcal{L}_{F,w})^{-1} : u \mapsto \sum_{n \geq 0} z^{-(n+1)} \mathcal{L}_{F,w}^n u,
\end{equation*}
admits a meromorphic extension to $\mathbb{C} \setminus \set{0}$ with residues of finite rank.
\end{thm}

The poles of $(z - \mathcal{L}_{F,w})^{-1}$ are called the (Ruelle) resonances of $\mathcal{L}_{F,w}$. The residues of $(z - \mathcal{L}_{F,w})^{-1}$ are related to spectral projectors of $\mathcal{L}_{F,w}$. The elements of their images are called the (generalized) resonant states of $\mathcal{L}_{F,w}$. The dimension of the image of the residue $\Pi_\lambda$ of $(z - \mathcal{L}_{F,w})^{-1}$ at a resonance $\lambda$ is called the multiplicity of $\lambda$ as a resonance.

When $w$ is positive, the left and right eigenvectors associated to the largest resonance of $\mathcal{L}_{F,w}$ may be used to construct the equilibrium measure $\mu_w$ for $F$ associated to the weight $\log w - \log |J_u|$ where $J_u$ denotes the unstable Jacobian of $F$. The smaller resonances may then be used to describe finer statistical properties of $\mu_w$. In particular, one can make explicit the asymptotics of correlations for this invariant measure for $F$. When $u$ and $v$ are $C^\infty$ functions on $M$, one can get an asymptotic expansion with an arbitrarily small geometric error for 
\begin{equation*}
\int_M u. v \circ F^n \mathrm{d}\mu_w
\end{equation*}
when $n$ goes to $+ \infty$. Further statistical properties of $F$ may be obtained by studying the spectral properties of $\mathcal{L}_{F,w}$: central limit theorem, almost sure invariance principle, \dots

If the case of positive $w$ (in particular the case $w = 1$ that corresponds to the SRB measure) is more relevant from a dynamical point of view, there is no particular reason to restrict to this case if we are merely interested in the spectral properties of $\mathcal{L}_{F,w}$. Hence, we will consider the general case in this paper.

In order to compute the resonances of $\mathcal{L}_{F,w}$, on may introduce the dynamical determinant defined for $z \in \mathbb{C}$ small by
\begin{equation}\label{eq:dynamical_determinant}
d_{F,w}(z) \coloneqq \exp\p{ - \sum_{n \geq 1} \frac{z^n}{n} \sum_{F^n x = x} \frac{\prod_{k = 0}^{n-1} w(F^k x)}{\va{\det\p{ I - D_x F^n}}}}.
\end{equation}
We will need the following important result that relates the holomorphic extension of $d_{F,w}$ with the Ruelle resonances of $\mathcal{L}_{F,w}$.

\begin{thm}[\cite{kitaev, kitaev_corrigendum,liverani_tsujii,baladi_tsujii_determinant}]\label{theorem:determinant}
Let $M$ be a closed $C^\infty$ manifold and $F$ be a $C^\infty$ Anosov diffeomorphism on $M$. Let $w$ be a $C^\infty$ function on $M$. The dynamical determinant $d_{F,w}$ has a holomorphic extension to $\mathbb{C}$ whose zeroes are the inverses of the Ruelle resonances of $\mathcal{L}_{F,w}$ (counted with multiplicities).
\end{thm}

\subsection{Main results}

We are interested in this paper in the distribution of Ruelle resonances. More precisely, if $F$ is a $C^\infty$ Anosov diffeomorphism on a compact manifold $M$ and $w$ a $C^\infty$ function on $M$, we will study the asymptotics when $r$ goes to $0$ of the counting function
\begin{equation*}
N_{F,w}(r) = \# \set{ \lambda \textup{ resonance of } \mathcal{L}_{F,w} \textup{ such that } \va{\lambda} \geq r},
\end{equation*}
where resonances are counted with multiplicities. The results from \cite{local_and_global} suggest that there is actually no general upper bound on $N_{F,w}(r)$ when $r$ goes to $0$. See in particular \cite[Proposition 2.10]{local_and_global} that proves the absence of general upper bound in the more general context of open hyperbolic diffeomorphisms.

However, when $M,F$ and $w$ are real-analytic, one expects a bound of the form
\begin{equation}\label{eq:upper_bound_resonances}
N_{F,w}(r) \underset{r \to 0}{=} \mathcal{O}\p{\va{\log r}^d},
\end{equation}
where $d$ is the dimension of $M$. This bound is suggested by related results on expanding maps \cite{ruelle_zeta_expanding, bandtlow_naud} and local model for the transfer operators associated to real-analytic expanding dynamical systems \cite{naud_local}. The bound \eqref{eq:upper_bound_resonances} can be deduced from the work of Rugh \cite{rugh_correlation, rugh_fredholm_axiomA} when $d=2$. In \cite[Theorem 7]{faure_roy}, Faure and Roy gives a new proof of this bound for $C^1$ small real-analytic perturbations of linear cat maps on the $2$-dimensional torus. Let us also mention related results for real-analytic Anosov flows in \cite{fried_selberg,fried_zeta,BJ20}. Our first result is the validity of the bound \eqref{eq:upper_bound_resonances} in full generality.

\begin{thm}\label{theorem:upper_bound_resonances}
Let $F$ be a real-analytic Anosov diffeomorphism on a real-analytic closed manifold $M$ of dimension $d$. Let $w$ be a real-analytic function on $M$. The number $N_{F,w}(r)$ of Ruelle resonances of $\mathcal{L}_{F,w}$ of modulus more than $r$ satisfies the asymptotic bound
\begin{equation*}
N_{F,w}(r) \underset{r \to 0}{=} \mathcal{O}\p{\va{\log r}^d}.
\end{equation*}
\end{thm}

Theorem \ref{theorem:upper_bound_resonances} is an immediate consequence of the following bound on the dynamical determinant \eqref{eq:dynamical_determinant} associated to $\mathcal{L}_{F,w}$.

\begin{thm}\label{thm:upper_bound_determinant}
Let $F$ be a real-analytic Anosov diffeomorphism on a real-analytic closed manifold $M$ of dimension $d$. Let $w$ be a real-analytic function on $M$. Then, there is a constant $C > 0$ such that for every $z \in \mathbb{C}$ we have
\begin{equation*}
\va{d_{F,w}(z)} \leq C \exp\p{C\p{\log(1 + \va{z})}^{d+1}}.
\end{equation*}
\end{thm}

A natural question with respect to Theorem \ref{theorem:upper_bound_resonances} is whether the bound is optimal or not. Our main result addresses this question.

\begin{thm}\label{theorem:optimal_dense}
Let $M$ be a closed real-analytic manifold of dimension $d$. Let $\mathcal{A}$ denote the set of $F \in \Anos^\omega(M)$ such that 
\begin{equation}\label{eq:optimality}
\limsup_{r \to 0} \frac{\log N_{F,1}(r)}{\log \va{\log r}} = d.
\end{equation}
If $W$ is a connected component of $\Anos^\omega(M)$ such that $\mathcal{A}\bigcap W \neq \emptyset$ then for every $F \in W$ there is a sequence of elements of $\mathcal{A}$ that converges to $F$ in the $C^\omega$ topology.
\end{thm}

\begin{remark}
The equality \eqref{eq:optimality} expresses that the exponent $d$ in \eqref{eq:upper_bound_resonances} is optimal.
\end{remark}

\begin{remark}
In \S \ref{subsection:definitions}, we detail the different topologies that we consider on the space $\Anos^\omega(M)$. In particular, we explain what it means to converge in the $C^\omega$ topology.

The three topologies that we define in \S \ref{subsection:definitions} ($C^1,C^\infty$ and $C^\omega$) induce the same connected components. This is why we did not specify in Theorem \ref{theorem:optimal_dense} which topology we use to define connected components of $\Anos^\omega(M)$.
\end{remark}

\begin{remark}
One may use \cite[Theorem 2.7]{gouezel_liverani_1} to prove that $\mathcal{A}$ is a $G_\delta$ for the $C^\infty$ topology. Notice also that, while we stated Theorem \ref{theorem:optimal_dense} in the most dynamically relevant case $w = 1$, we are going to prove a slightly more general statement, see Theorem \ref{theorem:more_general}.
\end{remark}

The proof of Theorem \ref{theorem:optimal_dense} is an adaptation of the method used by Bandtlow and Naud \cite{bandtlow_naud} in the context of real-analytic expanding maps of the circle. This method is based itself on a strategy in the context of scattering resonances by Christiansen and her coauthors \cite{christiansen_several,christiansen_euclidean,christiansen_potential,christiansen_schrodinger,christiansen_hyperbolic}. 

The main ingredient needed to apply this method in our context is the existence, when $F$ and $w$ are real-analytic, of a Hilbert space $\mathcal{H}$ on which the operator $\mathcal{L}_{F,w}$ defines a compact operator whose singular values decay fast enough (this fact also implies Theorems \ref{theorem:upper_bound_resonances} and \ref{thm:upper_bound_determinant}). Moreover, it is essential for the proof that the space $\mathcal{H}$ is constructed with enough flexibility to deal with the perturbations (even complex) of $F$ and $w$. The construction of the space $\mathcal{H}$ is carried out in \S \ref{section:construction}. 

The main idea in order to construct the space $\mathcal{H}$ is to use the construction from \cite{faure_roy} as a local model. However, due to the absence of real-analytic partition of unity, one cannot easily glue the locally defined spaces to get global spaces, as it is commonly done in the $C^\infty$ case. This issue will be solved by using a real-analytic Fourier--Bros--Iagolnitzer transform to design a process of localization that preserves real-analyticity. We will use the FBI transform described in \cite{BJ20}. However, we will only need to use the most basic properties of this transform, they are recalled in \S \ref{subsection:FBI}.

Functional spaces suited for the study of Koopman (or transfer) operators associated to real-analytic hyperbolic dynamical systems have already been proposed in \cite{faure_roy,first_example,other_example,optimal_examples}. The main novelty of our work is that we are able to design such a space $\mathcal{H}$ without restricting to the case of diffeomorphisms that are close to a linear model, or admit constant cone field. 

Theorem \ref{theorem:optimal_dense} would be less interesting in the absence of examples of real-analytic Anosov diffeomorphism satisfying \eqref{eq:optimality}. The existence of such examples are non-trivial. For instance, if $F$ is a \emph{cat map}, the simplest example of Anosov diffeomorphism, then $1$ is the only resonance of $\mathcal{L}_{F,1}$. The first example of diffeomorphism satisfying \eqref{eq:optimality} we are aware of appear in \cite{first_example}. Related examples are discussed in \cite{other_example}. In \cite{optimal_examples}, Bandtlow, Just and Slipantschuk produce many explicit examples of Anosov diffeomorphisms of the $2$-dimensional torus satisfying \eqref{eq:optimality}. With Theorem \ref{theorem:optimal_dense}, the existence of these examples implies:

\begin{thm}\label{theorem:optimality_torus}
Let $F$ be a real-analytic Anosov diffeomorphism on $\mathbb{T}^2$. Then, there is a sequence $(F_n)_{n \in \mathbb{N}}$ of Anosov diffeomorphisms on $\mathbb{T}^2$ that converges to $F$ in the $C^\omega$ topology and such that 
\begin{equation*}
\limsup_{r \to 0} \frac{\log N_{F_n,1}(r)}{\log \va{\log r}} = 2
\end{equation*}
for every $n \in \mathbb{N}$.
\end{thm}

More concretely, if $M = \mathbb{T}^2$ and $w = 1$, then the exponent $d$ in Theorem \ref{theorem:upper_bound_resonances} is optimal for a dense subset of real-analytic diffemorphisms.

Finally, let us mention that Theorems \ref{theorem:upper_bound_resonances} and \ref{thm:upper_bound_determinant} have an analogue using Gevrey regularity instead of real-analytic regularity, with a slightly simpler proof. This yields quantitative improvements (Theorem \ref{theorem:gevrey_determinant} and Corollary \ref{corollary:upper_bound_resonances_gevrey}) over the results from \cite[\S 2.2]{local_and_global} in the case of Anosov diffeomorphisms. The method of proof of Theorem \ref{theorem:optimal_dense} is restricted to real-analytic regularity, as we need to be able to work with complex perturbations of an Anosov diffeomorphism. However, we can consider complex perturbations of the weight $w$. Since it is easier to construct Gevrey than real-analytic weights, we are able to find another application of the method of proof of Theorem \ref{theorem:optimal_dense}, see Theorem \ref{theorem:optimal_strange}. Results in Gevrey regularity are gathered in Appendix \ref{section:Gevrey}.

\subsection{Structure of the paper}

In \S \ref{section:generalities}, we recall some fundamental facts about Anosov diffeomorphisms and several tools that will be needed for the proof of our main results.

In \S \ref{section:construction}, for $F$ and $w$ real-analytic, we construct a space on which $\mathcal{L}_{F,w}$ as good spectral properties. This is the core of the proof of our main results.

In \S \ref{section:consequences}, we deduce Theorems \ref{theorem:upper_bound_resonances}, \ref{thm:upper_bound_determinant}, \ref{theorem:optimal_dense} and \ref{theorem:optimality_torus}.

In Appendix \ref{appendix:kernel_estimates}, we prove several technical estimates that are needed for the analysis in \S \ref{section:construction}.

In Appendix \ref{section:Gevrey}, we explain how our analysis can be partially adapted to the Gevrey case. In particular, we improve certain results from \cite{local_and_global}.

\section{Generalities}\label{section:generalities}

In this section, we recall several definitions and results that we will need for the proof of our main results. In \S \ref{subsection:Anosov_diffeomorphisms}, we recall the definition of Anosov diffeomorphism. In \S \ref{subsection:definitions}, we introduce notation that are useful when working with real-analytic functions on a manifold. In \S \ref{subsection:koopman}, we discuss briefly Koopman operators associated to real-analytic dynamics. In \S \ref{subsection:FBI}, we recall some basic properties of a real-analytic Fourier--Bros--Iagolnitzer that we will use in the proof of Theorem \ref{theorem:anisotropic_space} below. In \S \ref{subsection:exponential_class}, we discuss a class of operators which is relevant to the study of Koopman operators associated to real-analytic hyperbolic dynamics.

\subsection{Anosov diffeomorphism}\label{subsection:Anosov_diffeomorphisms}

Let $M$ be a smooth closed manifold. Endow $M$ with any Riemannian metric. We recall \cite[Definition 6.4.2]{katok_hasselblatt} that a $C^1$ diffeomorphism $F: M \to M$ is said to be \emph{Anosov} if there is a splitting of the tangent bundle of $M$ as $TM = E_s \oplus E_u$ into the sum of two continuous bundles $E_s$ and $E_u$ invariant by the derivative of $F$, and such that there are constants $C > 0$ and $\lambda > 1$ with
\begin{itemize}
\item for every $x \in M, n \in \mathbb{N}$ and $v \in E_s(x)$, we have $|DF^{n}_x \cdot v| \leq C \lambda^{-n} |v|$;
\item for every $x \in M, n \in \mathbb{N}$ and $v \in E_u(x)$, we have $|DF^{-n}_x \cdot v| \leq C \lambda^{-n} |v|$.
\end{itemize}

As microlocal analysis is more naturally formulated on the cotangent bundle, we will rather work with the related decomposition $T^* M = E_s^* \oplus E_u^*$, where $E_u^*$ and $E_s^*$ are respectively the annihilators of $E_u$ and $E_s$. This convention ensures that for $\lambda > 1$ as above and some $C > 0$, we have
\begin{itemize}
\item for every $x \in M, n \in \mathbb{N}$ and $\xi \in E_s^*(x)$, we have $|{}^t (DF^{n}_x)^{-1} \cdot \xi| \leq C \lambda^{-n} |\xi|$;
\item for every $x \in M, n \in \mathbb{N}$ and $\xi \in E_u^*(x)$, we have $|{}^t (DF^{-n}_x)^{-1} \cdot \xi| \leq C \lambda^{-n} |\xi|$.
\end{itemize}

Let us point out that this definition does not depend on the choice of the Riemannian metric on $M$. In particular, one may choose an adapted metric, that is a metric for which $C = 1$ in the inequalities above.

If $F : M \to M$ is an Anosov diffeomorphism on $M$, then one may construct following \cite{faure_roy_sjostrand} an escape function $G$ on $T^* M$, that is a function that decreases under the action of $F$ on $T^* M$. We will define $G$ as follows. For $(x,\xi) \in T^* M$, we write $\xi = \xi_u + \xi_s$ for the decomposition of $\xi$ with respect to $T^*_x M = E_u^*(x) \oplus E_s^*(x)$, and then
\begin{equation}\label{eq:definition_escape_function}
G(x,\xi) = \va{\xi_s} - \va{\xi_u}.
\end{equation}
By choosing an adapted metric on $M$, we may ensure that for $\xi$ large enough, we have
\begin{equation*}
G(\mathcal{F}(x,\xi)) - G(x,\xi) \leq - C^{-1} \va{\xi},
\end{equation*}
for some constant $C > 0$ and all $(x,\xi) \in T^* M$. Here, $\mathcal{F}$ denotes the symplectic lift of $F$ defined by $\mathcal{F}(x,\xi) = (Fx , {}^t (DF_x)^{-1} \xi)$. For a general diffeomorphims $H$, if we do not want to introduce a specific notation for the symplectic lift, we will denote it by ${}^t D H^{-1}$. Let us point out here that we can give a particularly simple definition of the escape function $G$ because we do not care much about the smoothness of $G$.

\subsection{Real-analytic manifolds and diffeomorphisms}\label{subsection:definitions}

Let us introduce now notation that are useful when working with real-analytic functions on manifolds.

Let $M$ be a real-analytic compact manifold of dimension $d$. For convenience, we endow $M$ with a real-analytic Riemannian metric $g$ (which is possible according to \cite{morrey_embedding}). Let $\widetilde{M}$ denote a complexification for $M$, and endow $\widetilde{M}$ with a real, $C^\infty$ Riemannian metric $\tilde{g}$ (in particular, there is a smooth distance on $\widetilde{M}$). For $\epsilon > 0$ small, we let $(M)_\epsilon$ denotes the Grauert tube \cite{grauert_tube_I,grauert_tube_II} of size $\epsilon$ for $M$:
\begin{equation*}
(M)_\epsilon = \set{\exp_x(iv) : x \in M, v \in T_x M, g_x(v) < \epsilon^2},
\end{equation*}
where $\exp_x$ denotes the holomorphic extension of the exponential map. Here, we will use the Grauert tubes of $M$ as convenient complex neighbourhoods (they are Kähler and pseudoconvex). The Riemannian metric $g$ induces a decomposition of $T(T^* M)$ into a horizontal and a vertical bundle: for every $\alpha = (\alpha_x,\alpha_\xi) \in T^* M$, there is an identification
\begin{equation*}
T_\alpha(T^* M) \simeq T_{\alpha_x} M \oplus T_{\alpha_x}^* M,
\end{equation*}
such that the derivative of the canonical projection $T^* M \to M$ is given in this decomposition by $(u,v) \mapsto u$. We endow $T^* M$ with the Kohn--Nirenberg metric defined for $\alpha \in T^* M$ and $(u,v) \in T_\alpha(T^* M) \simeq T_{\alpha_x} M \oplus T_{\alpha_x}^* M$ by
\begin{equation*}
g_{KN,\alpha}(u,v) = g_{\alpha_x}(u) + \frac{g_{\alpha_x}(v)}{\brac{\alpha}^2},
\end{equation*}
where the Japanese bracket $\brac{\alpha}$ is defined by $\brac{\alpha} = \sqrt{1 + g_{\alpha_x}(\alpha_\xi)}$. Here, we identify a metric with the associated quadratic form. The distance associated to this metric will be denoted by $d_{KN}$. Notice that two points $\alpha,\beta \in T^* M$ are close for the Kohn--Nirenberg distance $d_{KN}$ if $\alpha_x$ and $\beta_x$ are close, $\alpha_\xi$ and $\beta_\xi$ have the same magnitude and, in local coordinates, which makes sense since $\alpha_x$ and $\beta_x$ are close, $|\alpha_\xi - \beta_\xi|$ is small with respect to this magnitude. We will sometimes also need the Japanese bracket $\brac{\va{\alpha}} = \sqrt{1 + \tilde{g}_{\alpha_x}(\alpha_\xi)}$ defined and positive for $\alpha= (\alpha_x,\alpha_\xi)$ in $T^* \widetilde{M}$ (we can identify $T^* \widetilde{M}$ with $T \widetilde{M}$ using a Hermitian metric). When $\epsilon > 0$ is small enough, we may define as above the Grauert tube $(T^* M)_\epsilon$ of $T^* M$ (using the metric $g_{KN}$), that identifies with a subset of $T^* \widetilde{M}$. \emph{Very roughly}, $(T^* M)_{\epsilon}$ may be thought in coordinates as a set of points $(x,\xi)$ such that the imaginary part of $x$ is less than $\epsilon$, and the imaginary part of $\xi$ is less than $\epsilon \brac{\re \xi}$. We refer to \cite[\S 1.1.1.2]{BJ20} for further discussion of this notion (see also \cite[\S 5.1]{upper_bound_resonances}).

Let us fix once for all $\epsilon_0$ such that $(M)_{\epsilon_0}, (M \times M)_{\epsilon_0}$ and $(T^* M)_{\epsilon_0}$ are well-defined and the holomorphic extension of the exponential map $(x,v) \mapsto \exp_x(v)$ for $g$ is well defined on neighbourhood of the zero section of $T(M)_{\epsilon_0}$. Notice that since we do not care about the actual choice of $g$, we could choose the value of $\epsilon_0$.

For $\epsilon \in (0,\epsilon_0)$. We let $\widetilde{\mathcal{O}}_\epsilon$ denotes the space of bounded holomorphic functions on $(M)_\epsilon$, endowed with the supremum norm. We write $\mathcal{O}_\epsilon$ for the closure in $\widetilde{\mathcal{O}}_\epsilon$ of the space of holomorphic functions on $(M)_{\epsilon_0}$. It follows from the Oka--Weil Theorem that the dependance of $\mathcal{O}_\epsilon$ on $\epsilon_0$ is irrelevant. We shall write $\mathcal{O}_\epsilon(M)$ if it is needed to specify the manifold $M$. We will often identify an element of $\mathcal{O}_\epsilon$ with its restriction to $M$, which is unambiguous due to the analytic continuation principle. Hence, if we say that a function $f$ on $M$ belongs to $\mathcal{O}_\epsilon$, strictly speaking it means that it has an extension to $(M)_\epsilon$ that belongs to $\mathcal{O}_\epsilon$, and $\n{f}_{\mathcal{O}_\epsilon}$ denotes the sup norm of this extension. We also let $\mathcal{V}_\epsilon$ be the space of holomorphic sections $X$ of $T (M)_\epsilon$ such that
\begin{equation*}
\va{X}_\infty \coloneqq \sup_{x \in (M)_\epsilon} \sqrt{\tilde{g}_x(X(x))} < + \infty.
\end{equation*}
We endow $\mathcal{V}_\epsilon$ with the norm $\va{\cdot}_\infty$ and let $\mathcal{V}_\epsilon^{\mathbb{R}}$ be the closed subspace of $\mathcal{V}_\epsilon$ consisting of vector fields that are tangent to $M$. We let then $\mathcal{V}^\omega$ be the space of real-analytic vector fields on $M$, that we identify with $\bigcup_{\epsilon \in (0,\epsilon_0)} \mathcal{V}_\epsilon^{\mathbb{R}}$. We can put on $\mathcal{V}^\omega$ the topology of the inductive limit (in the category of locally convex topological vector space)
\begin{equation*}
\mathcal{V}^\omega = \lim_{\to} \mathcal{V}_\epsilon^{\mathbb{R}}.
\end{equation*}
There are other possible topologies on $\mathcal{V}^\omega$: the $C^1$ and the $C^\infty$ topologies. They are the metrizable topologies associated with the uniform convergence of vector fields and respectively of their first derivatives or all their derivatives (which makes sense in coordinates for instance). Notice that these two topologies make $\mathcal{V}^\omega$ a locally convex topological vector space which is not complete.

Let $\Diff^\omega(M)$ denote the space of real-analytic diffeomorphisms from $M$ to itself and $\Anos^\omega(M)$ the subset of $\Diff^\omega(M)$ made of real-analytic Anosov diffeomorphisms on $M$.

Let $\pi : TM \to M$ be the canonical projection and $\exp : TM \to M$ the exponential map associated to $g$. There is a neighbourhood $\mathbb{U}$ in $TM$ of the zero section and a neighbourhood $\mathbb{V}$ of the diagonal in $M \times M$ such that the map $(\pi,\exp) : \mathbb{U} \to \mathbb{V}$ is a real-analytic diffeomorphism. If $X$ is a vector field on $M$ taking values in $\mathbb{U}$, we let $\Psi_X$ denote the map from $M$ to itself defined by
\begin{equation}\label{eq:near_identity}
\Psi_X : x \mapsto \exp_x(X(x)).
\end{equation}
Notice that if $X$ is $C^1$ small, then $\Psi_X$ is a diffeomorphism from $M$ to itself. For $F \in \Diff^\omega(M)$, write
\begin{equation*}
\mathcal{U}_F = \set{ G \in \Diff^\omega(M): (F(x),G(x)) \in \mathbb{V}}.
\end{equation*}
and
\begin{equation*}
\mathfrak{U} = \set{ X \in \mathcal{V}^\omega : \forall x \in M, X(x) \in \mathbb{U} \textup{ and } \Psi_X \in \Diff^\omega(M)}
\end{equation*}
It follows from the definition of $\mathbb{V}$ that the map
\begin{equation*}
H_F : X \mapsto \Psi_X \circ F
\end{equation*}
is a bijection from $\mathfrak{U}$ to $\mathcal{U}_F$. Notice also that $\mathfrak{U}$ is an open neighbourhood of the origin in $\mathcal{V}^\omega$ for the $C^1$ topology (and thus for the $C^\infty$ and $C^\omega$ topologies too). We define the $C^1, C^\infty$ and $C^\omega$ topologies on $\Diff^\omega(M)$ by saying that a a subset $W$ of $\Diff^\omega(M)$ is open if and only if for every $F \in \Diff^\omega(M)$ the set $H_F^{-1}\p{W \cap \mathcal{U}_F}$ is open in $\mathcal{V}^\omega$, respectively for the $C^1,C^\infty$ or $C^\omega$ topology. All these topologies are finer than the $C^1$ topology and consequently, $\Anos^\omega(M)$ is always an open subset of $\Diff^\omega(M)$ \cite{original_anosov}. 

The $C^\omega$ topology on $\Diff^\omega(M)$ is rather intricate. The interested reader may refer to \cite{kriegl_michor_real_analytic} or \cite[Chapter IX]{kriegl_michor_global_analysis} for further discussion of the structure of the space of analytic mappings from $M$ to $M$, including the (infinite-dimensional) real-analytic manifold structure. Fortunately, the only thing that matters to us concerning the $C^\omega$ topology on $\Diff^\omega(M)$ is the associated notion of convergence. A sequence $(F_n)_{n \in \mathbb{N}}$ of real-analytic diffeomorphisms on $M$ converges to $F \in \Diff^\omega(M)$ if and only if there are complex neighbourhoods $U$ and $V$ of $M$ such that $F$ and the $F_n$'s have holomorphic extensions from $U$ to $V$ and the holomorphic extension of $F_n$ converges to the holomorphic extension of $F$ uniformly on $U$. One can use \cite[7.5]{lokalkonvexe} to check that this is indeed the notion of convergence induced by the $C^\omega$ topology. Notice that we will not even need that, since we could use this characterization as a definition. Similarly, we will say that a sequence $(w_n)_{n \in \mathbb{N}}$ converges to $w$ in the $C^\omega$ topology, if the $w_n$'s have holomorphic extension that converges uniformly to a holomorphic extension of $g$ on a complex neighbourhood of $M$. Once again, one may check that this coincides with the notion of convergence associated to the topology on $C^\omega$ defined in \cite[\S 3]{kriegl_michor_real_analytic}.

Most of the time, we will not work directly with the space of real-analytic vector fields, as we will need some control over the size of the domain of the holomorphic extensions of vectors fields we are working with. To do so, let us introduce some additional notation. For $\epsilon \in (0,\epsilon_0)$, we will write $B_{\epsilon,\delta}$ for the ball of radius $\delta$ centered at the origin in $\mathcal{V}_\epsilon$ and $B_{\epsilon,\delta}^{\mathbb{R}}$ for $ B_{\epsilon,\delta} \cap \mathcal{V}_\epsilon^{\mathbb{R}}$. Notice that for any $\epsilon \in (0,\epsilon_0)$, if $\delta$ is small enough then, for every $X \in B_{\epsilon,\delta}^{\mathbb{R}}$, the map $\Psi_X$ given by \eqref{eq:near_identity} is well-defined and a real-analytic diffeomorphism from $M$ to itself.

Now if $\epsilon,\epsilon_2 \in (0, \epsilon_0)$ and $\epsilon_1 \in (0,\epsilon)$, notice that we can find $\delta > 0$ such that for every $X \in B_{\epsilon,\delta}$ the formula \eqref{eq:near_identity} defines a holomorphic embedding $\Psi_X$ of $(M)_{\epsilon_1}$ into $(M)_\epsilon$ such that $\Psi_X(M) \subseteq (M)_{\epsilon_2}$.

\subsection{Koopman operator}\label{subsection:koopman}

For $F \in \Diff^\omega(M)$ and $w \in C^\omega(M,\mathbb{R})$, we recall that the Koopman operator associated to $F$ and $w$ is defined by \eqref{eq:def_Koopman}.

Let $\epsilon \in (0,\epsilon_0)$ and $F \in \Diff^\omega(M)$. Since $F$ is real-analytic, for any $\epsilon'> 0$ small enough, the operator $\mathcal{L}_{F,w}$ is bounded from $\mathcal{O}_\epsilon$ to $\mathcal{O}_{\epsilon'}$ for every $w \in \mathcal{O}_{\epsilon'}$. This follows from the fact that the holomorphic extension of $F$ sends $(M)_{\epsilon'}$ into $(M)_\epsilon$.

Similarly, we can find a $\delta > 0$ such that for $\epsilon' > 0$ small enough and every $X \in B_{\epsilon,\delta}$ the map $\Psi_X \circ F$ sends $(M)_{\epsilon'}$ into $(M)_\epsilon$, and thus, for every $w \in \mathcal{O}_{\epsilon'}$, we may define the operator $\mathcal{L}_{\Psi_X \circ F,w}$ by \eqref{eq:def_Koopman} as a bounded operator from $\mathcal{O}_\epsilon$ to $\mathcal{O}_{\epsilon'}$, and its norm is less than $C \n{w}_{\mathcal{O}_{\epsilon'}}$, for some constant $C> 0$ that depends on $F$ and $\epsilon$, but not on $w$. When $X \in B_{\epsilon,\delta}^{\mathbb{R}}$, that is when $X$ is tangent to $M$, this is a standard Koopman operator associated to the map $\Psi_X \circ F : M \to M$, and its action on any function on $M$ may consequently be defined.

Let us define the formal adjoint of $\mathcal{L}_{\Psi_X \circ F,w}$ by
\begin{equation*}
\mathcal{L}_{\Psi_X \circ F,w}^* = \mathcal{L}_{F^{-1} \circ \Psi_X^{-1}, J (F^{-1} \circ \Psi_X^{-1}) w \circ F^{-1} \circ \Psi_{X}^{-1}},
\end{equation*}
where the Jacobian $J(F^{-1} \circ \Psi_X^{-1})$ is defined by the relation $(F^{-1} \circ \Psi_X^{-1})^* \mathrm{d}x = \pm J(F^{-1} \circ \Psi_X^{-1}) \mathrm{d}x$, with $\mathrm{d}x$ the holomorphic extension of the Riemannian volume of $M$, if $M$ is orientable (and $\pm$ is $+$ if $F$ preserves orientation and $-$ otherwise). If $M$ is not orientable, we define $J(F^{-1} \circ \Psi_X^{-1})$ by going to the bundle of orientation of $M$.

Notice that this is the formal adjoint for the $L^2$ bilinear scalar product (and not the sesquilinear scalar product). This definition makes sense under the same assumption as above (in particular, $X$ does not need to be tangent to $M$ provided it is small enough), and it has the same mapping properties than $\mathcal{L}_{\Psi_X \circ F,w}$. Moreover, for $\epsilon$ small enough and every $u,v \in \mathcal{O}_\epsilon$, we find by a contour shift that
\begin{equation*}
\int_M u (\mathcal{L}_{\Psi_X \circ F,w} v) \mathrm{d}x = \int_M (\mathcal{L}_{\Psi_X \circ F,w}^* u) v \mathrm{d}x.
\end{equation*}
As above, in the non-orientable case this formula is proved by going to the bundle of orientation of $M$ (and the integration is with respect to the Riemannian density on $M$).

\subsection{FBI transform}\label{subsection:FBI}

Let us recall that we want to construct a functional space adapted to a real-analytic hyperbolic dynamics (see Theorem \ref{theorem:anisotropic_space} below). As mentioned in the introduction, we will use a real-analytic Fourier--Bros--Iagolnitzer transform to get a localization procedure that preserves real-analyticity. We will use the real-analytic transform defined in \cite{BJ20}. However, we will only work with the basic properties of this transform (and thus avoid to use the most complicated results from \cite{BJ20}).

We recall that $M$ is a closed real-analytic manifold of dimension $d$. Let us consider a real-analytic FBI transform $T$ on $M$, as in \cite[Definition 2.1]{BJ20}. This is an operator, for instance from $C^\infty(M)$ to $C^\infty(T^* M)$, given by the formula
\begin{equation}\label{eq:fbi_from_kernel}
T u (\alpha) = \int_M K_T(\alpha,y) u(y) \mathrm{d}y,
\end{equation}
for every $u \in C^\infty(M)$ and $\alpha \in T^* M$. Here, $\mathrm{d}y$ denotes the (real-analytic) Riemannian density on $M$ and $K_T$ is a real-analytic kernel that satisfies the following properties:
\begin{itemize}
\item there is $\epsilon \in (0,\epsilon_0)$ such that $K_T$ has a holomorphic extension to $(T^* M)_{\epsilon} \times (M)_\epsilon$;
\item for every $c > 0$, there are $\epsilon \in (0,\epsilon_0)$ and $C > 0$ such that for every $\alpha = (\alpha_x,\alpha_\xi) \in (T^* M)_\epsilon$ and $y \in (M)_\epsilon$ if the distance between $\alpha_x$ and $y$ is larger than $c$ then 
\begin{equation}\label{eq:pseudo_local}
\va{K_T(\alpha,y)} \leq C \exp\p{- \frac{\brac{\va{\alpha}}}{C}};
\end{equation}
\item there are $c > 0, \epsilon \in (0,\epsilon_0)$ and $C > 0$ such that for every $\alpha = (\alpha_x,\alpha_\xi) \in (T^* M)_\epsilon$ and $y \in (M)_\epsilon$ if the distance between $\alpha_x$ and $y$ is less than $c$ then 
\begin{equation}\label{eq:local_descrption}
\va{K_T(\alpha,y) - e^{i \Phi_T(\alpha,y)} a(\alpha,y)} \leq C \exp\p{- \frac{\brac{\va{\alpha}}}{C}}.
\end{equation}
\end{itemize}
In \eqref{eq:local_descrption}, the symbol $a(\alpha,y)$ is a holomorphic function on the set $$\set{ (\alpha,y) \in (T^* M)_\epsilon \times (M)_\epsilon: d(\alpha_x,y) < 2c}$$ that satisfies
\begin{equation*}
C^{-1} \brac{\va{\alpha}}^{\frac{d}{4}} \leq \va{a(\alpha,y)} \leq C \brac{\va{\alpha}}^{\frac{d}{4}}.
\end{equation*}
The phase $\Phi_T$ is holomorphic on the same set and satisfies there
\begin{equation}\label{eq:control_phase}
\va{\Phi_T(\alpha,y)} \leq C \brac{\va{\alpha}}.
\end{equation}
Moreover, for every $\alpha = (\alpha_x,\alpha_\xi) \in T^* M$, we have the relations $\Phi_T(\alpha,\alpha_x) = 0, \mathrm{d}_y \Phi_T(\alpha,\alpha_x)= - \alpha_\xi$ and if $y \in M$ is at distance less than $2c$ from $\alpha_x$ then we have
\begin{equation}\label{eq:coercivity_phase}
\im \Phi_T(\alpha,y) \geq C^{-1} d(\alpha_x,y)^2.
\end{equation}

According to \cite[Theorem 6]{BJ20}, such an operator exists. Moreover, we may (and will) assume that $T$ is an isometry from $L^2(M)$ to $L^2(T^* M)$ (where $M$ and $T^* M$ are associated respectively with the Riemannian density and the volume form $\mathrm{d}\alpha$ associated to the canonical symplectic form). For every $u,v \in C^\infty(M)$, we have
\begin{equation*}
\int_M u \bar{v} \mathrm{d}y = \int_{T^* M} Tu \overline{T v} \mathrm{d}\alpha.
\end{equation*}
The integral on the right hand side converges, as one can check using the non-stationary phase method (i.e. doing repeated integration by parts) that, if $u \in C^\infty(M)$, then $Tu$ decreases faster than the inverse of any polynomial. One may define the formal adjoint $S = T^*$ of $T$ as the operator with real-analytic kernel $K_S(x,\alpha) = \overline{K_T(\bar{\alpha},\bar{x})}$. For $u$ a function on $T^* M$ that decreases faster than the inverse of any polynomial and $x \in M$, we set
\begin{equation*}
S u (x) = \int_{T^* M} K_S(x,\alpha) u(\alpha) \mathrm{d}\alpha.
\end{equation*}
Since $T$ is an isometry, we see that if $u \in C^\infty(M)$ then $S(Tu) = u$. One may object that in \cite{BJ20}, the FBI transform $T$ depends on a small semi-classical parameter $h > 0$, and is only defined and isometric when $h$ is small enough. However, one may reduce to the description given above by taking $h > 0$ small enough and fixed, so that there is a well defined and isometric FBI transform $\widetilde{T}$ and then define the kernel of $T$ as $K_T(\alpha,y) = h^{\frac{n}{2}} K_{\widetilde{T}}(h\alpha,y)$, where for $\alpha = (\alpha_x,\alpha_\xi) \in T^* M$ we set $h \alpha = (\alpha_x,h \alpha_\xi)$. This rescaling makes the parameter $h$ artificially disappear from the notation.

Notice that since the kernel $K_T$ of the FBI transform is real-analytic, formula \eqref{eq:fbi_from_kernel} makes sense when $u$ is a distribution (or even a hyperfunction) and defines a smooth function on $T^* M$. Similarly, one can define $Su$ as a distribution when $u$ is a measuable function on $T^* M$ growing at most polynomially. We refer to \cite[\S 2.1]{BJ20} for further discussions of the mapping properties of $S$ and $T$. It will not play a role here, as we will define our spaces of anisotropic hyperfunctions as completions of spaces of analytic functions.

We will need the following fact relating the real-analytic FBI transform $T$ to the real-analytic regularity.

\begin{proposition}[Lemma 2.4 in \cite{BJ20}]\label{proposition:decay_from_analytic}
Let $\epsilon > 0$. Then there are $C,\rho > 0$ such that for every $u \in \mathcal{O}_\epsilon$ and for every $\alpha \in T^* M$ we have
\begin{equation*}
\va{Tu(\alpha)} \leq C \n{u}_{\mathcal{O}_\epsilon} \exp\p{- \rho \brac{\alpha}},
\end{equation*}
where we recall that $\n{u}_{\mathcal{O}_\epsilon}$ denotes the supremum of $u$ on $(M)_\epsilon$.
\end{proposition}

To go in the other direction, we will use the following consequence of the inversion formula $STu = u$ and of \cite[Lemma 2.6]{BJ20} (see also \cite[Proposition 2.6]{BJ20})

\begin{proposition}\label{proposition:analytic_from_decay}
Let $\rho > 0$. There are $C,\epsilon > 0$ such that if $u \in C^\infty(M)$ is such that
\begin{equation*}
\sup_{\alpha \in T^* M} e^{\rho \brac{\alpha}} \va{Tu(\alpha)} < + \infty
\end{equation*}
then $u \in \mathcal{O}_\epsilon$ and
\begin{equation*}
\n{u}_{\mathcal{O}_\epsilon} \leq C \sup_{\alpha \in T^* M} e^{\rho \brac{\alpha}} \va{Tu(\alpha)}.
\end{equation*}
\end{proposition}

Let us recall here that when we write $u \in \mathcal{O}_\epsilon$, we mean that $u$ has a holomorphic extension that belongs to $\mathcal{O}_\epsilon$.

In addition to the real-analytic FBI transform $T$ on $M$, let us choose a real-analytic FBI transform $\mathcal{T}$ on the torus $\mathbb{T}^d$, and write $\mathcal{S} = \mathcal{T}^*$ for its adjoint. This is the same as above in the particular case $M = \mathbb{T}^d$, and we endow $\mathbb{T}^d$ with its standard flat metric.

For every $k \in \mathbb{Z}^d$, let us introduce the function
\begin{equation*}
e_k : x \mapsto e^{2 i \pi k \cdot x}
\end{equation*}
on the torus $\mathbb{T}^d$. Here, $k\cdot x$ denotes a scalar product. In the following proposition, we use the standard parallelization of $\mathbb{T}^d = \mathbb{R}^d / \mathbb{Z}^d$ to identify its cotangent bundle $T^* \mathbb{T}^d$ with $\mathbb{T}^d \times \mathbb{R}^d$.

\begin{proposition}\label{proposition:localisation_torus}
Let $c > 0$. There is a constant $C > 0$ such that for every $k \in \mathbb{Z}^d$ and $(x,\xi) \in T^*  \mathbb{T}^d$ such that $\va{2 \pi k - \xi} \geq c \brac{\xi}$ we have
\begin{equation}\label{eq:better_torus}
\va{\mathcal{T}e_k(x,\xi)} \leq C \exp\p{- \frac{\max(\va{\xi},\va{k})}{C}}.
\end{equation}
For every $(x,\xi) \in T^* M$, we have
\begin{equation}\label{eq:straightforward_torus}
\va{\mathcal{T}e_k(x,\xi)} \leq C \brac{\xi}^{\frac{d}{4}}.
\end{equation}
\end{proposition}

It will be crucial in \S \ref{section:construction} below to understand how the operators $T$ and $S$ lift the Koopman operators defined in \S \ref{subsection:koopman} to operators acting on functions on the cotangent bundle. In the following proposition, we identify the operator $T \mathcal{L}_{\Psi_X \circ F,w} S$ with its kernel, which is a real-analytic function on $T^* M \times T^* M$. See Remark \ref{remark:meaning_kernel} for the definition of this kernel.

\begin{proposition}\label{proposition:localisation_graph}
Let $F \in \Diff^\omega(M)$. Recall that the symplectic lift of $F$ is $\mathcal{F} : (\alpha_x,\alpha_\xi) \to (F \alpha_x, {}^t (D_{\alpha_x} F^{-1}) \alpha_\xi)$, and let 
\begin{equation*}
\mathcal{G} = \set{ (\alpha,\mathcal{F}(\alpha)): \alpha \in T^* M}
\end{equation*}
be the graph of $\mathcal{F}$. Let $c,\epsilon > 0$. Then there are constants $C,\delta > 0$ such that for every $X \in B_{\epsilon,\delta}, w \in \mathcal{O}_\epsilon$ and $\alpha,\beta \in T^* M$, then
\begin{equation}\label{eq:small_exponential_bound}
\va{T \mathcal{L}_{\Psi_X \circ F,w} S(\alpha,\beta)} \leq C \n{w}_{\mathcal{O}_\epsilon} \exp\p{c (\brac{\alpha} + \brac{\beta})}.
\end{equation}
If in addition the distance between $(\alpha,\beta)$ and $\mathcal{G}$ (for the Kohn--Nirenberg distance) is greater than $c$ then
\begin{equation}\label{eq:exponential_decay_Koopman}
\va{T \mathcal{L}_{\Psi_X \circ F,w} S(\alpha,\beta)} \leq C \n{w}_{\mathcal{O}_\epsilon} \exp\p{ - \frac{\brac{\alpha} + \brac{\beta}}{C}}.
\end{equation}
\end{proposition}

The proofs of Propositions \ref{proposition:localisation_torus} and \ref{proposition:localisation_graph} are given in Appendix \ref{appendix:kernel_estimates}.

\begin{remark}
Notice that if $X \in B_{\epsilon,\delta}^{\mathbb{R}}$ in Proposition \ref{proposition:localisation_graph} then the exponential bound \eqref{eq:small_exponential_bound} may be replaced by a polynomial bound.
\end{remark}

\begin{remark}
If $X = 0$ in Proposition \ref{proposition:localisation_graph} then the kernel $T \mathcal{L}_{\Psi_X \circ F,w} S(\alpha,\beta)$ may be approximated when $\alpha$ is close to $\mathcal{F}(\beta)$ by an application of the holomorphic stationary phase method. This computation is explained in the case $F = \Id_M$ in \cite[Lemma 2.10]{BJ20}. However, for our approach here the bound \eqref{eq:small_exponential_bound} is enough. The main point of Proposition \ref{proposition:localisation_graph} is to prove that, on the FBI transform side, the action on $\mathcal{L}_{\Psi_X \circ F, w}$ is negligible away from $\mathcal{G}$ (see the proof of Lemma \ref{lemma:discard} below).
\end{remark}

\subsection{Operators of exponential class}\label{subsection:exponential_class}

Let us end this toolbox section with the definition of a particular class of operators that appear naturally when studying Koopman operators associated to real-analytic or Gevrey dynamics. We refer for instance to \cite[\S 2.3.1]{pietsch_book} for the definition of approximation numbers (however, we use the convention that natural numbers start at $0$ rather than $1$, hence the indices in sequences of approximation numbers are shifted with respect to this reference).

\begin{definition}
Let $\alpha  >0$. Let $L : \mathcal{B}_1 \to \mathcal{B}_2$ be a compact operator between two Banach spaces. We say that $L$ is of exponential class $\alpha$ if there is a constant $C > 0$ such that the sequence of approximation numbers $(a_n)_{n \in \mathbb{N}}$ of $L$ satisfies $\va{a_n} \leq C \exp\p{ - n^\alpha /C}$ for every $n \geq 0$.
\end{definition}

This kind of operators appeared in the context of hyperbolic dynamics already in \cite{ruelle_zeta_expanding,fried_selberg,fried_zeta, rugh_correlation}. Similar classes of operators are used in \cite{bandtlow_naud} and studied in \cite{bandtlow_exponential} (the main difference is that we include here Banach spaces instead of only considering Hilbert spaces). Notice that it is clear that the product of a bounded operator and an operator of exponential class $\alpha$ is of exponential class $\alpha$.

Let us start by explaining how such operators usually appear in concrete setting. In the following lemma, the summation over $\mathbb{Z}^d$ naturally appears in two (related) contexts: it can correspond to a decomposition in Fourier modes or in power series.

\begin{lemma}\label{lemma:concrete}
Let $\mathcal{B}_1$ and $\mathcal{B}_2$ be two Banach spaces. Let $\nu > 0$. Let $(L_k)_{k \in \mathbb{Z}^d}$ be a family of operators of rank at most $1$ from $\mathcal{B}_1$ to $\mathcal{B}_2$ such that there are constants $C,\beta > 0$ such that $\n{L_k} \leq C e^{- \beta |k|^\nu}$ for every $k \in \mathbb{Z}^d$. Then the series $\sum_{k \in \mathbb{Z}^d} L_k$ converges to an operator of exponential class $\nu/d$.
\end{lemma}

\begin{proof}
The series is easily seen to be summable. Let us write $L$ for the sum, and $(a_n)_{n \in \mathbb{N}}$ for the sequence of approximation numbers of $L$. For $m \in \mathbb{N}$, let us approximate $L$ by the operator
\begin{equation*}
\sum_{\substack{k \in \mathbb{Z}^d \\ |k| \leq m}} L_k
\end{equation*}
whose rank is at most $C_0 m^d+1$ for some constant $C_0 > 0$ that does not depend on $m$. Then we bound
\begin{equation*}
\begin{split}
\n{L - \sum_{\substack{k \in \mathbb{Z}^d \\ |k| \leq m}} L_k} & \leq \sum_{\substack{k \in \mathbb{Z}^d \\ |k| \geq m}} \n{L_k} \leq C \sum_{p \geq m} p^{d-1} e^{- \beta p^\nu} \\
 & \leq C e^{- \frac{\beta m^{\nu}}{2}}
\end{split}
\end{equation*}
where the constant $C > 0$ may change from one line to another. We find consequently that $a_{C_0 m^d +1} \leq C e^{- \frac{\beta m^\nu}{2}}$, and the result then follows since the sequence $(a_n)_{n \geq 1}$ is decreasing.
\end{proof}

The following lemma is based on standard estimates on Fredholm determinant. We will use it to bound the number of Ruelle resonances for real-analytic Anosov diffeomorphism.

\begin{lemma}\label{lemma:order_exponential}
Let $\alpha > 0$. Let $\mathcal{H}$ be a Hilbert space and $L : \mathcal{H} \to \mathcal{H}$ be an operator of exponential class $\alpha$. Then, $L$ is trace class and there is a constant $C > 0$ such that for every $z \in \mathbb{C}$ we have
\begin{equation*}
\va{\det\p{ I - z L}} \leq C \exp\p{C \p{\log(1 + \va{z})}^{1+ \frac{1}{\alpha}}},
\end{equation*} 
and for every $r \in (0,\frac{1}{2})$ we have
\begin{equation}\label{eq:bound_eigenvalues}
\# \set{ \lambda \in \sigma(L) : \va{\lambda} \geq r} \leq C |\log r|^{\frac{1}{\alpha}},
\end{equation}
where the eigenvalues are counted with multiplicities.
\end{lemma}

\begin{proof}
Let $(a_n)_{n \in \mathbb{N}}$ denote the sequence of approximation numbers for $L$. For every $n \in \mathbb{N}$, let $L_n$ be an operator from $\mathcal{H}$ to itself of rank at most $2^n$ such that $\n{L - L_n} \leq 2 a_{2^n}$. Notice that we have $\n{L_n - L_{n+1}} \leq 2 (a_{2^n} + a_{2^{n+1}})$. Since $L$ is of exponential class $\alpha$, there is a constant $C > 0$ such that $\n{L_n - L_{n+1}} \leq C \exp(- 2^{\alpha n}/C)$ for every $n \in \mathbb{N}$. It follows that the series $\sum_{n \geq 0} L_{n+1} - L_{n}$ converges to $L - L_0$. 

For $n \in \mathbb{N}$, the rank of $L_{n+1} - L_n$ is at most $2^{n+2}$. From Auerbach's lemma and Hahn--Banach theorem, we find that for every $n \in \mathbb{N}$, there are linear forms $l_{1,n},\dots,l_{2^{n+2},n}$ of norm $1$, elements $e_{1,n},\dots,e_{2^{n+2},n}$ of $\mathcal{H}$ of norms $1$, and complex numbers $\lambda_{1,n},\dots,\lambda_{2^{n+2},n}$ of modulus less than $\n{L_{n+1} - L_n}$ such that
\begin{equation*}
L_{n+1} - L_n = \sum_{k = 1}^{2^{n+2}} \lambda_{k,n} e_{k,n} \otimes l_{k,n}.
\end{equation*}
Here, we write ``$e \otimes l$'' for the rank $1$ operator $u \mapsto l(u) e$. Summing over $n$ and relabelling, we find that there are a sequence of linear forms $(l_n)_{n \geq 1}$ of norm $1$, a sequence $(e_n)_{n \geq N}$ of elements of $\mathcal{H}$ of norms $1$, and a sequence $(\lambda_n)_{n \in \mathbb{N}}$ of complex numbers such that
\begin{equation*}
L = \sum_{n = 0}^{+ \infty} \lambda_n e_n \otimes l_n,
\end{equation*}
and $|\lambda_n| \leq C \exp(- n^\alpha/C)$ for some $C > 0$ and every $n \in \mathbb{N}$.

With this representation of $L$, the bound on the Fredholm determinant of $L$ follows for instance from \cite[Lemma 2.13]{local_and_global}. The bound of the number of resonances can then be deduced by applying Jensen's formula to the entire function $f : v \mapsto \det(I - e^v L)$. Indeed, we find that the number of zeros of $f$ of modulus less than $r$ is a $\mathcal{O}(r^{1 + \frac{1}{\alpha}})$. Using that each zero $z_0$ of $z \mapsto \det(I - zL)$ produces a line of zeros of $f$ of the form $v_0 + 2i \pi \mathbb{Z}$ with $\re v_0 = \log |z_0|$ and $\im v_0 \in [-\pi,\pi]$, we get the estimate \eqref{eq:bound_eigenvalues}.
\end{proof}

We end this section with a lemma that relates the notion of operators of exponential class $1/d$ with the holomorphic regularity. While the result is not surprising, we were not able to find a proof in the literature adapted to our particular case.

\begin{lemma}\label{lemma:exponential_inclusion}
Let $\epsilon,\epsilon' \in (0,\epsilon_0)$ be such that $\epsilon' < \epsilon$. Then the inclusion of $\mathcal{O}_\epsilon$ into $\mathcal{O}_{\epsilon'}$ is of exponential class $1/d$, where we recall that $d$ is the dimension of the closed real-analytic manifold $M$.
\end{lemma}

\begin{proof}
Since $(M)_{\epsilon'}$ is relatively compact in $(M)_\epsilon$, we can find two finite families $(A_i)_{i \in I}$ and $(B_i)_{i \in I}$ of open subsets of $(M)_\epsilon$ such that the closure of $(M)_{\epsilon'}$ is contained in $\bigcup_{i \in I} A_i$, the closure of $\bigcup_{i \in I} B_i$ is contained in $(M)_\epsilon$, and for every $i \in I$ there are holomorphic coordinates in which $A_i$ and $B_i$ are polydiscs with the same center, with the polyradius of $A_i$ strictly smaller than the polyradius of $B_i$. For $i \in I$, we write $E_i$ and $F_i$ for the spaces of bounded holomorphic functions respectively on $B_i$ and on $A_i$. 

For each $i \in I$, the map $E_i \to F_i$ obtained by restriction is of exponential class $1/d$, as can be seen by developing the elements of $E_i$ in power series around the center of $A_i$ and using Lemma \ref{lemma:concrete}. It follows that the map
\begin{equation*}
\begin{array}{ccccc}
\mathfrak{i} & : & \prod_{i \in I} E_i& \to & \prod_{i \in I} F_i \\
 & & (f_i)_{i \in I} & \mapsto & (f_{i|A_i})_{i \in I}
\end{array}
\end{equation*}
is also of exponential class $1/d$. 

Define then
\begin{equation*}
\mathcal{F} = \set{ (f_i)_{i \in I} \in \prod_{i \in I} F_i: \forall i,j \in I, f_{i|U_j} = f_{j|U_i}}.
\end{equation*}
This is a closed subspace of $\prod_{i \in I} F_i$, and there is a natural bounded linear map $P : \mathcal{F} \to \mathcal{O}_{\epsilon'}$ that associates to $(f_i)_{i \in I}$ the function on $(M)_{\epsilon'}$ that coincides with $f_i$ on $U_i \cap(M)_{\epsilon'}$ for each $i \in I$. As this function can be extended to a larger Grauert tube, it belongs indeed to $\mathcal{O}_{\epsilon'}$.

Let now $Q$ denotes the bounded linear map from $\mathcal{O}_{\epsilon}$ to $\prod_{i \in I} E_i$ given by $Qf = (f_{|B_i})_{i \in I}$. Using $Q$, we can define the linear operator $R = \mathfrak{i} \circ Q : \mathcal{O}_\epsilon \to \prod_{i \in I} F_i$, which is of exponential class $1/d$ as $\mathfrak{i}$ is. Notice also that $R$ actually maps $\mathcal{O}_\epsilon$ into $\mathcal{F}$, and let $\widetilde{R}$ be $R$ with codomain restricted to $\mathcal{F}$. Since the inclusion of $\mathcal{O}_\epsilon$ into $\mathcal{O}_{\epsilon'}$ is just the map $P \circ \widetilde{R}$, we only need to prove that $\widetilde{R}$ is of exponential class $1/d$.

Let us write $(a_n)_{n \in \mathbb{N}}$ for the approximation numbers of $R$ and $(\tilde{a}_n)_{n \in \mathbb{N}}$ for the approximation numbers of $\widetilde{R}$. Let $n \geq 1$. There is an operator $R_n : \mathcal{O}_\epsilon \to \prod_{i \in I} F_i$ of rank at most $n$ such that the operator norm of $R - R_n$ is less than $2 a_n$. Using Auerbach's lemma and the Hahn--Banach theorem, we find linear forms $l_1,\dots,l_n$ of norm $1$ on $\mathcal{O}_\epsilon$, functions $f_1,\dots,f_n$ of norm $1$ in $\mathcal{O}_\epsilon$ and elements $e_1,\dots,e_n$ of $\prod_{i \in I} F_i$ such that $R_n = \sum_{k = 1}^n e_k \otimes l_k$ and $l_k(f_\ell) = \delta_{k,\ell}$ for $k,\ell = 1,\dots,n$. Define then $\tilde{e}_k = \widetilde{R} f_k$ for $k =1,\dots,n$, and notice that these are elements of $\mathcal{F}$. Then, introduce the operator $\widetilde{R}_n = \sum_{k = 1}^n \tilde{e}_k \otimes l_k$. We have in operator norm
\begin{equation*}
\n{ R_n - \widetilde{R}_n} \leq \sum_{k = 1}^n \n{e_k - \tilde{e_k}} \n{l_k} = \sum_{k =1}^n \n{(R_n - R) f_k} \leq 2 n a_n.
\end{equation*}
It follows that $\n{\widetilde{R}_n - \widetilde{R}} \leq 2(n+1) a_n$ and thus that $\tilde{a}_n \leq  2(n+1) a_n$. This inequality implies that $\widetilde{R}$ is of exponential class $1/d$ since $R$ is.
\end{proof}

\section{Construction of a space of anisotropic hyperfunctions}\label{section:construction}

The goal of this section is to prove the following technical result from which our main results will follow, as explained in \S \ref{section:consequences} below. In this result, we refer to the Koopman operators discussed in \S \ref{subsection:koopman} and the map $\Psi_X$ for $X$ a holomorphic vector field in a neighbourhood of $M$ defined in \S \ref{subsection:definitions}.

\begin{thm}\label{theorem:anisotropic_space}
Let $M$ be a real-analytic closed manifold. Let $F \in \Anos^\omega(M)$. Let $\epsilon \in (0,\epsilon_0)$. There are $\delta > 0$ and a separable Hilbert space $\mathcal{H}$ with the following properties:
\begin{enumerate}[label=(\roman*)]
\item there are continuous injections with dense images $\iota$ and $j$ from $\mathcal{O}_\epsilon$ respectively to $\mathcal{H}$ and to $\mathcal{H}^*$; \label{item:injections}
\item for every $X \in B_{\epsilon,\delta}$ and $w \in \mathcal{O}_\epsilon$, there is a compact operator $\widetilde{\mathcal{L}}_{\Psi_X \circ F,w}$ from $\mathcal{H}$ to itself of exponential class $1/d$; \label{item:extension_operator}
\item if $z \mapsto X(z)$ is a holomorphic family of elements of $B_{\epsilon,\delta}$ and $z \mapsto w_z$ is a holomorphic family of elements of $\mathcal{O}_\epsilon$ then $z \mapsto \widetilde{\mathcal{L}}_{\Psi_{X(z)} \circ F,w_z}$ is a holomorphic family of trace class operator on $\mathcal{H}$; \label{item:holomorphic_family}
\item for every $X \in B_{\epsilon,\delta}^{\mathbb{R}}, w \in \mathcal{O}_\epsilon, n \in \mathbb{N}$ and $u,v \in \mathcal{O}_\epsilon$, we have
\begin{equation*}
j(u)\p{\widetilde{\mathcal{L}}^n_{\Psi_X \circ F,w} \iota(v)} = \int_M u \p{\mathcal{L}_{\Psi_X \circ F,w}^n v} \mathrm{d}x.
\end{equation*} \label{item:integrals}
\end{enumerate}
\end{thm}

We are going to construct $\mathcal{H}$ as the completion of $\mathcal{O}_\epsilon$ for a certain norm. However, one could also interpret $\mathcal{H}$ as a space of bounded linear functionals on $\mathcal{O}_\epsilon$. This is why we might call $\mathcal{H}$ a space of hyperfunctions. It does not agree with the most common terminology though, as the elements of $\mathcal{H}$ are not linear functionals on the space of real-analytic functions on $M$, but only on real-analytic functions with large enough radius of convergence. The elements of $\mathcal{H}$ are said to be anisotropic since they satisfy regularity hypotheses that are adapted to the dynamics of $F$. Morally, the elements of $\mathcal{H}$ are real-analytic in the stable direction of $F$ and dual of real-analytic in the unstable direction. Let us recall from the introduction the idea behind the construction of the space $\mathcal{H}$: we want to use the space from \cite{faure_roy} as a local model. However, due to the absence of real-analytic partition of unity, we need to use another localization procedure. We will rely on the FBI transform, see \S \ref{subsection:FBI}.

Let us start the proof of Theorem \ref{theorem:anisotropic_space}, which is the technical core of the paper. For the rest of this section, we fix a closed real-analytic manifold $M$ and $F \in \Anos^\omega(M)$. We let $G$ be an escape function for $F$, as defined in \S \ref{subsection:Anosov_diffeomorphisms}.

\subsection{Definition of the space}

Let $\varpi > 0$ be a very small constant. Assume that $\alpha,\beta,\gamma$ are three points in $T^* M$ large enough and such that $d_{KN}(\alpha,\beta) \leq \varpi$ and $d_{KN}(\mathcal{F} \beta,\gamma) \leq \varpi$. Then
\begin{equation}\label{eq:approximate_decay}
\begin{split}
G(\gamma) - G(\alpha) & = G(\mathcal{F}(\beta)) - G(\beta) + G(\gamma) - G(\mathcal{F}(\beta)) + G(\alpha) - G(\beta) \\
     & \leq - C^{-1} \va{\beta} + C \varpi^\varrho \va{\beta} \\
     & \leq - C^{-1} \va{\beta},
\end{split}
\end{equation}
where the constant $C$ may change on the last line, and we assume that $\varpi$ is small enough. Notice that we used that the unstable and stable directions are Hölder-continuous \cite[Theorem 6.3]{hirsch_pugh} (hence the exponent $\varrho$). Here, $\mathcal{F}$ denotes the symplectic lift of $\mathcal{F}$, as in \S \ref{subsection:Anosov_diffeomorphisms}.

Let $(U_\omega)_{\omega \in \Omega}$ be an open cover of $M$, such that for every $\omega \in \Omega$, there is a real-analytic diffeomorphism $\kappa_\omega : U_\omega \to V_\omega$ where $V_\omega$ is an open subset of the torus $\mathbb{T}^d = \mathbb{R}^d / \mathbb{Z}^d$. We assume in addition that the $\kappa_\omega$'s have Jacobian identically equal to $1$ and that the $U_\omega$'s are so small that for every $\omega \in \Omega$ and $k \in \mathbb{Z}^d$ the diameter (for the Kohn--Nirenberg distance) of
\begin{equation*}
\mathcal{W}_{\omega,k} = \set{\alpha \in T^* M : \alpha_x \in U_\omega, {}^t (D_{\alpha_x} \kappa_\omega)^{-1} \alpha_\xi = 2 \pi k}.
\end{equation*}
is less than $\varpi / 10$, where $\varpi$ is the constant defined above.

Let $(\chi_\omega)_{\omega \in \Omega}$ be a partition of unity subordinated to $(U_\omega)_{\omega \in \Omega}$. For every $\omega \in \Omega$, choose a $C^\infty$ function $\tilde{\chi}_\omega : M \to [0,1]$, supported in $U_\omega$ and such that $\tilde{\chi}_\omega \equiv 1$ on the support of $\chi_\omega$. Let also $\theta_\omega,\rho_\omega$ be $C^\infty$ functions from $\mathbb{T}^d$ to $[0,1]$ supported in $V_\omega$ and such that $\rho_\omega \equiv 1$ on a neighbourhood of the support of $\tilde{\chi}_\omega \circ \kappa_\omega^{-1}$ and $\theta_\omega \equiv 1$ on the neighbourhood of the support of $\rho_\omega$.

Introduce the operators $A_\omega = S \chi_\omega T$,$\widetilde{A}_\omega = S\tilde{\chi}_\omega T$ and $B_\omega = \mathcal{S}\rho_\omega \mathcal{T}$. Here, if $f$ is a function on $M$, we identify it with a function on $T^* M$ depending only on the position $(\alpha_x,\alpha_\xi) \mapsto f(\alpha_x)$, and with the associated multiplication operator. Let us also define the functions
\begin{equation*}
e_k^\omega = A_\omega (\kappa_\omega)^* \theta_\omega B_\omega e_k \textup{ and } \tilde{e}_k^\omega = \widetilde{A}_\omega (\kappa_\omega)^* \theta_\omega B_\omega e_k,
\end{equation*}
where the pullback $(\kappa_\omega)^*$ denotes just the operator of composition by $\kappa_\omega$.

\begin{remark}
The $e_k^\omega$'s and $\tilde{e}_k^\omega$'s will play the roles of local version the $e_k$'s that are defined on $M$. We will see below (Lemma \ref{lemma:ekomega_analytic}) that the $e_k^\omega$'s and $\tilde{e}_k^\omega$'s are analytic. It might seem surprising since their definitions involve the cut-off functions $\theta_\omega$'s. Let us explain how one can see quickly that $e_k^\omega$ is analytic.

Since $e_k$ is analytic, its FBI transform $\mathcal{T} e_k$ decays exponentially fast, and it follows then from an inspection of the kernel of $\mathcal{S}$ that $B_\omega e_k$ is analytic (this is the idea behind the proof of Proposition \ref{proposition:analytic_from_decay}, see \cite[Lemma 2.6]{BJ20}). Multiplying by $\theta_\omega$, we lose the real-analytic property, \emph{but not everywhere}. Indeed, $(\kappa_\omega)^* \theta_\omega B_\omega e_k$ is real-analytic on a neighbourhood of the support of $\chi_\omega$. We can use this knowledge to prove that $T (\kappa_\omega)^* \theta_\omega B_\omega e_k(\alpha)$ is exponentially small when $\alpha_x$ is in the support of $\chi_\omega$ (see for instance Lemma \ref{lemma:chgt_manifold}). It follows that $\chi_\omega T(\kappa_\omega)^* \theta_\omega B_\omega e_k$ decays exponentially fast and thus that $e_k^\omega$ is real-analytic.

Notice that we do not need the presence of $B_\omega$ in order to prove the analyticity of $e_k^\omega$. It is needed for the proof of Proposition \ref{proposition:reconstruction}, which will play the role of the Fourier inversion formula in our context.
\end{remark}

Let study the $e_k^\omega$'s and $\tilde{e}_k^\omega$'s more precisely.

\begin{lemma}\label{lemma:localisation}
Let $\omega \in \Omega$ and $c > 0$. Then, there is a constant $C > 0$ such that for every $k \in \mathbb{Z}^d$ and $\alpha \in T^* M$ such that the (Kohn--Nirenberg) distance between $\alpha$ and $\mathcal{W}_{\omega,k}$ is more than $c$ we have
\begin{equation}\label{eq:exponential_localisation}
\va{T e_k^\omega(\alpha)} \leq C \exp\p{- \frac{\max(\brac{\alpha},\va{k})}{C}}.
\end{equation}
Moreover, there are constants $C,N > 0$ such that for every $\omega \in \Omega,k \in \mathbb{Z}^d$ and $\alpha \in T^* M$, we have 
\begin{equation}\label{eq:straight_local}
\va{T e_k^\omega(\alpha)} \leq C \brac{\alpha}^{N}.
\end{equation}

The same result holds with the $e_k^\omega$'s replaced by the $\tilde{e}_k^\omega$'s.
\end{lemma}

\begin{proof}
We will prove the result for the $e_k^\omega$'s, the proof for the $\tilde{e}_k^\omega$'s is the same. Let us start by noticing that we have
\begin{equation}\label{eq:transform_ekomega}
T e_k^\omega = TS \chi_\omega T(\kappa_\omega)^* \theta_\omega \mathcal{S} \rho_\omega \mathcal{T} e_k.
\end{equation}
We are going to prove first that $\chi_\omega T(\kappa_\omega)^* \theta_\omega  \mathcal{S} \rho_\omega \mathcal{T} e_k$ satisfies the bounds \eqref{eq:exponential_localisation} and \eqref{eq:straight_local} that we claimed for $T e_k^\omega$. Introducing the kernel of the operator $T(\kappa_\omega)^* \theta_\omega \mathcal{S}$ as in Remark \ref{remark:meaning_kernel}, we write for $\alpha = (\alpha_x,\alpha_\xi) \in T^* M$
\begin{equation}\label{eq:first_kernel_to_transform}
\chi_\omega T(\kappa_\omega)^* \theta_\omega \mathcal{S} \rho_\omega \mathcal{T} e_k(\alpha) = \int_{T^* \mathbb{T}^d} \chi_\omega(\alpha_x) T(\kappa_\omega)^* \theta_\omega \mathcal{S}(\alpha,\beta) \rho_\omega(\beta_x) \mathcal{T} e_k(\beta) \mathrm{d}\mathrm{\beta}.
\end{equation}
Pick some small $s > 0$ and split the domain of integration in \eqref{eq:first_kernel_to_transform} into the sets
\begin{equation}\label{eq:split_in_four}
\begin{split}
& \set{\beta: d_{KN}(\beta,{}^t D \kappa_\omega^{-1} \alpha) \geq s \textup{ and } |\beta_\xi - 2 \pi k| \geq s \brac{\beta_\xi}},\\
& \set{\beta: d_{KN}(\beta,{}^t D \kappa_\omega^{-1} \alpha) \geq s \textup{ and } |\beta_\xi - 2 \pi k| \leq s \brac{\beta_\xi}}, \\ &
\set{\beta: d_{KN}(\beta,{}^t D \kappa_\omega^{-1} \alpha) \leq s \textup{ and } |\beta_\xi - 2 \pi k| \geq s \brac{\beta_\xi}} \\ & \textup{and }
\set{\beta: d_{KN}(\beta,{}^t D \kappa_\omega^{-1} \alpha) \leq s \textup{ and } |\beta_\xi - 2 \pi k| \leq s \brac{\beta_\xi}}.
\end{split}
\end{equation}
We are going to estimate the integral on each of these sets separately.

We start with the first set. When $\alpha$ is such that $\alpha_x$ is in the support of $\chi_\omega$ and $\beta$ is such that $\beta_x$ is in the support of $\rho_\omega$, since $\theta_\omega \circ \kappa_\omega \equiv 1$ on a neighbourhood of $\alpha_x$ and $\theta_\omega \equiv 1$ on a neighbourhood of $\beta_x$, we may use the first point in Lemma \ref{lemma:chgt_manifold} to control the kernel $T(\kappa_\omega)^* \theta_\omega \mathcal{S}(\alpha,\beta)$ when ${}^t D \kappa_\omega^{-1} \alpha$ is at distance at least $s$ from $\beta$. We find that it is bounded by $C\exp( - (\brac{\alpha} + \brac{\beta})/C)$ for some $C > 0$. If in addition we have $|\beta_\xi - 2 \pi k| \geq s \brac{\beta_\xi}$, then we can bound $\mathcal{T}e_k(\beta)$ using \eqref{eq:better_torus}. Thus, we find that the integral over the first set in \eqref{eq:split_in_four} is bounded by $C \exp( - \max(\brac{\alpha} ,|k|)/C)$ for some $C > 0$. 

To deal with the second set, we notice that the exponential bound on the kernel $T(\kappa_\omega)^* \theta_\omega \mathcal{S}(\alpha,\beta)$ is still valid here. We can only use \eqref{eq:straightforward_torus} to bound $\mathcal{T}e_k(\beta)$, but since we only consider a set of $\beta$'s that are approximatively of the size $|k|$ (in particular, we are integrating on a set of measure bounded by $C \brac{k}^d$), we get that the integral over the second set in \eqref{eq:split_in_four} is bounded by $C \exp( - \max(\brac{\alpha} ,|k|)/C)$ for some $C > 0$ using only the bound on the kernel $T(\kappa_\omega)^* \theta_\omega \mathcal{S}(\alpha,\beta)$.

For the third set, notice that it has measure less than $C \brac{\alpha}^d$ and that $\alpha$ is approximately of the size of $\beta$ there. We can then bound the kernel $T(\kappa_\omega)^* \theta_\omega \mathcal{S}(\alpha,\beta)$ by $C \brac{\alpha}^{\frac{d}{2}}$ (using directly the definition of the transforms) and bound $\mathcal{T}e_k(\beta)$ using \eqref{eq:better_torus}. It follows that the integral over the third set in \eqref{eq:split_in_four} is also bounded by $C \exp( - \max(\brac{\alpha} ,|k|)/C)$ for some $C > 0$.

It remains to deal with the last set. As in the previous case, this set has measure at most $C \brac{\alpha}^d$ and the integrand is bounded by some polynomials in $\brac{\alpha}$ there (using \eqref{eq:straightforward_torus} and the definition of the transforms). Hence, the integral on the third set is at most $C \brac{\alpha}^N$ for some $C,N > 0$. This ends the proof of \eqref{eq:straight_local}. Now, if we assume in addition that the distance between $\alpha$ and $\mathcal{W}_{\omega,k}$ is larger than $c$ then, by taking $s$ small enough, we may ensure that the last set in \eqref{eq:split_in_four} is empty, which proves \eqref{eq:exponential_localisation}.

We proved that $\chi_\omega T(\kappa_\omega)^* \theta_\omega \mathcal{S}\rho_\omega \mathcal{T} e_k$ satisfies the bounds that we claimed for $T e_k^\omega$. From Lemma \ref{lemma:chgt_manifold} (or \cite[Lemma 2.9]{BJ20}, see also Remark \ref{remark:TS_away_diagonal}), we see that the kernel of $TS$ is exponentially decaying away from the diagonal, which is enough to end the proof of the lemma writing
\begin{equation*}
T e_k^\omega (\alpha) = \int_{T^* M} TS(\alpha,\beta)\chi_\omega T(\kappa_\omega)^* \theta_\omega \mathcal{S}\rho_\omega \mathcal{T} e_k(\beta) \mathrm{d}\beta,
\end{equation*}
and splitting as above the integral with respect to $\beta$'s that are close to or away from $\alpha$, and close or away from $\mathcal{W}_{\omega,k}$.
\end{proof}

We can then use Lemma \ref{lemma:localisation} and the properties of the FBI transform from \S \ref{subsection:FBI} to relate the $e_k^\omega$'s and $\tilde{e}_k^\omega$'s to real-analytic regularity.

\begin{lemma}\label{lemma:ekomega_analytic}
Let $\rho > 0$. There are $C > 0$ and $\epsilon \in (0,\epsilon_0)$ such that for every $\omega \in \Omega$ and $k \in \mathbb{Z}^d$ we have $e_k^\omega, \tilde{e}_k^\omega \in \mathcal{O}_\epsilon$ and
\begin{equation*}
\n{e_k^\omega}_{\mathcal{O}_\epsilon} \leq C e^{\rho \va{k}} \textup{ and } \n{\tilde{e}_k^\omega}_{\mathcal{O}_\epsilon} \leq C e^{\rho \va{k}}.
\end{equation*}
\end{lemma}

\begin{proof}
This result follows immediately from Proposition \ref{proposition:analytic_from_decay} and Lemma \ref{lemma:localisation}.
\end{proof}

\begin{lemma}\label{lemma:dual_ekomega}
Let $\epsilon \in(0,\epsilon_0)$. There are $C,\rho > 0$ such that for every $\omega \in \Omega, k \in \mathbb{Z}^d$ and $u \in \mathcal{O}_\epsilon$, we have
\begin{equation*}
\va{\langle u , e_k^\omega \rangle_{L^2}} \leq C \n{u}_{\mathcal{O}_\epsilon}e^{- \rho \va{k}} \textup{ and } \va{\langle u , \tilde{e}_k^\omega \rangle_{L^2}} \leq C \n{u}_{\mathcal{O}_\epsilon} e^{- \rho \va{k}},
\end{equation*}
where the scalar product is in $L^2(M)$.
\end{lemma}

\begin{proof}
As above, let us prove the result only for the $e_k^\omega$'s. Since $T$ is an isometry, we have
\begin{equation}\label{eq:ex_isometry}
\langle u , e_k^\omega \rangle_{L^2(M)} = \langle T u, T e_k^\omega \rangle_{L^2(T^* M)}.
\end{equation}
The result then follows from Proposition \ref{proposition:decay_from_analytic} and Lemma \ref{lemma:localisation}.
\end{proof}

The construction of the space $\mathcal{H}$ from Theorem \ref{theorem:anisotropic_space} is based on the representation formula for smooth function given in the following proposition.

\begin{proposition}\label{proposition:reconstruction}
There is an operator $\mathcal{K}$ with real-analytic kernel such that for every analytic function $u$, we have
\begin{equation}\label{eq:series}
u = \sum_{\substack{\omega \in \Omega \\ k \in \mathbb{Z}^d}} \langle u , e_k^\omega \rangle_{L^2} \tilde{e}_k^\omega + \mathcal{K}u,
\end{equation}
where the sum converges in the real-analytic topology.
\end{proposition}

The proof of Proposition \ref{proposition:reconstruction} is probably the most technical from the paper and can be found in Appendix \ref{appendix:kernel_estimates}.

For every $\omega \in \Omega$, choose a point $x_\omega \in \supp \chi_\omega$ (we may discard the $\omega$'s for which $\supp \chi_\omega$ is empty), and for every $k \in \mathbb{Z}^d$ let us write 
\begin{equation*}
G_\omega(k) = G(x_\omega, {}^t D_{x_\omega} \kappa_\omega (2 \pi k)),
\end{equation*}
where $G$ is the escape function defined in \S \ref{subsection:Anosov_diffeomorphisms}. For $\gamma > 0$, let us introduce the norm
\begin{equation*}
\n{u}_\gamma^2 \coloneqq \sum_{\substack{\omega \in \Omega \\ k \in \mathbb{Z}^d}} e^{-2 \gamma G_\omega(k)} \va{\langle u, e_k^\omega \rangle_{L^2}}^2 + \n{\mathcal{K}u}_{L^2}^2,
\end{equation*}
where $\mathcal{K}$ is the operator from Proposition \ref{proposition:reconstruction}.

\begin{lemma}\label{lemma:finite_norm}
Let $\epsilon \in (0,\epsilon_0)$. Then there are $C,\gamma_0 >0$ such that for every $0 < \gamma \leq \gamma_0$ if $u \in \mathcal{O}_\epsilon$ then $\n{u}_\gamma \leq C \n{u}_{\mathcal{O}_\epsilon} < + \infty$.
\end{lemma}

\begin{proof}
This is an immediate consequence of Lemma \ref{lemma:dual_ekomega}.
\end{proof}

When the conclusion from Lemma \ref{lemma:finite_norm} holds, we let $\mathcal{H}_{\gamma,\epsilon}$ denotes the completion of $\mathcal{O}_\epsilon$ for the norm $\n{\cdot}_\gamma$, and $\iota : \mathcal{O}_\epsilon \to \mathcal{H}_{\gamma,\epsilon}$ be the inclusion. The space $\mathcal{H}$ in Theorem \ref{theorem:anisotropic_space} will be $\mathcal{H}_{\gamma,\epsilon}$ with $\gamma$ small enough.

\subsection{Basic properties of the space}

Now that we have a family of spaces defined, we check that they satisfy the basic properties required in Theorem \ref{theorem:anisotropic_space} (those that do not involve Koopman operators). The existence of the inclusion $\iota$ following from the definition, we explain now how the inclusion $j$ is constructed in Lemma \ref{lemma:injection_dual}. After a discussion of the action of operators with real-analytic kernel (Definition \ref{definition:analytic_kernels}), we prove that our space is separable in Lemma \ref{lemma:separability}.

\begin{lemma}\label{lemma:injection_dual}
Let $\epsilon \in(0,\epsilon_0)$. Then there are $C,\gamma_0 >0$ such that for every $0 < \gamma \leq \gamma_0$ if $v \in \mathcal{O}_\epsilon$ then the linear form
\begin{equation*}
u \mapsto \langle u, v \rangle_{L^2}
\end{equation*}
on $\mathcal{O}_\epsilon$ extends to a continuous linear form $l_v$ on $\mathcal{H}_{\gamma,\epsilon}$. Here, the scalar product is in $L^2(M)$ as usual. Moreover, the map $j : v \mapsto l_{\bar{v}}$ is $\mathbb{C}$-linear and continuous from $\mathcal{O}_\epsilon$ to the dual of $\mathcal{H}_{\gamma,\epsilon}$.
\end{lemma}

\begin{proof}
This is a consequence from Lemma \ref{lemma:dual_ekomega} and Proposition \ref{proposition:reconstruction}.
\end{proof}

Operators with real-analytic kernels play an important role in the analysis below, so let us define precisely what we mean by an operator with real-analytic kernel acting on $\mathcal{H}_{\gamma,\epsilon}$.

\begin{definition}\label{definition:analytic_kernels}
Let $\epsilon \in (0,\epsilon_0)$. If $L \in \mathcal{O}_\epsilon(M \times M)$, then we can define the operator with kernel $L$, that we also denote by $L$, and is defined for $u$ an integrable function on $M$ by
\begin{equation*}
Lu (x) = \int_M L(x,y) u(y) \mathrm{d}y \textup{ for } x \in M.
\end{equation*}
Notice that there is $\epsilon' \in (0,\epsilon_0)$ such that the function $(x \mapsto (y \mapsto L(x,y))$ is bounded and holomorphic from a neighbourhood of $\overline{(M)}_{\epsilon'}$ (in $(M)_{\epsilon_0}$) to $\mathcal{O}_{\epsilon'}$. Moreover, Lemma \ref{lemma:dual_ekomega} and Proposition \ref{proposition:reconstruction} imply that, if $\gamma > 0$ is small enough, the inclusion $\iota : \mathcal{O}_\epsilon \to \mathcal{H}_{\gamma,\epsilon}$ and the map $j : \mathcal{O}_\epsilon \to \mathcal{H}_{\gamma,\epsilon}^*$ extend to bounded operators from $\mathcal{O}_{\epsilon'}$ respectively to $\mathcal{H}_{\gamma,\epsilon}$ and $\mathcal{H}_{\gamma,\epsilon}^*$. Here, we recall that it follows from the definitions of $\mathcal{O}_\epsilon$ and $\mathcal{O}_{\epsilon'}$ that $\mathcal{O}_{\epsilon} \cap \mathcal{O}_{\epsilon'}$ is dense in $\mathcal{O}_{\epsilon'}$. We can consequently define an operator from $\mathcal{H}_{\gamma,\epsilon}$ to itself (still denoted by $L$) by the formula
\begin{equation*}
L u = \iota \p{ x \mapsto j(L(x,\cdot))(u)} \textup{ for } u \in \mathcal{H}_{\gamma,\epsilon}.
\end{equation*}
Notice that, up to taking $\epsilon'$ slightly smaller, we can find $\epsilon'' > \epsilon'$ such that the operator $L : \mathcal{H}_{\gamma,\epsilon} \to \mathcal{H}_{\gamma,\epsilon}$ factorizes through an operator from $\mathcal{O}_{\epsilon''} \to \mathcal{O}_{\epsilon'}$. Hence, it follows from Lemma \ref{lemma:exponential_inclusion} that $L$ is an operator in the exponential class $1/d$, where $d$ denotes the dimension of $M$. Notice also that the trace class operator norm of the operator $L$ is less than $C \n{L}_{\mathcal{O}_\epsilon}$ for some constant $C$ that may depend on $\epsilon$ but not on $L$.
\end{definition}

The following notation will be interesting in order to give an expression for the norm of $\mathcal{H}_{\gamma,\epsilon}$ in Lemma \ref{lemma:expression_norm} below.

\begin{definition}
Let $\epsilon \in (0,\epsilon_0)$. Let $\mathcal{K}(x,y)$ denote the kernel of the operator $\mathcal{K}$ from Proposition \ref{proposition:reconstruction}. Working as in Definition \ref{definition:analytic_kernels}, we see that when $\gamma > 0$ is small enough and $u \in \mathcal{H}_{\gamma,\epsilon}$, we can define an analytic function on $M$ by the formula $x \mapsto j(\mathcal{K}(x,\cdot))(u)$. We introduce then the abbreviation, for $u \in \mathcal{H}_{\gamma,\epsilon}$,
\begin{equation*}
\n{\mathcal{K}u}_{L^2}^2 = \int_M \va{j(\mathcal{K}(x,\cdot))(u)}^2 \mathrm{d}x.
\end{equation*}
\end{definition}

With this notation and Lemma \ref{lemma:injection_dual}, we can express the norm of $\mathcal{H}_{\gamma,\epsilon}$.

\begin{lemma}\label{lemma:expression_norm}
Let $\epsilon \in (0,\epsilon_0)$ be small enough. Let $\gamma > 0$ be small enough (depending on $\epsilon$). Then for every $u \in \mathcal{H}_{\gamma,\epsilon}$ the sum 
\begin{equation}\label{eq:finite_sum}
\sum_{\omega \in \Omega, k \in \mathbb{Z}^d} e^{-2 \gamma G_\omega(k)} \va{l_{e_k^\omega}(u)}^2
\end{equation}
converges. Here, the notation $l_v$ for $v \in \mathcal{O}_\epsilon$ is from Lemma \ref{lemma:injection_dual}. Moreover,
\begin{equation}\label{eq:expression_norm}
\n{u}_\gamma^2 = \sum_{\omega \in \Omega, k \in \mathbb{Z}^d} e^{-2 \gamma G_\omega(k)} \va{l_{e_k^\omega}(u)}^2 + \n{\mathcal{K}u}_{L^2}^2.
\end{equation}
In particular, the image of $j$ is dense in the dual of $\mathcal{H}_{\gamma,\epsilon}$.
\end{lemma}

\begin{proof}
Notice that if $u \in \mathcal{H}_{\gamma,\epsilon}$ is in the image of $\mathcal{O}_\epsilon$ by $\iota$, then the convergence of the sum \eqref{eq:finite_sum} and the equality \eqref{eq:expression_norm} holds by definition. For a general $u \in \mathcal{H}_{\gamma,\epsilon}$, let $(u_n)_{n \in \mathbb{N}}$ be a sequence of elements of the image of $\iota$ that converges to $u$. Since $\n{u_n}_{\gamma} \underset{n \to + \infty}{\to} \n{u}_\gamma$ and the $u_n$'s satisfy \eqref{eq:expression_norm}, it follows from Fatou's lemma that
\begin{equation}\label{eq:first_inequality}
\sum_{\omega \in \Omega, k \in \mathbb{Z}^d} e^{-2 \gamma G_\omega(k)} \va{l_{e_k^\omega}(u)}^2 + \n{\mathcal{K}u}_{L^2}^2 \leq \n{u}_\gamma^2.
\end{equation}
In particular the sum \eqref{eq:finite_sum} is finite. To get the reversed inequality, use the second triangular inequality to write
\begin{equation*}
\begin{split}
& \p{\sum_{\omega \in \Omega, k \in \mathbb{Z}^d} e^{-2 \gamma G_\omega(k)} \va{l_{e_k^\omega}(u)}^2 + \n{\mathcal{K}u}_{L^2}^2}^{\frac{1}{2}} \\
    & \qquad \qquad \geq \p{\sum_{\omega \in \Omega, k \in \mathbb{Z}^d} e^{-2 \gamma G_\omega(k)} \va{l_{e_k^\omega}(u_n)}^2 + \n{\mathcal{K}u_n}_{L^2}^2}^{\frac{1}{2}} \\ & \qquad \qquad \qquad \qquad - \p{\sum_{\omega \in \Omega, k \in \mathbb{Z}^d} e^{-2 \gamma G_\omega(k)} \va{l_{e_k^\omega}(u - u_n)}^2 + \n{\mathcal{K}(u-u_n)}_{L^2}^2}^{\frac{1}{2}}\\
    & \qquad \qquad \geq \n{u_n}_\gamma - \n{u-u_n}_\gamma.
\end{split}
\end{equation*}
Here, we used \eqref{eq:first_inequality} and the fact that \eqref{eq:expression_norm} holds for $u_n$. Letting $n$ tends to $+ \infty$, we get the reversed inequality to \eqref{eq:first_inequality} and thus \eqref{eq:expression_norm} holds.

It follows from \eqref{eq:expression_norm} that if $u \in \mathcal{H}_{\gamma,\epsilon}$ is such that $j(v)(u) = 0$ for every $v \in \mathcal{O}_{\epsilon}$ then $u = 0$. Since $\mathcal{H}_{\gamma,\epsilon}$ is reflexive (it is a Hilbert space), it follows that the image of $\mathcal{O}_\epsilon$ by $j$ is dense in the dual of $\mathcal{H}_{\gamma,\epsilon}$
\end{proof}

From Lemma \ref{lemma:expression_norm}, we deduce another part of Theorem \ref{theorem:anisotropic_space}.

\begin{lemma}\label{lemma:separability}
Let $\epsilon \in (0,\epsilon_0)$. Let $\gamma > 0$ be small enough (depending on $\epsilon$). Then the Hilbert space $\mathcal{H}_{\gamma,\epsilon}$ is separable.
\end{lemma}

\begin{proof}
Let $D$ be a countable dense subset of $M$. It follows from Lemma \ref{lemma:expression_norm} that if $u \in \mathcal{H}_{\gamma,\epsilon}$ is such that $j(\mathcal{K}(x,\cdot))(u) = 0$ for every $x \in D$ and $j(e_k^\omega)(u)=0$ for every $\omega \in \Omega$ and $k \in \mathbb{Z}^d$, then $u = 0$. Consequently, the span of the $j(\mathcal{K}(x,\cdot))$ for $x \in D$ and the $j(e_k^\omega)$ for $\omega \in \Omega$ and $k \in \mathbb{Z}^d$ is dense in $\mathcal{H}_{\gamma,\epsilon}^*$. Hence, $\mathcal{H}_{\gamma,\epsilon}^*$, and thus $\mathcal{H}_{\gamma,\epsilon}$, is separable.
\end{proof}

\subsection{Action of Koopman operators}

We are finally ready to study the action of Koopman operators on our space, and complete the proof of Theorem \ref{theorem:anisotropic_space}.

Let us introduce $\Gamma = \Omega \times \mathbb{Z}^d$ (where $\Omega$ is the set indexing our family of real-analytic charts introduced at the beginning of the section) and the relation $\hookrightarrow$ on $\Gamma$ by $(\omega,k) \hookrightarrow (\omega',k')$ if and only if the Kohn--Nirenberg distance between $\mathcal{W}_{\omega,k}$ and $\mathcal{F}\p{\mathcal{W}_{\omega',k'}}$ is less than $\varpi/10$. Here, we recall that $\mathcal{F}$ denotes the symplectic lift of $F$, as defined in \S \ref{subsection:Anosov_diffeomorphisms}. It follows from \eqref{eq:approximate_decay} and our choice of the $U_\omega$'s that there is $C > 0$ such that if $(\omega,k) \hookrightarrow (\omega',k')$ and $k$ is large enough then
\begin{equation}\label{eq:decay_escape}
G_{\omega}(k) \leq G_{\omega'}(k') - C^{-1} \va{k}. 
\end{equation}
This estimate is at the core of the proof of Theorem \ref{theorem:anisotropic_space}, along with the following lemma.

\begin{lemma}\label{lemma:discard}
\begin{enumerate}[label=(\roman*)]
\item \label{item:discard_not} Let $\epsilon \in (0,\epsilon_0)$. There are $C,\delta, \tau >0$ such that, if $X \in B_{\epsilon,\delta}$ and $w \in \mathcal{O}_\epsilon$, for every $(\omega,k),(\omega',k') \in \Gamma$ such that $(\omega,k) \not\hookrightarrow (\omega',k')$, we have 
\begin{equation*}
\va{\langle \mathcal{L}_{\Psi_X \circ F,w} \tilde{e}_{k}^{\omega}, e_{k'}^{\omega'} \rangle_{L^2}} \leq C \n{w}_{\mathcal{O}_\epsilon} e^{ - \tau \max(\va{k},\va{k'})}.
\end{equation*}
\item Let $\epsilon \in (0,\epsilon_0)$ and $\rho > 0$. Then there are $C, \delta > 0$ such that if $X \in B_{\epsilon,\delta}$ and $w \in \mathcal{O}_\epsilon$, for every $(\omega,k),(\omega',k') \in \Gamma$
\begin{equation}\label{eq:weak_discard}
\va{\langle \mathcal{L}_{\Psi_X \circ F,w} \tilde{e}_{k}^{\omega}, e_{k'}^{\omega'} \rangle_{L^2}} \leq C \n{w}_{\mathcal{O}_\epsilon} e^{\rho \max(\va{k},\va{k'})}.
\end{equation}
\end{enumerate}
\end{lemma}

\begin{proof}
To prove \eqref{eq:weak_discard}, let us first apply Lemma \ref{lemma:ekomega_analytic} to find $\epsilon_1 \in (0,\epsilon_0)$ such that for every $(\omega,k) \in \Gamma$ we have $\tilde{e}_{k}^{\omega} \in \mathcal{O}_{\epsilon_1}$ with norm less than $C e^{\rho \va{k}}$. Let us then recall from \S \ref{subsection:koopman} that there are $\delta > 0$ and $\epsilon_2$ such that for every $X \in B_{\epsilon,\delta}$ and $w \in \mathcal{O}_\epsilon$ the operator $\mathcal{L}_{\Psi_X \circ F,w}$ is bounded from $\mathcal{O}_{\epsilon_1}$ to $\mathcal{O}_{\epsilon_2}$ with operator norm less than $C \n{w}_{\mathcal{O}_\epsilon}$ where $C$ depends on $F$ and $\epsilon$. Thus $\mathcal{L}_{\Psi_X \circ F,w} \tilde{e}_{k}^{\omega}$ is in $\mathcal{O}_{\epsilon_2}$ with norm less than $C \n{w}_{\mathcal{O}_\epsilon} e^{\rho \va{k}}$. The estimate \eqref{eq:weak_discard} is then a consequence of Lemma \ref{lemma:dual_ekomega}.

We move to the proof of \ref{item:discard_not}. Let us assume that $(\omega,k) \not\hookrightarrow (\omega',k')$ and use that $T$ is an isometry to write
\begin{equation}\label{eq:integral_transfer}
\langle \mathcal{L}_{\Psi_X \circ F,w} \tilde{e}_{k}^{\omega}, e_{k'}^{\omega'} \rangle_{L^2} = \int_{T^* M \times T^* M} T \mathcal{L}_{\Psi_X \circ F,w} S(\alpha,\beta) T \tilde{e}_{k}^{\omega} (\beta) \overline{ T e_{k'}^{\omega'}(\alpha)} \mathrm{d}\alpha \mathrm{d}\beta,
\end{equation}
where the kernel $T \mathcal{L}_{\Psi_X \circ F,w}S(\alpha,\beta)$ is the one that we described in Proposition \ref{proposition:localisation_graph}. Let us split the integral \eqref{eq:integral_transfer} into pieces. Choose some small $s > 0$ and split the domain of the integral in \eqref{eq:integral_transfer} into the four pieces:
\begin{equation}\label{eq:integral_transfer_split}
\begin{split}
&\set{d_{KN}(\alpha, \mathcal{W}_{\omega',k'}) \geq s, d_{KN}(\beta,\mathcal{W}_{\omega,k})\geq s}, \\ & \set{d_{KN}(\alpha, \mathcal{W}_{\omega',k'}) \leq s, d_{KN}(\beta,\mathcal{W}_{\omega,k})\leq s}, \\ & \set{d_{KN}(\alpha, \mathcal{W}_{\omega',k'}) \leq s, d_{KN}(\beta,\mathcal{W}_{\omega,k})\geq s}, \\ & \textup{ and } \set{d_{KN}(\alpha, \mathcal{W}_{\omega',k'}) \geq s, d_{KN}(\beta,\mathcal{W}_{\omega,k})\leq s}.
\end{split}
\end{equation}

If $\alpha$ and $\beta$ are both at distance more than $s$ respectively from $\mathcal{W}_{\omega',k'}$ and $\mathcal{W}_{\omega,k}$, then we can use the estimate \eqref{eq:exponential_localisation} from Lemma \ref{lemma:localisation} to bound $T e_{k'}^{\omega'} (\alpha)$ and $T \tilde{e}_{k}^{\omega}(\beta)$. We bound the kernel $T \mathcal{L}_{\Psi_X \circ F,w} S(\alpha,\beta)$ using \eqref{eq:small_exponential_bound} where we recall that $c$ may be made arbitrarily small by taking $\delta$ small enough. Hence, the part of the integral in \eqref{eq:integral_transfer} corresponding to the first set in \eqref{eq:integral_transfer_split} is bounded by $C \n{w}_{\mathcal{O}_\epsilon} e^{- \tau \max(\va{k},\va{k'})}$ for some $C,\tau > 0$.

Now, if $\alpha$ and $\beta$ are both at distance less than $s$ respectively from $\mathcal{W}_{\omega',k'}$ and $\mathcal{W}_{\omega,k}$, then we can only use \eqref{eq:straight_local} to bound $T e_{k'}^{\omega'} (\alpha)$ and $T \tilde{e}_{k}^{\omega}(\beta)$. However, since $(\omega,k) \not\hookrightarrow (\omega',k')$, the distance between $\beta$ and $\mathcal{F}(\alpha)$ is more than $ \varpi/20$, provided $s$ is small enough. We can then use \eqref{eq:exponential_decay_Koopman} to bound $T \mathcal{L}_{\Psi_X \circ F,w} S(\alpha,\beta)$. The part of the integral in \eqref{eq:integral_transfer} corresponding to the second set in \eqref{eq:integral_transfer_split} is consequently also bounded by $C \n{w}_{\mathcal{O}_\epsilon} e^{- \tau \max(\va{k},\va{k'})}$ for some $C,\tau > 0$ (notice that in this case the order of magnitude of $\alpha$ and $\beta$ are $|k'|$ and $|k|$).

Let us now consider the third case in \eqref{eq:integral_transfer_split}: $\alpha$ at distance less than $s$ from $\mathcal{W}_{\omega',k'}$ and $\beta$ at distance more than $s$ from $\mathcal{W}_{\omega,k}$ (we will not detail the symmetric case, which is similar). In that case, we can bound $T \tilde{e}_{k}^{\omega}(\beta)$ using \eqref{eq:exponential_localisation}. If $\mathcal{F}(\alpha)$ is away from $\beta$, then we can use \eqref{eq:exponential_decay_Koopman} to bound $T \mathcal{L}_{\Psi_X \circ F,w} S(\alpha,\beta)$, since the order of magnitude of $\alpha$ is $|k'|$. If $\mathcal{F}(\alpha)$ is close to $\beta$, then the size of $\va{k}$ and $\va{k'}$ are approximately the same, so that applying \eqref{eq:exponential_localisation} to bound $T e_{k'}^{\omega'}(\beta)$ is enough to get the required decay (taking $\delta$ small enough so that $c$ in \eqref{eq:small_exponential_bound} is small enough).
\end{proof}

We will use Lemma \ref{lemma:discard} to estimate the norm of $\mathcal{L}_{\Psi_X \circ F,w} \tilde{e}_k^\omega$ in $\mathcal{H}_{\gamma,\epsilon}$. We will also need the following bound to do so.

\begin{lemma}\label{lemma:estimee_analytic_rest}
Let $\epsilon \in(0,\epsilon_0)$. There are $C,\delta, \tau >0$ such that, for every $(\omega,k) \in \Gamma, X \in B_{\epsilon,\delta}$ and $w \in \mathcal{O}_\epsilon$, we have
\begin{equation*}
\n{\mathcal{K}(\mathcal{L}_{\Psi_X \circ F,w} \tilde{e}_{k}^\omega)}_{L^2} \leq C \n{w}_{\mathcal{O}_\epsilon} e^{- \tau \va{k}}.
\end{equation*}
\end{lemma}

\begin{proof}
Let us rewrite for $x \in M$:
\begin{equation*}
\begin{split}
\mathcal{K}(\mathcal{L}_{\Psi_X \circ F,w} \tilde{e}_{k}^\omega) (x) & = \int_M \mathcal{K}(x,y) \mathcal{L}_{\Psi_X \circ F,w} \tilde{e}_{k}^\omega(y) \mathrm{d}y \\
     & = \int_M \mathcal{L}_{\Psi_X \circ F,w}^* (\mathcal{K}(x,\cdot))(y) \tilde{e}_{k}^\omega(y) \mathrm{d}y,
\end{split}
\end{equation*}
where we identifed the operator $\mathcal{K}$ with its kernel $\mathcal{K}(x,y)$, and the operator $\mathcal{L}_{\Psi_X \circ F,w}^*$ is discussed in \S \ref{subsection:koopman}. Notice that by taking $\delta$ small enough we may ensure that there is $\epsilon' > 0$ such that for every $x \in M$ the function $\mathcal{L}_{\Psi_X \circ F,w}^* (\mathcal{K}(x,\cdot))$ is an element of $\mathcal{O}_{\epsilon'}$ of norm less than $C \n{w}_{\mathcal{O}_\epsilon}$. The result then follows from Lemma \ref{lemma:dual_ekomega}.
\end{proof}

We are now ready to to estimate the norm of $\mathcal{L}_{\Psi_X \circ F,w} \tilde{e}_k^\omega$ in $\mathcal{H}_{\gamma,\epsilon}$.

\begin{lemma}\label{lemma:individual_bound}
Let $\epsilon \in (0,\epsilon_0)$. There are constants $C,\tau,\gamma_0,\delta > 0$ such that if $X \in B_{\epsilon,\delta}, w \in \mathcal{O}_\epsilon$ and $0 < \gamma \leq \gamma_0$ then, for every $(\omega,k) \in \Gamma$, we have
\begin{equation*}
\n{\mathcal{L}_{\Psi_X \circ F,w} \tilde{e}_{k}^\omega}_\gamma \leq C \n{w}_{\mathcal{O}_\epsilon} e^{- \tau \va{k} - \gamma G_\omega(k)}.
\end{equation*}
\end{lemma}

\begin{proof}
Let us compute
\begin{equation*}
\begin{split}
& \sum_{\substack{\omega' \in \Omega \\ k' \in \mathbb{Z}^d}} e^{-2 \gamma G_{\omega'}(k')} \va{\langle \mathcal{L}_{\Psi_X \circ F,w} \tilde{e}_{k}^{\omega}, e_{k'}^{\omega'}\rangle_{L^2}}^2 \\ & \qquad \qquad \leq \sum_{\substack{\omega' \in \Omega, k' \in \mathbb{Z}^d \\ (\omega,k) \hookrightarrow (\omega',k')}} e^{-2 \gamma G_{\omega'}(k')} \va{\langle \mathcal{L}_{\Psi_X \circ F,w} \tilde{e}_{k}^{\omega}, e_{k'}^{\omega'}\rangle_{L^2}}^2 \\ & \qquad \qquad \qquad \qquad \qquad + \sum_{\substack{\omega' \in \Omega, k' \in \mathbb{Z}^d \\ (\omega,k) \not\hookrightarrow (\omega',k')}} e^{-2 \gamma G_{\omega'}(k')} \va{\langle \mathcal{L}_{\Psi_X \circ F,w} \tilde{e}_{k}^{\omega}, e_{k'}^{\omega'}\rangle_{L^2}}^2 \\
     & \qquad \qquad \leq e^{-2 \gamma G_\omega(k)} \sum_{\substack{\omega' \in \Omega, k' \in \mathbb{Z}^d \\ (\omega,k) \hookrightarrow (\omega',k')}} e^{- \frac{2 \gamma \va{k}}{C}} \va{\langle \mathcal{L}_{\Psi_X \circ F,w} \tilde{e}_{k}^{\omega}, e_{k'}^{\omega'}\rangle_{L^2}}^2 \\ & \qquad \qquad \qquad \qquad \qquad + C^2 \n{w}_{\mathcal{O}_\epsilon}^2 \sum_{\substack{\omega' \in \Omega, k' \in \mathbb{Z}^d \\ (\omega,k) \not\hookrightarrow (\omega',k')}} e^{-2 \gamma G_{\omega'}(k') - 2 \tau \max(\va{k},\va{k'}) } \\
     & \qquad \qquad \leq C \n{w}_{\mathcal{O}_\epsilon}^2 e^{-2 \gamma G_\omega(k)}e^{- \frac{2 \gamma \va{k}}{C}} e^{\rho \va{k}}  + C \n{w}_{\mathcal{O}_\epsilon}^2 e^{- \frac{\va{k}}{C}} \\
     & \qquad \qquad \leq C \n{w}_{\mathcal{O}_\epsilon}^2 e^{-2 \gamma G_\omega(k)} e^{ - \frac{\va{k}}{C}}
\end{split}
\end{equation*}
where the constant $C$ may change from one line to another. To go from the second to the third line, we applied Lemma \ref{lemma:discard} and the estimate \eqref{eq:decay_escape} that is valid when $(\omega,k) \hookrightarrow (\omega',k')$. To go from the third line to the fourth one, we assumed that $\gamma$ is small enough (to deal with the second sum) and applied the estimate \eqref{eq:weak_discard} in Lemma \ref{lemma:discard}. Notice in particular that $\rho$ may be chosen arbitrarily small by taking $\delta$ small, which allows us to get from the fourth line to the last line. 

It remains to bound $\n{\mathcal{K}(\mathcal{L}_{\Psi_X \circ F,w} \tilde{e}_{k}^\omega)}_{L^2}$. To do so, we apply Lemma \ref{lemma:estimee_analytic_rest} to get
\begin{equation*}
\n{\mathcal{K}(\mathcal{L}_{\Psi_X \circ F,w} \tilde{e}_{k}^\omega)}_{L^2} \leq C \n{w}_{\mathcal{O}_\epsilon} e^{ - \frac{\tau}{2} |k| + \gamma G_\omega(k)}  e^{- \frac{\tau}{2} |k| - \gamma G_\omega(k)}.
\end{equation*}
Since $G_\omega(k)$ is controlled by $|k|$, the factor $e^{ - \frac{\tau}{2} |k| + \gamma G_\omega(k)}$ is bounded provided $\gamma$ is small enough, and the result follows.
\end{proof}

Notice that the norm of the linear form $l_{e_k^\omega}$, defined in Lemma \ref{lemma:injection_dual}, on $\mathcal{H}_{\gamma,\epsilon}$ is less than $e^{\gamma G_\omega(k)}$. Consequently, under the assumptions of Lemma \ref{lemma:individual_bound}, for $\omega\in \Omega$ and $k \in \mathbb{Z}^d$ large, the norm of the rank $1$ operator 
\begin{equation*}
\iota(\mathcal{L}_{\Psi_X \circ F,w} \tilde{e}_{k}^\omega) \otimes l_{e_k^\omega}
\end{equation*}
on $\mathcal{H}_{\gamma,\epsilon}$ is less than $Ce^{- \beta \va{k}}$, for some $C,\beta > 0$. From the estimate above, the sum 
\begin{equation*}
\sum_{\omega \in \Omega,k \in \mathbb{Z}^d} \iota(\mathcal{L}_{\Psi_X \circ F,w} \tilde{e}_{k}^\omega) \otimes l_{e_k^\omega}
\end{equation*}
converges to an operator from $\mathcal{H}_{\gamma,\epsilon}$ to itself of exponential class $1/d$ (see Lemma \ref{lemma:concrete}). Notice also that if $X \in B_{\epsilon,\delta}$, with $\delta$ small enough, and $w \in \mathcal{O}_\epsilon$, then the operator $\mathcal{L}_{\Psi_X \circ F,w} \mathcal{K}$ has a real-analytic kernel, and thus extends to an operator of exponential class $1/d$ on $\mathcal{H}_{\gamma,\epsilon}$, see Definition \ref{definition:analytic_kernels}. Thus, if $\epsilon > 0$, and $\gamma, \delta > 0$ are small enough (depending on $\epsilon$), we may define an operator
\begin{equation}\label{eq:def_L_tilde}
\widetilde{\mathcal{L}}_{\Psi_X \circ F,w} = \sum_{\omega \in \Omega,k \in \mathbb{Z}^d} \iota(\mathcal{L}_{\Psi_X \circ F,w} \tilde{e}_{k}^\omega) \otimes l_{e_k^\omega} + \mathcal{L}_{\Psi_X \circ F,w} \mathcal{K}
\end{equation}
on $\mathcal{H}_{\gamma,\epsilon}$. This is the operator from Theorem \ref{theorem:anisotropic_space}. In the following two lemmas, we check that it satisfies the required properties.

\begin{lemma}\label{lemma:properties_operator}
Let $\epsilon \in (0,\epsilon_0)$. There are constants $C,\gamma_0,\delta,a > 0$ such that, if $\gamma \in (0,\gamma_0)$, then, for every $X \in B_{\epsilon,\delta}$ and $w \in \mathcal{O}_\epsilon$, the operator $\widetilde{\mathcal{L}}_{\Psi_X \circ F,w}$ on $\mathcal{H}_{\gamma,\epsilon}$ is of exponential class $1/d$. 

Moreover, if $z \mapsto X(z)$ is a holomorphic family of elements of $B_{\epsilon,\delta}$ and $z \mapsto w_z$ is a holomorphic family of elements of $\mathcal{O}_\epsilon$ then $z \mapsto \widetilde{\mathcal{L}}_{\Psi_{X(z)} \circ F,w_z}$ is a holomorphic family of trace class operator on $\mathcal{H}_{\gamma,\epsilon}$.
\end{lemma}

\begin{proof}
That $\widetilde{\mathcal{L}}_{\Psi_X \circ F,w}$ is of exponential class $1/d$ follows from Lemma \ref{lemma:concrete}. The holomorphic dependence on the parameter follows from the uniform convergence in \eqref{eq:def_L_tilde}.
\end{proof}

\begin{lemma}\label{lemma:integrals_and_operators}
Let $\epsilon \in (0,\epsilon_0)$. There are constants $C,\gamma_0,\delta > 0$ such that, if $\gamma \in (0,\gamma_0)$, then, for every $X \in B_{\epsilon,\delta}^{\mathbb{R}},w \in \mathcal{O}_\epsilon, n \in \mathbb{N}$ and $u,v \in \mathcal{O}_\epsilon$ we have
\begin{equation}\label{eq:natural_inclusion}
j(u)\p{\widetilde{\mathcal{L}}_{\Psi_X \circ F,w}^n \iota(v)} = \int_M u (\mathcal{L}^n_{\Psi_X \circ F,w}v) \mathrm{d}x.
\end{equation}
\end{lemma}

\begin{proof}
Let $\tilde{\epsilon}$ and $\delta$ be so small that for every $X \in B_{\epsilon,\delta}^{\mathbb{R}}$ and $w \in \mathcal{O}_\epsilon$ the operator $\mathcal{L}_{\Psi_X \circ F,w}^*$ defined in \S \ref{subsection:koopman} is bounded from $\mathcal{O}_\epsilon$ to $\mathcal{O}_{\tilde{\epsilon}}$. By taking $\gamma$ small enough, we ensure that the map $j : \mathcal{O}_\epsilon \to \mathcal{H}_{\gamma,\epsilon}^*$ extends to a continuous map $j : \mathcal{O}_{\tilde{\epsilon}} \to \mathcal{H}_{\gamma,\epsilon}^*$ (this is a consequence of Lemma \ref{lemma:dual_ekomega}).

The proof is by induction, the case $n = 0$ being a consequence of the definition of the maps $\iota$ and $j$. Let $n \in \mathbb{N}$ and assume that the equality \eqref{eq:natural_inclusion} holds for every $u,v \in \mathcal{O}_\epsilon$ (the other parameters are fixed). Let $u,v \in \mathcal{O}_\epsilon$ and consider a sequence $(v_m)_{m \in \mathbb{N}}$ of elements of $\mathcal{O}_\epsilon$ such that $(\iota v_m)_{m \in \mathbb{N}}$ converges to $\widetilde{\mathcal{L}}_{\Psi_X \circ F,w}^n \iota v$ in $\mathcal{H}_{\gamma,\epsilon}$. We have consequently
\begin{equation*}
j(u)\p{\widetilde{\mathcal{L}}_{\Psi_X \circ F,w}^{n+1} \iota(v)} = \lim_{m \to + \infty} j(u)\p{\widetilde{\mathcal{L}}_{\Psi_X \circ F,w} \iota(v_m)}.
\end{equation*}
It follows from the definition \eqref{eq:def_L_tilde} of $\widetilde{\mathcal{L}}_{\Psi_X \circ F,w}$ that
\begin{equation*}
\widetilde{\mathcal{L}}_{\Psi_X \circ F,w} \iota(v_m) =  \sum_{\omega \in \Omega,k \in \mathbb{Z}^d}  \brac{v_m,e_k^\omega}_{L^2} \iota\p{ \mathcal{L}_{\Psi_X \circ F,w} \tilde{e}_{k}^\omega} + \iota\p{\mathcal{L}_{\Psi_X \circ F,w}  \mathcal{K} v_m}.
\end{equation*}
Thus
\begin{equation*}
\begin{split}
& j(u)\p{\widetilde{\mathcal{L}}_{\Psi_X \circ F,w} \iota(v_m)} \\ & \qquad \qquad = \sum_{\omega \in \Omega,k \in \mathbb{Z}^d}  \brac{v_m,e_k^\omega}_{L^2} \int_M u\p{ \mathcal{L}_{\Psi_X \circ F,w} \tilde{e}_{k}^\omega} \mathrm{d}x + \int_M u \p{\mathcal{L}_{\Psi_X \circ F,w}  \mathcal{K} v_m} \mathrm{d}x \\
    & \qquad \qquad = \int_M u \p{\sum_{\omega \in \Omega,k \in \mathbb{Z}^d}  \brac{v_m,e_k^\omega}_{L^2} \mathcal{L}_{\Psi_X \circ F,w} \tilde{e}_{k}^\omega + \mathcal{L}_{\Psi_X \circ F,w}  \mathcal{K} v_m} \mathrm{d}x \\
    & \qquad \qquad = \int_M u \mathcal{L}_{\Psi_X \circ F,w} v_m \mathrm{d}x = \int_M (\mathcal{L}_{\Psi_X \circ F,w}^* u)v_m \mathrm{d}x \\
    & \qquad \qquad = j(\mathcal{L}_{\Psi_X \circ F,w}^* u)(\iota v_m).
\end{split}
\end{equation*}
Here, we can use Lemma \ref{lemma:dual_ekomega} to justify the interversion of the series and the integral (from the second to the third line). We use \eqref{eq:series} to go from the third to the fourth line. In the last line, we use the extension of $j$ to $\mathcal{O}_{\tilde{\epsilon}}$. Letting $m$ tends to $+ \infty$, we find that
\begin{equation*}
j(u)\p{\widetilde{\mathcal{L}}_{\Psi_X \circ F,w}^{n+1} \iota(v)} = j(\mathcal{L}_{\Psi_X \circ F,w}^* u)\p{\widetilde{\mathcal{L}}_{\Psi_X \circ F,w}^{n} \iota(v)}.
\end{equation*}
Since $\mathcal{O}_\epsilon$ is dense in $\mathcal{O}_{\tilde{\epsilon}}$, it follows from our induction hypothesis that
\begin{equation*}
\begin{split}
j(u)\p{\widetilde{\mathcal{L}}_{\Psi_X \circ F,w}^{n+1} \iota(v)} & = \int_M (\mathcal{L}_{\Psi_X \circ F,w}^* u) (\mathcal{L}_{\Psi_X \circ F,w}^{n} v) \mathrm{d}x \\
   & = \int_M u (\mathcal{L}_{\Psi_X \circ F,w}^{n+1} v) \mathrm{d}x.
\end{split}
\end{equation*}
This ends the proof of the lemma.
\end{proof}

Let us finally gather all our findings and prove Theorem \ref{theorem:anisotropic_space}.

\begin{proof}[Proof of Theorem \ref{theorem:anisotropic_space}]
We take $\mathcal{H} = \mathcal{H}_{\gamma,\epsilon}$ with $\gamma$ small enough. Remember that $\mathcal{H}$ is separable (Lemma \ref{lemma:separability}). The point \ref{item:injections} follows from the definition of the space $\mathcal{H}_{\gamma,\epsilon}$ and Lemmas \ref{lemma:injection_dual} and \ref{lemma:expression_norm}. The operator $\widetilde{\mathcal{L}}_{\Psi_X \circ F,g}$ is defined by the formula \eqref{eq:def_L_tilde} (provided $\delta$ is small enough) and satisfies points \ref{item:extension_operator} and \ref{item:holomorphic_family} according to Lemma \ref{lemma:properties_operator}. Finally, point \ref{item:integrals} is a consequence of Lemma \ref{lemma:integrals_and_operators}.
\end{proof}

\section{Consequences}\label{section:consequences}

\subsection{First consequences}

Let $M$ be a closed real-analytic manifold and $F \in \Anos^\omega(M)$. Let $\epsilon \in (0,\epsilon_0)$, where $\epsilon_0$ is as in \S \ref{subsection:definitions}. We will use the notation from Theorem \ref{theorem:anisotropic_space}. The goal of this section is to use Theorem \ref{theorem:anisotropic_space} to study the Ruelle resonances of $F$.  We start by proving Theorems \ref{theorem:upper_bound_resonances} and \ref{thm:upper_bound_determinant}, the proof of Theorems \ref{theorem:optimal_dense} and \ref{theorem:optimality_torus} can be found respectively in \S \ref{subsection:optimal_dense} and \S \ref{subsection:optimality_torus}.

We will only work with the properties given in Theorem \ref{theorem:anisotropic_space}. In particular, we will not need to go back to the construction of the space $\mathcal{H}$. The first step is to explain how the eigenvalues of the operators from Theorem \ref{theorem:anisotropic_space} are related to Ruelle resonances.

\begin{proposition}\label{proposition:determnation_spectrum}
Let $X \in B_{\epsilon,\delta}^{\mathbb{R}}$ and $w \in \mathcal{O}_\epsilon$. The non-zero eigenvalues of $\widetilde{\mathcal{L}}_{\Psi_X \circ F,w}$ are the Ruelle resonances of $\mathcal{L}_{\Psi_X \circ F,w}$ (counted with multiplicity).
\end{proposition}

\begin{proof}
Let $\lambda$ be a non-zero complex number. We let $\widetilde{E}_\lambda$ denote the generalized eigenspace of $\widetilde{\mathcal{L}}_{\Psi_X \circ F,w}$ associated to $\lambda$ and $\tilde{\Pi}_\lambda$ the associated spectral projector. Similarly, we let $E_\lambda$ be the space of generalized resonant states for $\mathcal{L}_{\Psi_X \circ F,w}$ and $\Pi_\lambda$ be the associated spectral projector, that is the residue at $\lambda$ of the meromorphic continuation of $(z - \mathcal{L}_{\Psi_X \circ F ,g})^{-1}$ given by Theorem \ref{theorem:resonances}. We recall that $\Pi_\lambda$ is a finite rank operator from $C^\infty(M)$ to $\mathcal{D}'(M)$, and that $E_\lambda$ is its image. Since $\widetilde{E}_\lambda$ and $E_\lambda$ are finite dimensional and $\iota(\mathcal{O}_\epsilon)$ and $\mathcal{O}_\epsilon$ are dense respectively in $\mathcal{H}$ and in $C^\infty(M)$, we see that the operators $\widetilde{\Pi}_\lambda \circ \iota : \mathcal{O}_\epsilon \to \widetilde{E}_\lambda$ and $\Pi_\lambda : \mathcal{O}_\epsilon \to E_\lambda$ are surjective. We will prove that these operators have the same kernel, which will imply the existence of an isomorphism between $\widetilde{E}_\lambda$ and $E_\lambda$, and hence the result.

Let $u,v \in \mathcal{O}_\epsilon$. Let us introduce for $z \in \mathbb{C}$ large
\begin{equation*}
\Psi_{u,v}(z) = \sum_{n \geq 0} z^{-(n+1)} \int_M u (\mathcal{L}_{\Psi_X \circ F,w}^n v) \mathrm{d}x = \int_{M} u (z - \mathcal{L}_{\Psi_X \circ F,w})^{-1} v \mathrm{d}x.
\end{equation*}
Thanks to the meromorphic continuation of the resolvent $(z - \mathcal{L}_{\Psi_X \circ F,w})^{-1}$ given by Theorem \ref{theorem:resonances}, we see that $\Psi_{u,v}$ has a meromorphic continuation to $\mathbb{C}$. Moreover, the residue of $\Psi_{u,v}$ at $\lambda$ is $\int_M u \Pi_\lambda v \mathrm{d}x$.

On the other hand, using the last point in Theorem \ref{theorem:anisotropic_space}, we see that for $|z|$ large, we have
\begin{equation*}
\Psi_{u,v}(z) = \ \sum_{n \geq 0} z^{-(n+1)} j(u)\p{\widetilde{\mathcal{L}}_{\Psi_X \circ F,w}^n \iota(v)} = j(u)\p{(z - \widetilde{\mathcal{L}}_{\Psi_X \circ F,w})^{-1} \iota v}.
\end{equation*}
Consequently, the residue of $\Psi_{u,v}$ at $\lambda$ is $j(u)\p{\widetilde{\Pi}_\lambda \iota(v)}$. Hence, we have
\begin{equation*}
\int_M u \Pi_\lambda v \mathrm{d}x = j(u)\p{\widetilde{\Pi}_\lambda \iota(v)}.
\end{equation*}
Since $\mathcal{O}_\epsilon$ is dense in $C^\infty(M)$ and the image of $j$ is dense in the dual of $\mathcal{H}$, we find that the kernels of $\widetilde{\Pi}_\lambda \circ \iota : \mathcal{O}_\epsilon \to \widetilde{E}_\lambda$ and $\Pi_\lambda : \mathcal{O}_\epsilon \to E_\lambda$ are the same, and thus there is an isomorphism between $\widetilde{E}_\lambda$ and $E_\lambda$.
\end{proof}

\begin{remark}
One could actually prove that the isomorphism between $\widetilde{E}_\lambda$ and $E_\lambda$ constructed in the proof of Proposition \ref{proposition:determnation_spectrum} conjugates the actions of $\widetilde{\mathcal{L}}_{\Psi_X \circ F,w}$ and $\mathcal{L}_{\Psi_X \circ F,w}$. In particular, the operators $\widetilde{\mathcal{L}}_{\Psi_X \circ F,w}$ and $\mathcal{L}_{\Psi_X \circ F,w}$ have the same Jordan blocks.
\end{remark}

In prevision of the proof of Theorem \ref{thm:upper_bound_determinant}, we relate the dynamical determinant \eqref{eq:dynamical_determinant} with the operators from Theorem \ref{theorem:anisotropic_space}.

\begin{proposition}\label{proposition:trace_formula}
Let $X \in B_{\epsilon,\delta}^{\mathbb{R}}$ and $w \in \mathcal{O}_\epsilon$. For every $n \in \mathbb{N}^*$, the trace of the trace class operator $\widetilde{\mathcal{L}}_{\Psi_X \circ F,w}^n$ is
\begin{equation}\label{eq:trace_formula}
\tr\p{\widetilde{\mathcal{L}}_{\Psi_X \circ F,w}^n} = \sum_{(\Psi_X \circ F)^n x = x} \frac{\prod_{k = 0}^{n-1} w((\Psi_X \circ F)^k x)}{\va{\det\p{ I - D_x (\Psi_X \circ F)^n}}}.
\end{equation}
\end{proposition}

\begin{proof}
Since $\widetilde{\mathcal{L}}_{\Psi_X \circ F,w}^n$ is a trace class operator, its trace is the sum of its eigenvalues (counted with multiplicities), which by Proposition \ref{proposition:determnation_spectrum} is also the sum of the Ruelle resonances of $\mathcal{L}_{\Psi_X \circ F,w}^n$. Since $\Psi_X \circ F$ and $w$ are real-analytic, they are in particular Gevrey, so that it follows from \cite[Theorem 2.12 (iv)]{local_and_global} that the sum of the Ruelle resonances of $\mathcal{L}_{\Psi_X \circ F,w}^n$. coincides with the right-hand side of \eqref{eq:trace_formula}.
\end{proof}

We are now ready to prove Theorems \ref{theorem:upper_bound_resonances} and \ref{thm:upper_bound_determinant}.

\begin{proof}[Proof of Theorems \ref{theorem:upper_bound_resonances} and \ref{thm:upper_bound_determinant}]
It follows from Proposition \ref{proposition:trace_formula} that for every $z \in \mathbb{C}$ we have $d_{F,w}(z) = \det(I - z \widetilde{\mathcal{L}}_{F,g})$. The results then follow from Lemma \ref{lemma:order_exponential} since $\widetilde{\mathcal{L}}_{F,w}$ is of exponential class $1/d$: Theorem \ref{thm:upper_bound_determinant} follows from the upper bound on $\det(I - z \widetilde{\mathcal{L}}_{F,g})$ and Theorem \ref{theorem:upper_bound_resonances} from \eqref{eq:bound_eigenvalues}.
\end{proof}

\subsection{Proof of Theorem \ref{theorem:optimal_dense}}\label{subsection:optimal_dense}

We are going to prove a slightly more general statement than Theorem \ref{theorem:optimal_dense}. To do so, we will need the following definition.

\begin{definition}\label{definition:analytic_curve}
Let $I$ be an interval of $\mathbb{R}$ and $M$ a closed real-analytic manifold. We say that a function $c : I \to \Anos^\omega(M)$ is a real-analytic curve if the map $(t,x) \mapsto c(t)(x)$ is real-analytic. 

We define similarly real-analytic curves from $I$ to $C^\omega(M)$, and we say that a function from $I$ to $\Anos^\omega(M) \times C^\omega(M)$ is a real-analytic curve if its components are.
\end{definition}

Notice that this definition coincides with the one given in \cite{kriegl_michor_real_analytic} using the structure of real-analytic manifold on the space of real-analytic maps from $M$ to itself (according to \cite[Lemma 8.6]{kriegl_michor_real_analytic}). With this definition, we can state a generalization of Theorem \ref{theorem:optimal_dense}.

\begin{thm}\label{theorem:more_general}
Let $M$ be a closed real-analytic manifold of dimension $d$. Let $\mathcal{Q}$ be a subset of $\Anos^\omega(M) \times C^\omega(M)$ such that
\begin{enumerate}[label=(\roman*)]
\item $\mathcal{Q}$ is connected for the topology induced by the $C^1$ topology;
\item for every $x \in \mathcal{Q}$, there is a neighbourhood $U$ of $x$ in $\mathcal{Q}$ (for the $C^1$ topology) such that for every $y,z \in U$ there is a real-analytic curve that joins $y$ and $z$ in $\mathcal{Q}$;
\item there is $(F_0,w_0) \in \mathcal{Q}$ such that 
\begin{equation*}
\limsup_{r \to 0} \frac{\log N_{F_0,w_0}(r)}{\log |\log r|} = d.
\end{equation*}
Then, for every $(F,w) \in \mathcal{Q}$, there is a sequence $((F_n,w_n))_{n \geq 1}$ of elements of $\mathcal{Q}$ that converges to $(F,w)$ in the $C^\omega$ topology and such that
\begin{equation*}
\limsup_{r \to 0} \frac{\log N_{F_n,w_n}(r)}{\log |\log r|} = d
\end{equation*}
for every $n \in \mathbb{N}$.
\end{enumerate}
\end{thm}

Theorem \ref{theorem:optimal_dense} is deduced from Theorem \ref{theorem:more_general} by taking $\mathcal{Q} = W \times \set{1}$. The following lemma ensures that the hypotheses of Theorem \ref{theorem:more_general} hold in that case.

\begin{lemma}
Let $F \in \Anos^\omega(M)$. Then there is a neighbourhood $U$ of $F$ in $\Anos^\omega(M)$ for the $C^1$ topology such that if $G,H \in U$ then there is a real-analytic curve that joins $G$ and $H$.
\end{lemma}

\begin{proof}
Recall that there is a $C^1$ neighbourhood $V$ of $0$ in $\mathcal{V}^\omega$ such that if $X \in V$ then $\Psi_X$ is a diffeomorphism of $M$. Moreover, we may assume that $V$ is convex (since the $C^1$ topology is normable). Notice that
\begin{equation*}
U = \set{ \Psi_X \circ F : X \in V \textup{ and } \Psi_X \in \Diff^\omega(M)}
\end{equation*}
is a neighbourhood of $F$ for the $C^1$ topology. By taking $V$ small enough, we may ensure that $U$ is contained in $\Anos^\omega(M)$ \cite{original_anosov}. Now, if $G$ and $H$ are in $U$, then there are $X_0,X_1 \in V$ such that $G = \Psi_{X_0} \circ F$ and $H = \Psi_{X_1} \circ F$. If we set $c(t) = \Psi_{X_t} \circ F$ where $X_t = (1-t)X_0 + t X_1$ for $t \in [0,1]$, then $c: [0,1] \to \Anos^\omega(M)$ is a real-analytic curve (since the exponential map associated to a real-analytic metric is real-analytic). Notice here that the fact that $c(t) \in \Anos^\omega(M)$ for every $t \in [0,1]$ is ensured by the fact that $V$ is convex.
\end{proof}

The proof of Theorem \ref{theorem:more_general} is based on the strategy from \cite{bandtlow_naud,christiansen_several,christiansen_euclidean,christiansen_potential,christiansen_schrodinger,christiansen_hyperbolic}. This method is based on potential theoretic tools. Let us start by recalling the following definition.

\begin{definition}
If $f$ is an entire function, we define the order of growth of $f$ as
\begin{equation*}
\limsup_{r \to + \infty} \sup_{\va{z} = r} \frac{\log(\max(1, \log \va{f(z)}))}{\log r}.
\end{equation*}
\end{definition}

Let us now relate the notion of order of growth with the number of resonances for Koopman operators.

\begin{lemma}\label{lemma:optimal_order}
Let $(F,w) \in \Anos^\omega(M) \times C^\omega(M)$. Then the order of growth of $z \mapsto d_{F,w}(e^z)$ is less than or equal to $d+1$ with equality if and only if
\begin{equation}\label{eq:optimal}
\limsup_{r \to 0} \frac{\log N_{F,w}(r)}{\log |\log r|} = d.
\end{equation}
\end{lemma}

\begin{proof}
That the order of growth of $z \mapsto d_{F,w}(e^z)$ is less than $d+1$ follows from Theorem \ref{thm:upper_bound_determinant}. It follows from Jensen's formula (as in the proof of Lemma \ref{lemma:order_exponential}) that if the order of growth of $z \mapsto d_{F,w}(e^z)$ is strictly less than $d+1$ then \eqref{eq:optimal} does not hold. Reciprocally, assume that \eqref{eq:optimal} does not hold and choose a real number $\tau$ such that
\begin{equation*}
\limsup_{r \to 0} \frac{\log N_{F,w}(r)}{\log |\log r|} < \tau < d.
\end{equation*}
Notice then that $N_{F,w}(r) \leq |\log r|^\tau$ for $r$ large enough. Let then $(\lambda_n)_{n \in \mathbb{N}}$ be the sequence of Ruelle resonances of $\mathcal{L}_{F,w}$ (ordered so that the modulus is decreasing). It follows from Theorem \ref{thm:upper_bound_determinant} that $d_{F,w}$ has genus zero, and thus for $z \in \mathbb{C}$ we have
\begin{equation*}
d_{F,w}(z) = \prod_{n = 0}^{+ \infty} (1 - z \lambda_n).
\end{equation*}
Thus, we have for $z \in \mathbb{C}$ large and $r$ small enough:
\begin{equation*}
\begin{split}
\log |d_{F,w}(z)| & \leq C \sum_{n = 0}^{+ \infty} \log(1 + |z| |\lambda_n|) \\
    & \leq C N_{F,w}(r) \log(1 + |z|) + |z| \sum_{n \geq N_{F,w}(r)}|\lambda_n| \\
    & \leq C  |\log r|^\tau \log(1 + |z|) + |z| \int_{0}^r N_{F,w}(t) \mathrm{d}t \\
    &  \leq C  |\log r|^\tau \log(1 + |z|) + |z| \int_{0}^r |\log t|^\tau \mathrm{d}t,
\end{split}
\end{equation*}
where the constant $C$ may change from one line to another. Here, we used Fubini's theorem to get
\begin{equation*}
\begin{split}
\sum_{n \geq N_{F,w}(r)}|\lambda_n| & = \sum_{n \geq N_{F,w}(r)} \int_0^r 1_{[0,|\lambda_n|]}(t) \mathrm{d}t = \int_0^r \sum_{n \geq N_{F,w}(r)} 1_{[0,|\lambda_n|]}(t) \mathrm{d}t \\
    & \leq \int_0^r \sum_{n \geq 0} 1_{[0,|\lambda_n|]}(t) \mathrm{d}t = \int_0^r N_{F,w}(t) \mathrm{d}t.
\end{split}
\end{equation*}
Noticing that
\begin{equation*}
\int_{0}^r |\log t|^\tau \mathrm{d}t = \int_{|\log r|}^{+ \infty} e^{-x}x^\tau \mathrm{d}x \underset{r \to 0}{\sim} r|\log r|^\tau
\end{equation*}
and taking $r = |z|^{-1}$, we find that for large $z \in \mathbb{C}$ we have
\begin{equation*}
\log |d_{F,w}(z)| \leq C \log(1 + |z|)^{1+\tau},
\end{equation*}
for some new constant $C$. It follows that the order of growth of the function $z \mapsto d_{F,w}(e^z)$ is of order at most $\tau + 1 < d+1$.
\end{proof}

The following technical lemma will be fundamental in order to understand how the asymptotic of the number of resonances varies along a real-analytic curve in $\Anos^\omega(M) \times C^\omega(M)$

\begin{lemma}\label{lemma:family_determinant}
Let $I$ be an interval in $\mathbb{R}$ and $c : I \to \Anos^\omega(M) \times C^\omega(M)$ be a real-analytic curve. There is a complex neighbourhood $\mathcal{I}$ of $I$ and a holomorphic function $D : \mathcal{I} \times \mathbb{C} \to \mathbb{C}$ such that the following holds:
\begin{enumerate}[label=(\roman*)]
\item for every $t \in I$ and every $z \in \mathbb{C}$, we have $D(t,z) = d_{c(t)}(e^z)$;
\item for every $t \in \mathcal{I}$, the order of growth of the entire function $z \mapsto D(t,z)$ is less than $d+1$.
\end{enumerate}
\end{lemma}

\begin{proof}
Let us write $c(t) = (F_t,w_t)$ for $t \in I$. Thanks to the unique continuation principle, we only need to construct $D$ locally in $t$. Let $t_0 \in I$. From the real-analytic implicit function theorem, we know that there is a real-analytic function $(t,x) \mapsto X_t(x)$, from a neighbourhood of $\set{t_0} \times M$ in $\mathbb{R} \times M$ to $TM$, such that for every $t \in I$ near $t_0$ and $x \in M$, we have $X_t(F_{t_0}(x)) \in T_{F_{t_0}(x)} M$ and
\begin{equation*}
\exp_{F_{t_0}(x)}\p{X_t(F_{t_0}(x))} = F_t(x).
\end{equation*}
Notice that when $t$ is fixed, $X_t$ defines a real-analytic vector field on $M$. Using that the map $(t,x) \mapsto X_t(x)$ has a holomorphic extension to a complex neighbourhood of $\set{t_0} \times M$, we find that there is $\epsilon \in (0,\epsilon_0)$ such that $t \mapsto X_t$ is a real-analytic curve taking values in the Banach space $\mathcal{V}_\epsilon$. Similarly, up to making $\epsilon$ smaller, we find that $t \mapsto w_t$ is a real-analytic curve from $I$ to $\mathcal{O}_\epsilon$.

Apply then Theorem \ref{theorem:anisotropic_space}, for that value of $\epsilon$ and with $F = F_{t_0}$, to get $\delta > 0$ and $\mathcal{H}$ satisfying the conclusion of the theorem. Since $X_{t_0} = 0$, we find that for $t$ in a complex neighbourhood $U$ of $t_0$, we have $X_t \in B_{\epsilon,\delta}$ and $F_t = \Psi_{X_t} \circ F_{t_0}$. Up to making $U$ smaller, we can assume that the map $t \mapsto w_t$, originally defined from $I$ to $\mathcal{O}_\epsilon$, extends holomorphically as a map from $U$ to $\mathcal{O}_\epsilon$. Using the notation from Theorem \ref{theorem:anisotropic_space}, we define for $t \in U$ and $z \in \mathbb{C}$:
\begin{equation*}
D(t,z) = \det\p{ I - e^z \widetilde{\mathcal{L}}_{\Psi_{X_t} \circ F_{t_0},w_t}}.
\end{equation*}
If $t$ is real, it follows from Proposition \ref{proposition:trace_formula}, that $D(t,z) = d_{F_t,w_t}(e^z)$. It follows from (iii) in Theorem \ref{theorem:anisotropic_space} that $D$ is holomorphic. Finally, we get from point (ii) in Theorem \ref{theorem:anisotropic_space} and Lemma \ref{lemma:order_exponential} that for every $t \in U$ the order of growth of the entire function $z \mapsto D(t,z)$ is less than $d+1$.
\end{proof}

Before understanding how the number of resonances varies on $\mathcal{Q}$, we start by studying the particular case of a real-analytic curve. This is where potential theoretic tools are used.

\begin{lemma}\label{lemma:generic_curve}
Let $I$ be an interval in $\mathbb{R}$ and $c : I \to \Anos^\omega(M) \times C^\omega(M)$ be a real-analytic curve. Assume that there is a $t \in I$ such that 
\begin{equation}\label{eq:optimal_c}
\limsup_{r \to 0} \frac{\log N_{c(t)}(r)}{\log |\log r|} = d.
\end{equation}
Then the set of $t \in I$ such that \eqref{eq:optimal_c} does not hold has Hausdorff dimension zero.
\end{lemma}

\begin{proof}
Let $D$ be the holomorphic function from Lemma \ref{lemma:family_determinant}. We may assume that $\mathcal{I}$ is connected. We follow then the lines of the argument in \cite[p.309]{bandtlow_naud}. For $t \in \mathcal{I}$, we let $\rho(t)$ denotes the order of growth of the entire function $z \mapsto D(t,z)$. Let $U$ be a connected, relatively compact subset of $\mathcal{I}$ that contains $I$. Then, according to \cite[Proposition 1.40]{book_lelong_gruman} applied to the plurisubharmonic function $(t,z) \mapsto \max(1,\log |D(t,z)|)$, there is a sequence $(\psi_n)_{n \in \mathbb{N}}$ of subharmonic functions bounded above on $U$ such that for every $t \in U$ we have
\begin{equation*}
\limsup_{k \to + \infty} \psi_k(t)= \frac{1}{d+1} - \frac{1}{\rho(t)} \leq 0.
\end{equation*}
By assumption and Lemma \ref{lemma:optimal_order}, there is a $t \in I$ such that $\rho(t) = d+1$ and for such a $t$,
\begin{equation*}
\limsup_{k \to + \infty} \psi_k(t)= 0.
\end{equation*}
Hence, according to \cite[Proposition 1.39]{book_lelong_gruman}, the set of $t \in U$ such that $\rho(t) < d+1$ is polar. It is also a Borel set (since $\rho$ is measurable), and thus has Hausdorff dimension $0$ (see for instance \cite[Theorems 5.10 and 5.13]{hayman_kennedy_book}), and so does the intersection with $I$. However, according to Lemma \ref{lemma:optimal_order}, a point $t \in I$ is such that \eqref{eq:optimal_c} holds if and only if $\rho(t) = d+1$.
\end{proof}

With Lemma \ref{lemma:generic_curve}, we have all the tools to prove Theorem \ref{theorem:optimal_dense}.

\begin{proof}[Proof of Theorem \ref{theorem:more_general}]
Let us write $\mathcal{C}$ for the set of $x \in \mathcal{Q}$ such that
\begin{equation*}
\limsup_{r \to 0} \frac{\log N_{x}(r)}{\log |\log r|} = d.
\end{equation*}
We let $\mathcal{B}$ be the set of $x \in \mathcal{Q}$ such that there is a sequence $(x_n)_{n \in \mathbb{N}}$ of elements of $\mathcal{C}$ that converges to $x$ in the real-analytic topology. We want to prove that $\mathcal{B} = \mathcal{Q}$.

Let $x \in \mathcal{Q}$ be in the adherence of $\mathcal{C}$ for the $C^1$ topology. By assumption, there is a neighbourhood $U$ of $x \in \mathcal{Q}$ such that for every $y,z \in U$ there is a real-analytic curve in $\mathcal{Q}$ that joins $y$ and $z$. Since $x$ is in the adherence of $\mathcal{C}$, there is an element $y \in \mathcal{C} \cap U$. Now, if $z \in U$, there is a real-analytic curve that joins $y$ and $z$ and it follows from Lemma \ref{lemma:generic_curve} that $z$ actually belongs to $\mathcal{B}$. Thus, the neighbourhood $U$ of $x$ is included in $\mathcal{B}$, and thus in the adherence of $\mathcal{C}$. 

It follows that the adherence of $\mathcal{C}$ is open. Since $\mathcal{C}$ is not empty and $\mathcal{Q}$ is connected, we get that $\mathcal{Q}$ is the adherence of $\mathcal{C}$. Moreover, we showed that the adherence of $\mathcal{C}$ coincides with $\mathcal{B}$, so that $\mathcal{B} = \mathcal{Q}$.
\end{proof}

\subsection{Proof of Theorem \ref{theorem:optimality_torus}}\label{subsection:optimality_torus}

In this section, we specify to the case $M = \mathbb{T}^2$, and deduce Theorem \ref{theorem:optimality_torus} from Theorem \ref{theorem:optimal_dense}. We will rely on the examples given in \cite{optimal_examples}. For certain Anosov diffeomorphisms $F$ of $\mathbb{T}^2$, the authors of \cite{optimal_examples} construct a space $H_\nu$ on which $\mathcal{L}_{F,1}$ (denoted there by $C_F$) is trace class and has an explicit spectrum. They mention \cite[Remark 1.4]{optimal_examples} that the Fredholm determinant of $\mathcal{L}_{F,1}$ acting on $H_\nu$ coincides with its dynamical determinant defined by \eqref{eq:dynamical_determinant}. From Theorem \ref{theorem:determinant}, it follows that the spectrum of $\mathcal{L}_{F,1}$ on $H_\nu$ coincides with its Ruelle spectrum. This fact will be crucial in the proof of Theorem \ref{theorem:optimality_torus}, and since its proof is not detailed in \cite{optimal_examples}, we explain here why it holds.

\begin{lemma}\label{lemma:lesmemes}
If $F$ is one of the $C^\omega$ Anosov diffeomorphisms of $\mathbb{T}^2$ given in \cite[Theorem 1.3(i)]{optimal_examples} then, the spectrum of $\mathcal{L}_{F,1}$ on the space $H_\nu$ from \cite[Theorem 1.3]{optimal_examples} is the Ruelle spectrum of $\mathcal{L}_{F,1}$.
\end{lemma}

\begin{proof}
Using the notation from \cite{optimal_examples}, we recall that $(e_n)_{n \in \mathbb{Z}^2}$ denotes the family of trigonometric monomials on $\mathbb{T}^2$. Then, for every $n \in \mathbb{Z}^2$, the function $e_n$ belongs to $H_\nu$. We let $q_n$ denote $e_n / \n{e_n}_{H_\nu}$. Then, $(q_n)_{n \in \mathbb{Z}^2}$ is an orthonormal basis of $H_\nu$ (see the remark after Definition 2.2 in \cite{optimal_examples}). Moreover, it is proven in \cite{optimal_examples} that $\mathcal{L}_{F,1}$ is trace class on $H_\nu$. 

Hence, for every $k \in \mathbb{N}^*$ we have
\begin{equation*}
\tr \p{\mathcal{L}_{F,1}^k} = \sum_{n \in \mathbb{Z}^2} \brac{\mathcal{L}_{F,1}^k q_n, q_n }_{H_\nu}.
\end{equation*} 
The trace here is the trace of $\mathcal{L}_{F,1}^k$ as an operator on $H_\nu$. Moreover, it follows from \cite[Remark 4.7]{optimal_examples} that for $k \in \mathbb{N}^*$ and $n \in \mathbb{Z}^2$
\begin{equation*}
\brac{\mathcal{L}_{F,1}^k q_n, q_n }_{H_\nu} = \brac{\mathcal{L}_{F,1}^k e_n, e_n}_{L^2} = \int_{\mathbb{T}^2} e^{2 i \pi n \cdot (F^k(x) - x)} \mathrm{d}x.
\end{equation*}
For $k \in \mathbb{N}^*$, let $g_k$ be the map from $\mathbb{T}^2$ to itself defined by $g_k(x) = F^k(x) -x$.

The map $F$ is of the form given by \cite[(19)]{optimal_examples}. Let $k \in \mathbb{N}^*$. From the computation of $D F^k$ that is made in \cite[Proposition 5.5]{optimal_examples}, we find that, for every $x \in \mathbb{T}^2$, the number $1$ is not an eigenvalue of $D_xF^k$ (where we identify $D_xF$ with an endomorphism of $\mathbb{R}^2$ using the usual parallelization of $\mathbb{T}^2$). Indeed, $D_x F^k$ is, up to sign, a product of matrices of the form
\begin{equation*}
\left[ \begin{array}{cc}
a & 1 \\ 1 & 0
\end{array}\right]
\end{equation*}
with $a > 0$. Consequently, $(D_x F^k)^4$ is of the form $I + A$, where $A$ is a matrix with positive coefficients. It follows from Perron--Frobenius theorem that $(D_x F)^4$ has an eigenvector with positive coefficients, and it must consequently corresponds to an eigenvalue greater than $1$. Since the determinant of $(D_x F)^4$ is $1$, its other eigenvalue is in $(0,1)$. It implies that $1$ is not an eigenvalue of $(D_x F)^4$.

Consequently, for $k \in \mathbb{N}^*$, the map $g_k$ is a local diffeomorphism, and thus by changing variable we find that for $n \in \mathbb{Z}^2$ we have
\begin{equation*}
\brac{\mathcal{L}_{F,1}^k e_n, e_n}_{L^2} = \int_{\mathbb{T}^2} e^{2 i \pi n \cdot x} \sum_{F^k y -y = x} \frac{1}{\va{\det\p{I - D_y F^k}}} \mathrm{d}x.
\end{equation*}
We recognize here a Fourier coefficient of a smooth function and it follows from the Fourier inversion formula that
\begin{equation*}
\tr \p{\mathcal{L}_{F,1}^k} = \sum_{n \in \mathbb{Z}^2} \brac{\mathcal{L}_{F,1}^k e_n, e_n}_{L^2} = \sum_{\substack{x \in \mathbb{T}^2\\ F^k x = x}}\frac{1}{\va{\det\p{I - D_x F^k}}}.
\end{equation*}
Consequently, the Fredholm determinant of $\mathcal{L}_{F,1}$ acting on $H_\nu$ coincides with the dynamical determinant \eqref{eq:dynamical_determinant}, and it follows from Theorem \ref{theorem:determinant} that the spectrum of $\mathcal{L}_{F,1}$ on $H_\nu$ is its Ruelle spectrum.
\end{proof}

Now that this precision has been made, we can write the proof of Theorem \ref{theorem:optimality_torus}.

\begin{proof}[Proof of Theorem \ref{theorem:optimality_torus}]
Let $F$ be a real-analytic diffeomorphism of the torus and let $W$ denote the connected component of $W$ in $\Anos^\omega(\mathbb{T}^2)$. According to Theorem \ref{theorem:optimal_dense}, we only need to prove that $W$ contains an element $G$ such that
\begin{equation}\label{eq:optimal_again}
\limsup_{r \to 0} \frac{\log N_{G,1}(r)}{\log \va{\log r}} = d.
\end{equation}
Using the standard identification of the fundamental group of $\mathbb{T}^2$ with $\mathbb{Z}^2$, we let $A \in \SL(2,\mathbb{Z})$ be the matrix giving the action of $F$ on the fundamental group of $\mathbb{Z}^2$. By a result of Maning \cite{no_new}, the matrix $A$ is hyperbolic (it has no eigenvalues of modulus $1$). According to \cite[Theorem 1.3(i)]{optimal_examples}, there is a $C^\omega$ Anosov diffeomorphism $G$ on $M$ whose action on the fundamental group is given by $A$ and such that \eqref{eq:optimal_again} holds. Here, we use Lemma \ref{lemma:lesmemes} to identify the eigenvalues from \cite{optimal_examples} with the Ruelle resonances.

It follows from results of Franks and Manning \cite{franks_classification, no_new},  that $F$ and $G$ are topologically conjugate to the action of $A$ on $\mathbb{T}^2$ via a homeomorphism homotopic to the identity. Then, \cite[Theorem 1]{space_Anosov} implies that $G$ belongs to $W$. This reference only gives a path of $C^\infty$ diffeomorphisms joining $F$ and $G$. However, one can then deduce the existence of a path of $C^\omega$ diffeomorphisms joining $F$ and $G$ by a mollification argument. 
\end{proof}

\appendix

\section{Kernel estimates}\label{appendix:kernel_estimates}

In this appendix, we gather several technical estimates that are used at different places of the paper.

\begin{proof}[Proof of Proposition \ref{proposition:localisation_torus}]

The proof is based on an application of the holomorphic non-stationary phase method. However, there is a subtlety in the proof that forbids to apply directly the versions of this argument in the literature we are aware of. 

Let $\mathcal{K}_{\mathcal{T}}$ denotes the kernel of the FBI transform $\mathcal{T}$ on $\mathbb{T}^d$, and let $\mathfrak{a}$ and $\Phi_{\mathcal{T}}$ be the associated symbol and phase, from \eqref{eq:local_descrption}. By definition if $(x,\xi) \in T^* \mathbb{T}^d$ and $k \in \mathbb{Z}^d$ then
\begin{equation}\label{eq:explicit_transform_torus}
\mathcal{T} e_k (x,\xi) = \int_{\mathbb{T}^d} K_{\mathcal{T}}(x,\xi,y) e^{2 i \pi k \cdot y} \mathrm{d}y.
\end{equation}
The estimate \eqref{eq:straightforward_torus} then follows by using a straightforward $L^\infty$ bound on $\mathcal{K}_{\mathcal{T}}$.

Let now $c> 0$ and assume that $\va{\xi - 2 \pi k} > c \brac{\xi}$. Choose some small number $s > 0$. Since the integrand in \eqref{eq:explicit_transform_torus} is holomorphic, we may shift contour and replace the integral on $\mathbb{T}^d$ by an integral over $\mathbb{T}^d + is\frac{2 \pi k - \xi}{\va{2 \pi k - \xi}}$. After a change of variable, it yields
\begin{equation}\label{eq:torus_after_shift}
\mathcal{T} e_k (x,\xi) = \int_{\mathbb{T}^d} K_{\mathcal{T}}(x,\xi,y + i s \frac{2 \pi k - \xi}{\va{2 \pi k - \xi}}) e^{2 i \pi k \cdot y} e^{- s 2 \pi k \cdot \frac{2 \pi k - \xi}{\va{2 \pi k - \xi}}} \mathrm{d}y.
\end{equation}
The result will follow by bounding pointwise the integrand in this new integral. Let $r > 0$ be small. If the distance between $y$ and $x$ is larger than $r$, and $s$ is small enough (depending on $r$), it follows from \eqref{eq:pseudo_local} that we have
\begin{equation*}
\va{K_{\mathcal{T}}(x,\xi,y + i s \frac{2 \pi k - \xi}{\va{2 \pi k - \xi}})} \leq C \exp\p{ - \frac{\brac{\xi}}{C}},
\end{equation*}
where the constant $C > 0$ depends on $r$. Consequently, the integrand in \eqref{eq:torus_after_shift} is less than
\begin{equation*}
\begin{split}
C \exp\p{ - \frac{\brac{\xi}}{C} - s  2 \pi k \frac{2 \pi k - \xi}{\va{2 \pi k - \xi}} } & \leq C \exp\p{ - \frac{\brac{\xi}}{C} - s  \va{2 \pi k - \xi} - s \xi \cdot \frac{2 \pi k - \xi}{\va{2 \pi k - \xi}} } \\
   & \leq C \exp\p{ - \frac{\brac{\xi}}{C} - 2 \pi s \frac{c}{c+1} \va{k}  + s \va{\xi} } \\
   & \leq C \exp\p{ - \frac{\brac{\xi}}{2 C} - 2 \pi s \frac{c}{c+1} \va{k} },
\end{split}
\end{equation*}
where in the last line we assumed that $s < \frac{1}{2C}$. Let us now consider the case of $y$ and $x$ at distance less than $r$. We want to get a similar bound on the integrand in \eqref{eq:torus_after_shift}. To do so, we only need to bound
\begin{equation}\label{eq:local_integrand}
e^{i\Phi_{\mathcal{T}}(x,\xi,y + i s \frac{2 \pi k - \xi}{\va{2 \pi k - \xi}})} \mathfrak{a}(x,\xi, y + i s \frac{2 \pi k - \xi}{\va{2 \pi k - \xi}}) e^{2 i \pi k \cdot y} e^{- s 2 \pi k \cdot \frac{2 \pi k - \xi}{\va{2 \pi k - \xi}}}.
\end{equation}
Indeed, the error term in \eqref{eq:local_descrption} may be dealt with as in the previous case ($x$ away from $y$). Let us estimate the phase in \eqref{eq:local_integrand} using Taylor's formula
\begin{equation*}
\begin{split}
& \Phi_{\mathcal{T}}(x,\xi,y + i s \frac{2 \pi k - \xi}{\va{2 \pi k - \xi}}) \\ & \qquad \qquad \qquad = \Phi_{\mathcal{T}}(x,\xi,y) + i s \mathrm{d}_y \Phi_{\mathcal{T}}(x,\xi,y) \cdot \frac{2 \pi k - \xi}{\va{2 \pi k - \xi}} + \mathcal{O}\p{s^2 \brac{\xi}} \\
     \\ & \qquad \qquad \qquad = \Phi_{\mathcal{T}}(x,\xi,y) + i s \mathrm{d}_y \Phi_{\mathcal{T}}(x,\xi,x) \cdot \frac{2 \pi k - \xi}{\va{2 \pi k - \xi}} + \mathcal{O}\p{s (s+r) \brac{\xi}}.
\end{split}
\end{equation*}
Here, we used Cauchy's formula and \eqref{eq:control_phase} to bound the second derivative of $\Phi_{\mathcal{T}}$. Consequently, we have for some $C > 0$:
\begin{equation*}
\im \Phi_{\mathcal{T}}(x,\xi,y + i s \frac{2 \pi k - \xi}{\va{2 \pi k - \xi}}) \geq - s \xi \cdot \frac{2 \pi k - \xi}{\va{2 \pi k - \xi}} - C s (s+r) \brac{\xi}. 
\end{equation*}
It follows that \eqref{eq:local_integrand} is less than
\begin{equation*}
\begin{split}
& C \brac{\xi}^{\frac{d}{4}} \exp\p{- s \va{2 \pi k - \xi} + C s (s+r) \brac{\xi}} \\ & \qquad \qquad \qquad \leq C \brac{\xi}^{\frac{d}{4}} \exp\p{ - \frac{cs}{c+1} \max(\va{k}, \va{\xi})+ C s (s+r) \brac{\xi} }.
\end{split}
\end{equation*}
The result then follows by taking $s$ and $r$ small enough.
\end{proof}

\begin{remark}\label{remark:meaning_kernel}
The following proof establishes estimates on kernel of operators of the form $T \mathfrak{L} S$ where $T$ and $S$ are a real-analytic FBI transform on $M$ and its adjoint, as defined in \S \ref{subsection:FBI}, and $\mathfrak{L}$ is an operator that maps real-analytic functions on $M$ to smooth functions (we could deal with more general operators, but this is the only case that will appear here). Let us detail what we mean by this.

Pick $\epsilon_1 > 0$ small enough so that for every $\alpha \in T^* M$ the function $x \mapsto K_S(x,\alpha)$ is in $\mathcal{O}_{\epsilon_1}$, where $K_S$ is the kernel of $S$. Assume that $\mathfrak{L}$ is a bounded operator from $\mathcal{O}_{\epsilon_1}$ to $C^\infty(M)$, for $\alpha,\beta \in T^* M$, we define formally the kernel $T \mathfrak{L} S (\alpha,\beta)$ of $T\mathfrak{L}S$ as
\begin{equation*}
T\mathfrak{L}S(\alpha,\beta) = \int_M K_T(\alpha,y) \mathfrak{L}(K_S(\cdot,\beta))(y) \mathrm{d}y.
\end{equation*}
We use here the common practice to identify an operator with its kernel.

Such a kernel is interesting for the following reason. If $\epsilon \in (0,\epsilon_0)$ and $\epsilon_1$ is small enough, then one may use Proposition \ref{proposition:decay_from_analytic} and  \cite[Lemma 2.5]{BJ20} to find that, for every $u \in \mathcal{O}_\epsilon$ and $\alpha \in T^* M$, the function on $M \times T^* M$ given by $(y,\beta) \mapsto K_T(\alpha,y) \mathfrak{L}(K_S(\cdot,\beta))(y) Tu(\beta)$ is integrable on $M \times T^* M$. It follows then from Fubini's theorem that
\begin{equation}\label{eq:usage_kernel}
T \mathfrak{L} u(\alpha) = \int_{T^* M} T\mathfrak{L}S(\alpha,\beta) Tu(\beta) \mathrm{d}\beta.
\end{equation}
In particular, the integrand is integrable (one can even check that it decays exponentially fast). We will use such formulae to get estimates on $T\mathfrak{L}u$ from estimates on $Tu$ and on the kernel of $T\mathfrak{L}S$.

In Proposition \ref{proposition:localisation_graph}, we have $\mathfrak{L} = \mathcal{L}_{\Psi_X \circ F, w}$. When $X$ is tangent to $M$, this is just the composition of a multiplication and a composition operator, but this is in general more complicated as $\Psi_X$ may not leave $M$ invariant. In Lemma \ref{lemma:chgt_manifold} below, we study what happens when $\mathfrak{L}$ is the composition of a multiplication and a composition operator, but we allow there to multiply by a function which is only $C^\infty$ (rather than analytic) at some places.
\end{remark}

\begin{proof}[Proof of Proposition \ref{proposition:localisation_graph}]
Let us write the kernel of $T\mathcal{L}_{\Psi_X \circ F, w} S$ as
\begin{equation}\label{eq:def_kernel_Koopman}
T\mathcal{L}_{\Psi_X \circ F, w} S(\alpha,\beta) = \int_M K_T(\alpha,y) w(y) K_S(\Psi_X \circ F(y),\beta) \mathrm{d}y
\end{equation}
for $\alpha,\beta \in T^* M$. We start by proving \eqref{eq:small_exponential_bound}. By direct inspection of the kernel $K_S$ of $S$, we see that there is $\epsilon' > 0$ such that for every $y \in (M)_{\epsilon'}$ and $\beta \in T^* M$ we have
\begin{equation*}
\va{K_S(y,\beta)} \leq C \exp\p{\frac{c \brac{\beta}}{2}}.
\end{equation*}
By taking $\delta > 0$ small enough, we may ensure that for every $X \in B_{\epsilon,\delta}$ the map $\Psi_X$ sends $M$ into $(M)_{\epsilon'}$. Using a crude estimate to bound the kernel of $T$, we find that 
\begin{equation*}
\begin{split}
& \va{T\mathcal{L}_{\Psi_X \circ F, w} S(\alpha,\beta)} \\ & \qquad \quad \leq C \n{w}_{L^\infty} \brac{\alpha}^{\frac{d}{4}} \exp\p{\frac{c \brac{\beta}}{2}} \leq C \n{w}_{\mathcal{O}_\epsilon} \exp\p{c (\brac{\alpha} + \brac{\beta})}.
\end{split}
\end{equation*}

Let us now prove \eqref{eq:exponential_decay_Koopman}. We assume that the distance between $(\alpha,\beta)$ and $\mathcal{G}$ is larger than $c$. Let $s > 0$ be very small and assume first that the distance between $\alpha_x$ and $F^{-1} \beta_x$ is larger than $s/2$. Let us split the integral in \eqref{eq:def_kernel_Koopman} as
\begin{equation}\label{eq:split_kernel_Koopman}
\begin{split}
& T\mathcal{L}_{\Psi_X \circ F, w} S(\alpha,\beta) \\ & \qquad \qquad = \p{\int_{D(\alpha_x,s/100)} + \int_{D(F^{-1} \beta_x,s/100)} + \int_{M \setminus \p{D(\alpha_x,s/100) \cup D(F^{-1} \beta_x,s/100)}}} \\ & \qquad \qquad \qquad \qquad \qquad \qquad \qquad \qquad \qquad K_T(\alpha,y) w(y) K_S(\Psi_X \circ F(y),\beta) \mathrm{d}y,
\end{split}
\end{equation}
where $D(x,r)$ denotes the ball of center $x$ and radius $r$. From the second point \eqref{eq:pseudo_local} in the definition of a real-analytic FBI transform, we see that if $X \in B_{\epsilon,\delta}$ with $\delta$ small enough, then the third integral in \eqref{eq:split_kernel_Koopman} decays exponentially fast with $\alpha$ and $\beta$ (because if $F(y)$ is away from $\beta_x$ and $\delta$ is small enough then $\Psi_X \circ F(y)$ will also be away from $\beta_x$, and it will also be close to $M$).

Let us now deal with the second integral in \eqref{eq:split_kernel_Koopman} (the first one is dealt with similarly taking $F$ to be the identity and $X$ to be zero). From the definition of a real-analytic FBI transform, we see that for $y$ on a small complex neighbourhood of $D(F^{-1} \beta_x,s/100)$, the kernel $K_T(\alpha,y)$ is an $\mathcal{O}\p{\exp\p{- \brac{\va{\alpha}}/C}}$ and $K_S(\Psi_X \circ F(y),\beta)$ is given by $e^{i \Phi_S(\Psi_X \circ F(y),\beta)} b(\Psi_X \circ F(y),\beta)$ up to a $\mathcal{O}\p{\exp\p{- \brac{\va{\beta}}/C}}$. Here, the phase $\Phi_S$ is given by $\Phi_S(y,\beta) = - \overline{\Phi_T(\bar{\beta},\bar{y})}$ and the symbol $b(y,\beta)$ is given by $b(y,\beta) = \overline{a(\bar{\beta},\bar{y})}$. Hence, up to a negligible term, the second integral in \eqref{eq:split_kernel_Koopman} is given by
\begin{equation}\label{eq:simplified_kernel}
\int_{D(F^{-1} \beta_x,s/100)} e^{i \Phi_S(F (y),\beta)} h_{\alpha,\beta}(y) \mathrm{d}y,
\end{equation}
where
\begin{equation*}
h_{\alpha,\beta}(y) \coloneqq w(y) b(\Psi_X \circ F(y),\beta) K_T(\alpha,y) e^{i (\Phi_S(\Psi_X \circ F (y),\beta) - \Phi_S(F(y),\beta))}.
\end{equation*}
Notice that the function $h_{\alpha,\beta}$ has a holomorphic extension to a small complex neighbourhood of $D(\beta_x,s/100)$, and that this extension is bounded by
\begin{equation*}
C \n{w}_{\mathcal{O}_\epsilon} \exp\p{ - \gamma \brac{\alpha} + \gamma \brac{\beta} },
\end{equation*}
where the constant $\gamma > 0$ can be made arbitrarily small by taking $\delta$ small enough. Notice also that the phase in \eqref{eq:simplified_kernel} is non-stationary when $\beta_\xi$ is large:
\begin{equation*}
\mathrm{d}_y\p{\Phi_S (F(y),\beta)}_{| y = F^{-1} \beta_x} = {}^t D_{F^{-1} \beta_x} F \cdot \beta_\xi.
\end{equation*}
Consequently, if $s$ is small enough, then the norm of $\mathrm{d}_y\p{\Phi_S (F(y),\beta)}$ will be larger than $C^{-1} \brac{\beta}$ for every $y \in D(F^{-1} \beta_x,s/100)$. We can then use the holomorphic non-stationary phase \cite[Proposition 1.1]{BJ20} method with phase $\Phi_S(F(y),\beta)/ \brac{\beta}$ and large parameter $\brac{\beta}$ (notice that the phase is positive on the boundary of $D(F^{-1} \beta_x,s/100)$ as a consequence of \eqref{eq:coercivity_phase}) to get that \eqref{eq:simplified_kernel} is bounded by
\begin{equation*}
C \n{w}_{\mathcal{O}_\epsilon} \exp\p{ - \gamma \brac{\alpha} + \gamma \brac{\beta} - \tau \brac{\beta}},
\end{equation*}
for some $\tau$ that does not depend on $\gamma$. Taking $\delta > 0$ small enough, we get $\gamma < \tau$, and we find that the second integral in \eqref{eq:split_kernel_Koopman} is exponentially small in $\brac{\alpha}$ and $\brac{\beta}$. 

Let us now deal with the case of $\alpha$ and $\beta$ such that the distance between $\alpha_x$ and $F^{-1} \beta_x$ is less than $s$. We can consequently work in coordinates, and in these coordinates the distance between  $\alpha_\xi$ and ${}^t D_{F^{-1} \beta_x} F \cdot \beta_\xi$ is much larger than $s\p{\brac{\alpha} + \brac{\beta}}$ when $s$ is much smaller than $c$. As above, we can write the integral \eqref{eq:def_kernel_Koopman}, up to a negligible term, as
\begin{equation*}
\int_{D(\alpha_x,100s)} e^{i (\Phi_T(\alpha,y) + \Phi_S(F(y),\beta))} h_{\alpha,\beta}(y) \mathrm{d}y
\end{equation*}
where
\begin{equation*}
h_{\alpha,\beta}(y) = a(\alpha,y) w(y) b(\Psi_X \circ F(y),\beta) e^{i (\Phi_S(\Psi_X \circ F(y),\beta) - \Phi_S(F(y),\beta)}.
\end{equation*}
As above, we see that $h_{\alpha,\beta}$ as a holomorphic extension to a complex neighbourhood of $D(\alpha_x,s/100)$, and that this extension is bounded by
\begin{equation*}
C \brac{\alpha}^{\frac{d}{4}} \n{w}_{\mathcal{O}_\epsilon} \exp\p{\gamma \brac{\beta}},
\end{equation*}
where $\gamma > 0$ can be made arbitrarily small by taking $\delta > 0$ small enough. Provided $s$ is small enough, the phase $y \mapsto \Phi_T(\alpha,y) + \Phi_S(F(y),\beta)$ is non-stationary on $D(\alpha_x,100s)$ as
\begin{equation*}
\mathrm{d}_y\p{\Phi_T(\alpha,y) + \Phi_S(F(y),\beta)} = - \alpha_\xi + {}^t D_{F^{-1} \beta_x}F \cdot \beta_\xi + \mathcal{O}\p{s\p{\brac{\alpha} + \brac{\beta}}},
\end{equation*}
and the distance between $\alpha_\xi$ and ${}^t D_{F^{-1} \beta_x} F \cdot \beta_\xi$ is much larger than $s\p{\brac{\alpha} + \brac{\beta}}$. The result then follows from the holomorphic non-stationary phase method using a rescaling argument as in the previous case.
\end{proof}

We state now a result that allows to deal both with change of variables and multiplication by $C^\infty$ (\emph{a priori} not analytic) functions when working with FBI transforms.

\begin{lemma}\label{lemma:chgt_manifold}
Let $M_1$ and $M_2$ be closed real-analytic manifolds endowed with real-analytic FBI transform $T_1,T_2$ with adjoint $S_1,S_2$. Let $U_1,U_2$ be open subsets respectively of $M_1$ and $M_2$ and $\kappa : U_1 \to U_2$ be a real-analytic diffeomorphism. Let $\theta$ be a $C^\infty$ function on $M_1$ supported in $U_1$. Let $W$ be a closed subset of $M_1$ such that $\theta$ is real-analytic on a neighbourhood of $W$. Then, for every $c,N > 0$, there is a constant $C > 0$ such that for every $\alpha \in T^* M_1$ and $\beta \in T^* M_2$ such that the Kohn--Nirenberg distance between ${}^t D \kappa^{-1} \alpha$ and $\beta$ is more than $c$ then the kernel of $T_1 \theta \kappa^* S_2$ satisfies
\begin{itemize}
\item if $\alpha_x \in W$ and $\beta_x \in \kappa(W)$ then
\begin{equation*}
\va{T_1 \theta \kappa^* S_2(\alpha,\beta)} \leq C \exp\p{- \frac{\brac{\alpha} + \brac{\beta}}{C}};
\end{equation*}
\item if $\alpha_x \in W$ then
\begin{equation*}
\va{T_1 \theta \kappa^* S_2(\alpha,\beta)} \leq C \exp\p{- \frac{\brac{\alpha}}{C}} \brac{\beta}^{-N};
\end{equation*}
\item if $\beta_x \in \kappa(W)$ then
\begin{equation*}
\va{T_1 \theta \kappa^* S_2(\alpha,\beta)} \leq C \brac{\alpha}^{-N} \exp\p{- \frac{\brac{\beta}}{C}};
\end{equation*}
\item in general
\begin{equation*}
\va{T_1 \theta \kappa^* S_2(\alpha,\beta)} \leq C \p{\brac{\alpha} + \brac{\beta}}^{-N}.
\end{equation*}
\end{itemize}
\end{lemma}

\begin{remark}
The kernel of the operator $T_1 \theta \kappa^* S_2$ in Lemma \ref{lemma:chgt_manifold} does not fit exactly in the framework from Remark \ref{remark:meaning_kernel} as $T_1$ and $S_2$ are associated to different manifolds. However, the kernel is defined similary as
\begin{equation*}
T_1 \theta \kappa^* S_2 (\alpha,\beta) = \int_{M_1} K_{T_1}(\alpha,y)\theta(y) K_{S_2}(\kappa(y), \beta) \mathrm{d}y,
\end{equation*}
for $\alpha \in T^* M_1$ and $\beta \in T^* M_2$. One can adapt \eqref{eq:usage_kernel} similarly.
\end{remark}

\begin{proof}[Proof of Lemma \ref{lemma:chgt_manifold}]
In the first case, one may write
\begin{equation*}
\begin{split}
T_1 \theta \kappa^* S_2(\alpha,\beta) & = \int_{M_1} K_{T_1}(\alpha,y) \theta(y) K_{S_2}(\kappa y, \beta) \mathrm{d}y \\
     & = \int_{V} K_{T_1}(\alpha,y) \theta(y) K_{S_2}(\kappa y, \beta) \mathrm{d}y \\ & \qquad \qquad \qquad + \int_{M_1 \setminus V} K_{T_1}(\alpha,y) \theta(y) K_{S_2}(\kappa y, \beta) \mathrm{d}y,
\end{split}
\end{equation*}
where $V$ is a neighbourhood of $W$ on which $\theta$ is analytic. The integral on $V$ may be dealt as in the proof of Proposition \ref{proposition:localisation_graph} (this is actually slightly simpler since there is no $\Psi_X$ involved). The second integral is negligible since both kernels are exponentially small on $M_1 \setminus V$ when $\alpha_x$ and $\beta_x$ are in $W$.

Let us now consider the second case. We only need to deal with the case in which the distance between $\kappa^{-1} \beta_x$ and $W$ (and thus $\alpha_x$) is more than some small constant $s > 0$ (since otherwise the first case would apply). As in the proof of Proposition \ref{proposition:localisation_graph}, we split the integral defining the kernel of $T_1 \theta \kappa^* S_2$ into three:
\begin{equation}\label{eq:split_chgt_manifold}
\begin{split}
& T_1 \theta \kappa^* S_2(\alpha,\beta) \\  & \qquad = \p{\int_{D(\alpha_x,s/100)} + \int_{D(\kappa^{-1} \beta_x,s/100)} + \int_{M \setminus (D(\alpha_x,s/100) \cup D(\kappa^{-1} \beta_x,s/100)}} \\ & \qquad \qquad \qquad \qquad \qquad \qquad \qquad \qquad \qquad \qquad \quad K_{T_1}(\alpha,y) \theta(y) K_{S_2}(\kappa (y), \beta) \mathrm{d}y.
\end{split}
\end{equation}
The first and the last integrals may be dealt with as in the proof of Proposition \ref{proposition:localisation_graph}, and thus they are actually exponentially decaying with $\alpha$ and $\beta$. In the last integrals, both kernels $K_{T_1}$ and $K_{S_2}$ are exponentially small, and in the first integral the integrand is analytic, so that it can be dealt with using the holomorphic non-stationary phase method as for the first and second integral in \eqref{eq:split_kernel_Koopman} in the proof of Proposition \ref{proposition:localisation_graph}. Up to an exponentially decaying term, the second integral in \eqref{eq:split_chgt_manifold} is given by 
\begin{equation}\label{eq:simplified_mixed}
\int_{D(\kappa^{-1} \beta_x,s/100)} e^{i \Phi_{S_2}(\kappa( y),\beta)} \underbrace{b_2(\kappa (y), \beta) \theta(y) K_{T_1}(\alpha,y)}_{= h_{\alpha,\beta}(y)} \mathrm{d}y.
\end{equation}
Here the indices $1$ and $2$ are used to denote objects associated respectively to the manifolds $M_1$ and $M_2$. Since the imaginary part of $\Phi_{S_2}(\kappa( y),\beta)$ is positive when $y$ is away from $\kappa^{-1} \beta_x$, we may mutliply $h_{\alpha,\beta}$ by a $C^\infty$ bump function and assume that $h_{\alpha,\beta}$ is compactly supported in the interior of $D(\kappa^{-1} \beta_x,s/100)$. Using that $K_{T_1}$ is negligible away from the diagonal, we see that there are constants $\gamma,L > 0$ such that, for every $k \in \mathbb{N}$, there is a constant $C > 0$ such that the $C^k$ norm of $h_{\alpha,\beta}$ is less than
\begin{equation*}
C \brac{\beta}^L \exp\p{ - \gamma \brac{\alpha}}.
\end{equation*}
As in the proof of Proposition \ref{proposition:localisation_graph}, when $\beta$ is large the phase $y\mapsto \Phi_{S_2}(\kappa (y), \beta)$ is non-stationary, and we can consequently use the $C^\infty$ non-stationary phase method \cite[Theorem 7.7.1]{hormander_book_1} to find, using the $C^k$ bound on $h_{\alpha,\beta}$, that the integral \eqref{eq:simplified_mixed} is a $\mathcal{O}\p{\brac{\beta}^{-N} \exp\p{- \gamma \brac{\alpha}}}$ for every $N > 0$.

The third case is similar to the second, with the roles of $\alpha$ and $\beta$ swapped.

The fourth case follows the lines of the proof of Proposition \ref{proposition:localisation_graph} in the case $X = 0$, with the application of the holomorphic non-stationary phase method replaced by an application of the $C^\infty$ non-stationary phase method \cite[Theorem 7.7.1]{hormander_book_1}.
\end{proof}

\begin{remark}\label{remark:TS_away_diagonal}
Taking $M_1 = M_2 = M, \kappa$ the identity and $\theta$ identically equal to $1$ in Lemma \ref{lemma:chgt_manifold}, we retrieve a consequence of \cite[Lemma 2.9]{BJ20}: the kernel of the operator $TS$ is exponentially small away from the diagonal (we are always in the first case in Lemma \ref{lemma:chgt_manifold}).
\end{remark}

\begin{remark}\label{remark:small_to_analytic}
In the proof of Proposition \ref{proposition:reconstruction} below, we are going to make extensive use of the following fact: if $\mathfrak{K}$ is a measurable function on $T^* M \times T^*M$ such that
\begin{equation}\label{eq:meaning_exponentially_small}
\mathfrak{K}(\alpha,\beta) = \mathcal{O}\p{\exp\p{- \frac{\brac{\alpha} + \brac{\beta}}{C}}}
\end{equation}
for some $C > 0$, then the function
\begin{equation}\label{eq:explicit_analytic_kernel}
y \mapsto \int_{T^* M \times T^* M} K_S(y,\alpha) \mathfrak{K}(\alpha,\beta) K_T(\beta,y) \mathrm{d}\alpha \mathrm{d}\beta
\end{equation}
is real-analytic on $M$ (see for instance \cite[Lemmas 2.5 and 2.6]{BJ20}). Here, $K_T$ and $K_S$ are the kernels of $T$ and $S$ from \S \ref{subsection:FBI}. In terms of operators, the function \eqref{eq:explicit_analytic_kernel} is the kernel of the operator $S \mathfrak{K} T$, where we identify $\mathfrak{K}$ with the operator on $T^* M$ with kernel $\mathfrak{K}$.

In the proof of Proposition \ref{proposition:reconstruction}, we are going to combine this observation with Remark \ref{remark:meaning_kernel}.
\end{remark}

\begin{proof}[Proof of Proposition \ref{proposition:reconstruction}]
It follows from Lemmas \ref{lemma:ekomega_analytic} and \ref{lemma:dual_ekomega} that if $u$ is $C^\omega$ then the sum $\sum_{\substack{\omega \in \Omega \\ k \in \mathbb{Z}^d}} \langle u , e_k^\omega \rangle_{L^2} \tilde{e}_k^\omega$ converges in $C^\omega$. For $\omega \in \Omega$ and $k \in \mathbb{Z}^d$, we have, since $A_\omega$ and $B_\omega$ are self-adjoint and $\kappa_\omega$ has Jacobian $1$:
\begin{equation*}
\langle u , e_k^\omega \rangle_{L^2} = \langle B_\omega \theta_\omega (\kappa_\omega^{-1})^* A_\omega u, e_k \rangle_{L^2}.
\end{equation*}
For $\omega \in \Omega$, since $B_\omega \theta_\omega (\kappa_\omega^{-1})^* A_\omega u$ is $C^\omega$ (it follows from Propositions \ref{proposition:decay_from_analytic} and \ref{proposition:analytic_from_decay} and Lemma \ref{lemma:chgt_manifold}), we can write it as the sum of its Fourier series
\begin{equation*}
B_\omega \theta_\omega (\kappa_\omega^{-1})^* A_\omega u = \sum_{k \in \mathbb{Z}^d} \langle u , e_k^\omega \rangle_{L^2} e_k.
\end{equation*}
Since the operator $\widetilde{A}_\omega (\kappa_\omega)^* \theta_\omega B_\omega $ is bounded on $C^\omega$ (this is also a consequence of Propositions \ref{proposition:decay_from_analytic} and \ref{proposition:analytic_from_decay} and Lemma \ref{lemma:chgt_manifold}), we get
\begin{equation*}
\widetilde{A}_\omega (\kappa_\omega)^* \theta_\omega B_\omega B_\omega \theta_\omega (\kappa_\omega^{-1})^* A_\omega u  =  \sum_{k \in \mathbb{Z}^d} \langle u , e_k^\omega \rangle_{L^2} \tilde{e}_k^\omega.
\end{equation*}
Since $(\chi_\omega)_{\omega \in \Omega}$ is a partition of unity and $S(Tu) = u$, we have
\begin{equation*}
u = \sum_{\substack{\omega \in \Omega \\ k \in \mathbb{Z}^d}} \langle u , e_k^\omega \rangle_{L^2} \tilde{e}_k^\omega + \underbrace{\sum_{\omega \in \Omega} \p{A_\omega - \widetilde{A}_\omega (\kappa_\omega)^* \theta_\omega B_\omega B_\omega \theta_\omega (\kappa_\omega^{-1})^* A_\omega}u}_{= \mathcal{K}u}.
\end{equation*}
Thus, we only need to prove that for $\omega \in \Omega$ the operator 
\begin{equation}\label{eq:analytic_whole}
A_\omega - \widetilde{A}_\omega (\kappa_\omega)^* \theta_\omega B_\omega B_\omega \theta_\omega (\kappa_\omega^{-1})^* A_\omega
\end{equation}
has a real-analytic kernel.

We start by proving that the operator
\begin{equation}\label{eq:analytic_1}
\begin{split}
\widetilde{A}_\omega (\kappa_\omega)^* \theta_\omega (I - B_\omega) & B_\omega \theta_\omega (\kappa_\omega^{-1})^* A_\omega \\
     & = S \tilde{\chi}_\omega T (\kappa_\omega)^*\theta_\omega \mathcal{S} (1 - \rho_\omega) \mathcal{T} \mathcal{S} \rho_\omega \mathcal{T} \theta_\omega ( \kappa_\omega^{-1})^* S \chi_\omega T 
\end{split}
\end{equation}
has a real-analytic kernel. Recall that $\theta_\omega \equiv 1$ on a neighbourhood of the supports of $\rho_\omega$ and $\chi_\omega \circ \kappa_\omega^{-1}$. Consequently, it follows from Lemma \ref{lemma:chgt_manifold} that the kernel $\rho_\omega \mathcal{T} \theta_\omega ( \kappa_\omega^{-1})^* S \chi_\omega(\alpha,\beta)$ of $\rho_\omega \mathcal{T} \theta_\omega ( \kappa_\omega^{-1})^* S \chi_\omega$ is exponentially small (it satisfies \eqref{eq:meaning_exponentially_small}) when $\alpha$ is away from ${}^t D \kappa_\omega^{-1} \beta$. Similarly, see Remark \ref{remark:TS_away_diagonal}, the kernel of $\mathcal{T}\mathcal{S}$ is exponentially small away from the diagonal of $T^* \mathbb{T}^d \times T^* \mathbb{T}^d$. Writing the kernel of $\mathcal{T} \mathcal{S} \rho_\omega \mathcal{T} \theta_\omega ( \kappa_\omega^{-1})^* S \chi_\omega$ as
\begin{equation*}
\mathcal{T} \mathcal{S} \rho_\omega \mathcal{T} \theta_\omega ( \kappa_\omega^{-1})^* S \chi_\omega(\alpha,\beta) = \int_{T^* \mathbb{T}^d} \mathcal{T}\mathcal{S}(\alpha,\gamma) \rho_\omega \mathcal{T} \theta_\omega ( \kappa_\omega^{-1})^* S \chi_\omega(\gamma,\beta) \mathrm{d}\gamma, 
\end{equation*}
we find that $\mathcal{T} \mathcal{S} \rho_\omega \mathcal{T} \theta_\omega ( \kappa_\omega^{-1})^* S \chi_\omega(\alpha,\beta)$ is exponentially small when $\alpha$ is away from ${}^t D \kappa_\omega^{-1} \beta$. Since $\rho_\omega \equiv 1$ on a neighbourhood of the support of $\chi_\omega \circ \kappa_\omega^{-1}$, we find that the kernel of $\sqrt{1 - \rho_\omega}\mathcal{T} \mathcal{S} \rho_\omega \mathcal{T} \theta_\omega ( \kappa_\omega^{-1})^* S \chi_\omega$ is exponentially small (it satisfies \eqref{eq:meaning_exponentially_small} everywhere). We study now the kernel of $\tilde{\chi}_\omega T (\kappa_\omega)^*\theta_\omega \mathcal{S} \sqrt{1 - \rho_\omega}$. Since $\theta_\omega \equiv 1$ on a neighbourhood of the support of $\tilde{\chi}_\omega \circ \kappa_\omega^{-1}$, we can apply the second point in Lemma \ref{lemma:chgt_manifold} to bound the kernel of  $\tilde{\chi}_\omega T (\kappa_\omega)^*\theta_\omega \mathcal{S} \sqrt{1 - \rho_\omega}$. Using in addition that $\rho_\omega \equiv 1$ on a neighbourhood of the support of $\tilde{\chi}_\omega \circ \kappa_\omega^{-1}$, we find that for every $N > 0$ there is a $C > 0 $ such that the kernel $\tilde{\chi}_\omega T (\kappa_\omega)^*\theta_\omega \mathcal{S} \sqrt{1 - \rho_\omega}(\alpha,\beta)$ of $\tilde{\chi}_\omega T (\kappa_\omega)^*\theta_\omega \mathcal{S} \sqrt{1 - \rho_\omega}$ is bounded by the quantity $C\exp(- \brac{\alpha}/C) \brac{\beta}^{-N}$. Writing
\begin{equation*}
\begin{split}
\tilde{\chi}_\omega & T (\kappa_\omega)^*\theta_\omega \mathcal{S} (1 - \rho_\omega) \mathcal{T} \mathcal{S} \rho_\omega \mathcal{T} \theta_\omega ( \kappa_\omega^{-1})^* S \chi_\omega(\alpha,\beta) \\ & = \int_{T^* \mathbb{T}^d} \tilde{\chi}_\omega T (\kappa_\omega)^*\theta_\omega \mathcal{S} \sqrt{1 - \rho_\omega}(\alpha,\gamma) \sqrt{1 - \rho_\omega}\mathcal{T} \mathcal{S} \rho_\omega \mathcal{T} \theta_\omega ( \kappa_\omega^{-1})^* S \chi_\omega(\gamma,\beta) \mathrm{d}\gamma
\end{split}
\end{equation*}
and using the bound we just proved, we find that the kernel $\tilde{\chi}_\omega T (\kappa_\omega)^*\theta_\omega \mathcal{S} (1 - \rho_\omega) \mathcal{T} \mathcal{S} \rho_\omega \mathcal{T} \theta_\omega ( \kappa_\omega^{-1})^* S \chi_\omega(\alpha,\beta)$ is exponentially small. It follows then from Remark \ref{remark:small_to_analytic} that the operator \eqref{eq:analytic_1} is analytic.

We prove now that the operator
\begin{equation}\label{eq:analytic_2}
\begin{split}
\widetilde{A}_\omega (\kappa_\omega)^* \theta_\omega (I - B_\omega) & \theta_\omega (\kappa_\omega^{-1})^* A_\omega \\
     & = S \tilde{\chi}_\omega T (\kappa_\omega)^*\theta_\omega \mathcal{S} (1 - \rho_\omega) \mathcal{T} \theta_\omega ( \kappa_\omega^{-1})^* S \chi_\omega T 
\end{split}
\end{equation}
has a real-analytic kernel. As we proved above that, for every $N > 0$ there is a $C > 0$ such that the kernel $\tilde{\chi}_\omega T (\kappa_\omega)^*\theta_\omega \mathcal{S} \sqrt{1 - \rho_\omega}(\alpha,\beta)$ is bounded by $C\exp(- \brac{\alpha}/C) \brac{\beta}^{-N}$, we find that the kernel $\sqrt{1 - \rho_\omega} \mathcal{T} \theta_\omega ( \kappa_\omega^{-1})^* S \chi_\omega(\alpha,\beta)$ is a $\mathcal{O}(\brac{\alpha}^{-N} \exp(- \brac{\beta}/C))$. Writing
\begin{equation*}
\begin{split}
\tilde{\chi}_\omega & T (\kappa_\omega)^*\theta_\omega \mathcal{S} (1 - \rho_\omega) \mathcal{T} \theta_\omega ( \kappa_\omega^{-1})^* S \chi_\omega(\alpha,\beta)\\ & = \int_{T^* \mathbb{T}^d} \tilde{\chi}_\omega T (\kappa_\omega)^*\theta_\omega \mathcal{S} \sqrt{1 - \rho_\omega}(\alpha,\gamma)\sqrt{1 - \rho_\omega} \mathcal{T} \theta_\omega ( \kappa_\omega^{-1})^* S \chi_\omega(\gamma,\beta) \mathrm{d}\gamma,
\end{split}
\end{equation*}
we find that the kernel of $\tilde{\chi}_\omega T (\kappa_\omega)^*\theta_\omega \mathcal{S} (1 - \rho_\omega) \mathcal{T} \theta_\omega ( \kappa_\omega^{-1})^* S \chi_\omega$ is exponentially small. It follows then from Remark \ref{remark:small_to_analytic} that \eqref{eq:analytic_2} has a real-analytic kernel.

We keep going and prove that the operator
\begin{equation}\label{eq:analytic_3}
\widetilde{A}_\omega A_\omega - \widetilde{A}_\omega (\kappa_\omega)^* \theta_\omega \theta_\omega (\kappa_\omega^{-1})^* A_\omega = S \tilde{\chi}_\omega T (1 - \theta_\omega^2 \circ \kappa_\omega ) S \chi_\omega T
\end{equation}
has a real-analytic kernel. To do so, we write for $\alpha,\beta \in T^* M$
\begin{equation}\label{eq:integral_petite}
\tilde{\chi}_\omega T (1 - \theta_\omega^2 \circ \kappa_\omega ) S \chi_\omega = \int_{M} \tilde{\chi}_\omega(\alpha_x) \chi_{\omega}(\beta_x) (1 - \theta_\omega^2 ( \kappa_\omega(y)) ) K_T(\alpha,y) K_S(y,\beta) \mathrm{d}y.
\end{equation}
Since $\theta_\omega \circ \kappa_\omega \equiv 1$ on a neighbourhood of the supports of $\tilde{\chi}_\omega$ and $\chi_\omega$, and the kernels $K_T(\alpha,y)$ and $K_S(y,\beta)$ are exponentially small when $y$ is away respectively from $\alpha_x$ and $\beta_x$, we find that the integral in \eqref{eq:integral_petite} satisfies the estimate \eqref{eq:meaning_exponentially_small}. It follows then from Remark \ref{remark:small_to_analytic} that the operator \eqref{eq:analytic_3} has a real-analytic kernel.

Finally, we prove that the operator
\begin{equation}\label{eq:analytic_4}
A_\omega - \widetilde{A}_\omega A_\omega = S (1 - \tilde{\chi}_\omega)TS\chi_\omega T
\end{equation}
has a real-analytic kernel. Since $\tilde{\chi}_\omega \equiv 1$ on the support of $\chi_\omega$, it follows from Remark \ref{remark:TS_away_diagonal} that the kernel of $(1 - \tilde{\chi}_\omega)TS\chi_\omega$ is exponentially small (it satisfies \eqref{eq:meaning_exponentially_small}), and thus the operator \eqref{eq:analytic_4} has a real-analytic kernel according to Remark \ref{remark:small_to_analytic}.

Summing the operators \eqref{eq:analytic_1}, \eqref{eq:analytic_2}, \eqref{eq:analytic_3} and \eqref{eq:analytic_4}, we retrieve the operator \eqref{eq:analytic_whole}, which consequently has a real-analytic kernel.
\end{proof}

\section{Gevrey case}\label{section:Gevrey}

\subsection{Main results}

Let us explain how the analysis above can be applied in the Gevrey case. Let $\sigma \geq 1$. We recall that a $C^\infty$ function $f$ from an open subset of $\mathbb{R}^d$ to $\mathbb{C}$ is $\sigma$-Gevrey if for every compact subset $K$ of $U$, there are constants $C,R > 0$ such that for every $\alpha \in \mathbb{N}^d$ and $x \in K$, we have
\begin{equation*}
\va{\partial^\alpha f(x)} \leq C R^{\va{\alpha}} \alpha!^\sigma.
\end{equation*}
Notice that $1$-Gevrey functions are exactly real-analytic functions. When $\sigma > 1$ however, there are $\sigma$-Gevrey compactly supported functions. 

From now on, we assume that $\sigma > 1$. We say that a function valued in $\mathbb{R}^m$ is $\sigma$-Gevrey if its components are. With this definition, $\sigma$-Gevrey mappings are stable under composition \cite{GevreyOG}. We can then define a $\sigma$-Gevrey manifold: we just modify the usual definition of $C^\infty$ or $C^\omega$ manifold by asking for $\sigma$-Gevrey (instead of $C^\infty$ or $C^\omega$) change of charts in the atlas defining the $\sigma$-Gevrey structure. Since there is a $\sigma$-Gevrey version of the implicit function theorem (see for instance \cite{komatsu_implicit}), most of the basic differential geometry carries to the $\sigma$-Gevrey setting. Notice that a $C^\omega$ manifold has a natural structure of $\sigma$-Gevrey manifold, since real-analytic functions are $\sigma$-Gevrey. Reciprocally, if $M$ is a compact $\sigma$-Gevrey manifold, one can endow $M$ with a structure of $C^\omega$ manifold that is coherent with its $\sigma$-Gevrey structure. We say that a function on a $\sigma$-Gevrey manifold, or a map between $\sigma$-Gevrey manifolds, is $\sigma$-Gevrey if it is $\sigma$-Gevrey in $\sigma$-Gevrey coordinates.

The Gevrey analogue of Theorem \ref{thm:upper_bound_determinant} is then:

\begin{thm}\label{theorem:gevrey_determinant}
Let $\sigma > 1$. Let $M$ be a closed $\sigma$-Gevrey manifold and $F : M \to M$ be a $\sigma$-Gevrey Anosov diffeomorphism. If $w :M \to \mathbb{C}$ is a $\sigma$-Gevrey function, then there is a constant $C > 0$ such that for every $z \in \mathbb{C}$, we have
\begin{equation*}
\va{d_{F,w}(z)} \leq C\exp\p{C\p{\log(1 + \va{z})}^{\sigma d+1}}.
\end{equation*}
\end{thm}

As a corollary, we have

\begin{corollary}\label{corollary:upper_bound_resonances_gevrey}
Under the assumption of Theorem \ref{theorem:gevrey_determinant}, the number $N_{F,w}(r)$ of Ruelle resonances of $\mathcal{L}_{F,w}$ of modulus more than $r$ satisfy the asymptotic bound
\begin{equation*}
N_{F,w}(r) \underset{r \to 0}{=} \mathcal{O}\p{\va{\log r}^{\sigma d}}.
\end{equation*}
\end{corollary}

\begin{remark}
Notice that Theorem \ref{theorem:gevrey_determinant} and Corollary \ref{corollary:upper_bound_resonances_gevrey} are quantitative improvements over the bounds given in \cite[Theorem 2.12]{local_and_global}.
\end{remark}

Theorem \ref{theorem:gevrey_determinant} will follow from a Gevrey analogue of Theorem \ref{theorem:anisotropic_space} (Theorem \ref{theorem:anisotropic_space_gevrey} below). To state this result, we will need to define the Gevrey analogue of the spaces $\mathcal{O}_\epsilon$. To do so, let $\sigma > 1$ and $M$ be a closed $\sigma$-Gevrey manifold of dimension $d$, and cover $M$ by a finite family $(K_\lambda)_{\lambda \in \Lambda}$ of compact subsets such that for every $\lambda \in \Lambda$, the set $K_\lambda$ is contained in the domain $O_\lambda$ of a $\sigma$-Gevrey chart $\psi_\lambda : O_\lambda \to \mathfrak{O}_\lambda$. If $f : M \to \mathbb{C}$ is a $C^\infty$ function and $\epsilon > 0$, we define the norm
\begin{equation*}
\n{f}_{\sigma,\epsilon} \coloneqq \sup_{\lambda \in \Lambda} \sup_{x \in K_\lambda} \sup_{\alpha \in  \mathbb{N}^d} \frac{\epsilon^{\va{\alpha}}\va{\partial^\alpha (f \circ \psi_\lambda^{-1})(\psi_\lambda x)}}{\alpha!^\sigma}.
\end{equation*}
We let then $\mathcal{G}^\sigma_\epsilon$ be the space of $C^\infty$ functions $f$ on $M$ such that $\n{f}_{\sigma,\epsilon} < \infty$. One easily checks that $\mathcal{G}^\sigma_\epsilon$ is a Banach space. As in the real-analytic case the space $\mathcal{O}_\epsilon$ depends on the choice of a real-analytic metric on $M$, the space $\mathcal{G}^\sigma_\epsilon$ depends on the particular choice of charts $(\psi_\lambda)_{\lambda \in \Lambda}$, but the union $\bigcup_{\epsilon > 0} \mathcal{G}^\sigma_\epsilon$ is the space of $\sigma$-Gevrey functions on $M$, and thus does not depend on our particular choices.

We can now state the Gevrey analogue of Theorem \ref{theorem:anisotropic_space}.

\begin{thm}\label{theorem:anisotropic_space_gevrey}
Let $\sigma > 1$. Let $M$ be a closed $\sigma$-Gevrey manifold. Let $F$ be a $\sigma$-Gevrey diffeomorphism on $M$. Let $\epsilon > 0$. There is a separable Hilbert space $\mathcal{H}$ with the following properties:
\begin{enumerate}[label=(\roman*)]
\item there are continuous injections with dense images $\iota$ and $j$ from $\mathcal{G}^\sigma_\epsilon$ respectively to $\mathcal{H}$ and to $\mathcal{H}^*$; \label{item:injections_gevrey}
\item for every $w \in \mathcal{G}^\sigma_\epsilon$, there is a compact operator $\widetilde{\mathcal{L}}_{F,w}$ from $\mathcal{H}$ to itself of exponential class of type $1/(\sigma d)$; \label{item:extension_operator_gevrey}
\item the map $w \mapsto \widetilde{\mathcal{L}}_{F,w}$ is a bounded linear operator from $\mathcal{G}^\sigma_\epsilon$ to the space of trace class operators on $\mathcal{H}$; \label{item:holomorphic_family_gevrey}
\item for every $w \in \mathcal{G}^\sigma_\epsilon, n \in \mathbb{N}$ and $u,v \in \mathcal{G}^\sigma_\epsilon$, we have
\begin{equation*}
j(u)\p{\widetilde{\mathcal{L}}^n_{F,w} \iota(v)} = \int_M u \p{\mathcal{L}_{F,w}^n v} \mathrm{d}x.
\end{equation*} \label{item:integrals_gevrey}
\end{enumerate}
\end{thm}

\begin{proof}[Proof of Theorem \ref{theorem:gevrey_determinant} and Corollary \ref{corollary:upper_bound_resonances_gevrey}]
Using the point \ref{item:integrals_gevrey} from Theorem \ref{theorem:anisotropic_space_gevrey} as in the proof of Proposition \ref{proposition:trace_formula}, one gets that 
\begin{equation*}
d_{F,w}(z) = \det(I - z\widetilde{\mathcal{L}}_{F,w}),
\end{equation*}
and the results then follow from Lemma \ref{lemma:order_exponential}.
\end{proof}

The proof of Theorem \ref{theorem:anisotropic_space_gevrey} follows roughly the same lines as the proof of Theorem \ref{theorem:anisotropic_space}, so that we will only outline the main differences. The proof is actually simpler due to the existence of compactly supported Gevrey functions. For the rest of this section, let us fix $\sigma > 1$, a closed $\sigma$-Gevrey manifold $M$ and $F$ be a $\sigma$-Gevrey Anosov diffeomorphism on $M$.

\subsection{Construction of the space}

As in the real-analytic case, our construction starts with the choice of an escape function. In order to take into account our Gevrey parameter, we replace \eqref{eq:definition_escape_function} by
\begin{equation*}
G(x,\xi) = \va{\xi_s}^{\frac{1}{\sigma}} - \va{\xi_u}^{\frac{1}{\sigma}}.
\end{equation*}

The crucial estimate in that case is that there are constants $ \varpi > 0$ and $C$ such that if $\alpha,\beta,\gamma \in T^* M$ are large enough and such that $d_{KN}(\alpha,\beta) \leq \varpi$ and $d_{KN}(\mathcal{F} \beta,\gamma) \leq \varpi$, then
\begin{equation*}
G(\gamma) - G(\alpha) \leq - C^{-1} \va{\beta}^{\frac{1}{\sigma}}.
\end{equation*}

We follow the exposition from \S \ref{section:construction}. We let the $(U_\omega)_{\omega \in \Omega}, (\kappa_\omega)_{\omega \in \Omega}, (\chi_\omega)_{\omega \in \Omega}$ and $(\tilde{\chi}_\omega)_{\omega}$ be as in \S \ref{section:construction}. The only difference is that we require that the $\kappa_\omega$'s, the $\chi_\omega$'s and the $\tilde{\chi}_\omega$'s are $\sigma$-Gevrey (instead of real-analytic or $C^\infty$). In that case, we can give a simpler definition of the $e_k^\omega$'s and the $\tilde{e}_k^\omega$'s. For $\omega \in \Omega$ and $k \in \mathbb{Z}^d$, we define
\begin{equation*}
e_k^\omega = \chi_\omega (\kappa_\omega)^* e_k \quad \textup{and} \quad \tilde{e}_k^\omega = \tilde{\chi}_\omega (\kappa_\omega)^* e_k.
\end{equation*}
This simpler definition is made possible by the existence of $\sigma$-Gevrey bump functions.

It would be possible to state an analogue of Lemma \ref{lemma:localisation}, but we will not need it. The analogues of Lemmas \ref{lemma:ekomega_analytic} and \ref{lemma:dual_ekomega} in this context are:

\begin{lemma}\label{lemma:ekomega_gevrey}
Let $\rho > 0$. Then there are $C,\epsilon > 0$ such that for every $\omega \in \Omega$ and $k \in \mathbb{Z}^d$ we have
\begin{equation*}
\n{e_k^\omega}_{\sigma,\epsilon} \leq C e^{\rho \va{k}^{\frac{1}{\sigma}}} \quad \textup{ and } \n{\tilde{e}_k^\omega}_{\sigma,\epsilon} \leq C e^{\rho \va{k}^{\frac{1}{\sigma}}}.
\end{equation*}
\end{lemma}

\begin{proof}
Let us start by estimating the derivatives of the functions $e_k$'s defined on the torus $\mathbb{T}^d$. If $k \in \mathbb{Z}^d$ and $\alpha \in \mathbb{N}^d$, then we have
\begin{equation*}
\begin{split}
\frac{\va{\partial^\alpha (e^{2 i \pi k \cdot x})}\epsilon^{\va{\alpha}}}{\alpha!^\sigma} & = \frac{(2 \pi \epsilon )^{\va{\alpha}}}{\alpha!^\sigma} |k^\alpha| \\
     & \leq \frac{(2 \pi \va{k} \epsilon)^{\va{\alpha}}}{\alpha!^\sigma}.
\end{split}
\end{equation*}
Using Stirling's formula to estimate $\alpha!^\sigma$, we find constants $C$ and $A$ that does not depend on $k$ nor $\alpha$ such that
\begin{equation*}
\frac{\va{\partial^\alpha (e^{2 i \pi k \cdot x})}\epsilon^{\va{\alpha}}}{\alpha!^\sigma} \leq C \frac{(A \va{k} \epsilon)^{\va{\alpha}}}{\va{\alpha}^{\sigma \va{\alpha}}}.
\end{equation*}
Then, we notice that for every $a > 0$ we have
\begin{equation*}
\sup_{x \in \mathbb{R}_+^*} a^x x^{- \sigma x} = \exp\p{\sigma a^{\frac{1}{\sigma}}/e},
\end{equation*}
and thus
\begin{equation*}
\frac{\va{\partial^\alpha (e^{2 i \pi k x})}\epsilon^{\va{\alpha}}}{\alpha!^\sigma} \leq C \exp\p{\frac{\sigma (A \va{k} \epsilon)^{\frac{1}{\sigma}}}{e}}.
\end{equation*}
To end the proof, we only need to notice that such estimates are preserved by multiplication by a Gevrey function and composition by a Gevrey mapping (up to making $C$ larger and the $\epsilon$ on the right hand side smaller). A proof of this fact may be found for instance in the original paper of Gevrey \cite{GevreyOG}.
\end{proof}

\begin{lemma}\label{lemma:dual_ekomega_gevrey}
Let $\epsilon > 0$. Then there are constants $C,\rho > 0$ such that for every $\omega \in \Omega, k \in \mathbb{Z}^d$ and $u \in \mathcal{G}^\sigma_\epsilon$ we have
\begin{equation*}
\va{\brac{u,e_k^\omega}_{L^2}} \leq C \n{u}_{\sigma,\epsilon} e^{- \rho \va{k}^{\frac{1}{\sigma}}} \quad \textup{and} \quad \va{\brac{u,\tilde{e}_k^\omega}_{L^2}} \leq C \n{u}_{\sigma,\epsilon} e^{- \rho \va{k}^{\frac{1}{\sigma}}}.
\end{equation*}
\end{lemma}

\begin{proof}
We will deal only with the case of $e_k^\omega$'s. The $\tilde{e}_k^\omega$ are dealt with similarly. Let us change variable and write
\begin{equation*}
\brac{u,e_k^\omega}_{L^2} =\int_{\mathbb{T}^d} e^{-2 i \pi k \cdot x} v(x) \mathrm{d}x,
\end{equation*}
where $v = (\chi_\omega u) \circ \kappa_\omega^{-1}$ is a $\sigma$-Gevrey function. Integrating by parts, we find that for every $L \in \mathbb{N}$, we have
\begin{equation*}
\begin{split}
\brac{u,e_k^\omega}_{L^2} & = \int_{\mathbb{T}^d} \frac{(I - \Delta)^L (e^{-2 i \pi k \cdot x})}{(1+ 4 \pi^2 |k|^2)^L} v(x) \mathrm{d}x \\
     & = \int_{\mathbb{T}^d} e^{-2 i \pi k \cdot x} \frac{(I - \Delta)^L v (x)}{(1+ 4 \pi^2 |k|^2)^L} \mathrm{d}x,
\end{split}
\end{equation*}
where $\Delta$ denotes the standard (non-positive) Laplace operator on $\mathbb{T}^d$. It follows that, for some constants $C,R > 0$ (that depend on $\epsilon$), we have
\begin{equation*}
\begin{split}
\va{\brac{u,e_k^\omega}_{L^2}} & \leq \frac{1}{(1+ 4 \pi^2 |k|^2)^L} \sup_{\mathbb{T}^d} \va{(I - \Delta)^L v} \\
     & \leq C \n{u}_{\sigma,\epsilon} \p{\frac{R}{1+ 4 \pi^2 |k|^2}}^{L} L^{2 \sigma L}.
\end{split}
\end{equation*}
Taking $L$ to be approximately $\p{\frac{1+ 4 \pi^2 \va{k}^2}{R}}^{\frac{1}{2\sigma}}/e$, we get the announced result.
\end{proof}

The analogue of Proposition \ref{proposition:reconstruction} is just the following direct consequence of Fourier inversion formula.

\begin{proposition}
If $u$ is a smooth function on $M$ then
\begin{equation*}
u = \sum_{\omega \in \Omega} \sum_{k \in \mathbb{Z}^d} \langle u, e_k^\omega \rangle_{L^2} \tilde{e}^\omega_k.
\end{equation*}
\end{proposition}

We define the $G_\omega$'s and the norms $\n{\cdot}_\gamma$'s by the same formulae as in \S \ref{section:construction} (using our new escape function). Using Lemma \ref{lemma:dual_ekomega_gevrey}, we immediately get the following analogue of Lemma \ref{lemma:finite_norm}.

\begin{lemma}\label{lemma:finite_norm_gevrey}
Let $\epsilon > 0$. There are constants $C,\gamma_0 > 0$ such that for every $0 < \gamma \leq \gamma_0$ and every $u \in \mathcal{G}_\epsilon^\sigma$ we have $\n {u}_{\gamma} \leq C\n{u}_{\sigma,\epsilon}$.
\end{lemma}

For $\gamma$ and $\epsilon$ as in Lemma \ref{lemma:finite_norm_gevrey}, we let $\mathcal{H}_{\gamma,\epsilon}$ denotes the completion of $\mathcal{G}_\epsilon^\sigma$ for the norm $\mathcal{H}_{\gamma,\epsilon}$, and we let $\iota : \mathcal{G}_\epsilon^\sigma \to \mathcal{H}_{\gamma,\epsilon}$ be the inclusion. The analogue of Lemma \ref{lemma:injection_dual} is the following consequence of Lemma \ref{lemma:dual_ekomega_gevrey}.

\begin{lemma}\label{lemma:dual_inclusion_gevrey}
Let $\epsilon > 0$. Then there are $C,\gamma_0 >0$ such that for every $0 < \gamma \leq \gamma_0$ if $v \in \mathcal{G}_\epsilon^s$ then the linear form
\begin{equation*}
u \mapsto \langle u, v \rangle_{L^2}
\end{equation*}
extends to a continuous linear form $l_v$ on $\mathcal{H}_{\gamma,\epsilon}$. Moreover, the map $j : v \mapsto l_{\bar{v}}$ is $\mathbb{C}$-linear and continuous from $\mathcal{G}_\epsilon^\sigma$ to the dual of $\mathcal{H}_{\gamma,\epsilon}$.
\end{lemma}

We also have an analogue of Lemma \ref{lemma:expression_norm}: for $\epsilon > 0$ and $\gamma > 0$, the norm of $\mathcal{H}_{\gamma,\epsilon}$ is given by the expression \eqref{eq:expression_norm}, and as a consequence the Hilbert space $\mathcal{H}_{\gamma,\epsilon}$ is separable, and the map $j$ has dense image.

With $\Gamma = \Omega \times \mathbb{Z}^d$, we define the relation $\hookrightarrow$ as in \S \ref{section:construction}, and we notice that there is a constant $C > 0$ such that if $(\omega,k) \hookrightarrow (\omega',k')$ and $k$ is large enough then
\begin{equation*}
G_{\omega}(k) \leq G_{\omega'}(k') - C \va{k}^{\frac{1}{\sigma}}.
\end{equation*}

The main point to prove Theorem \ref{theorem:anisotropic_space_gevrey} is to get an analogue of Lemma \ref{lemma:discard}.

\begin{lemma}\label{lemma:discard_gevrey}
Let $\epsilon > 0$. There are $C, \tau >0$ such that, if $w \in \mathcal{G}_\epsilon^s$, for every $(\omega,k),(\omega',k') \in \Gamma$ such that $(\omega,k) \not\hookrightarrow (\omega',k')$ we have
\begin{equation*}
\va{\langle \mathcal{L}_{F,w} \tilde{e}_{k}^{\omega}, e_{k'}^{\omega'} \rangle_{L^2}} \leq C \n{w}_{\sigma,\epsilon} e^{ - \tau \max(\va{k},\va{k'})^{\frac{1}{\sigma}}}.
\end{equation*}
\end{lemma}

\begin{proof}
Changing variables, we may write
\begin{equation}\label{eq:in_the_chart}
\langle \mathcal{L}_{F,w} \tilde{e}_{k}^{\omega}, e_{k'}^{\omega'} \rangle_{L^2} = \int_{\mathbb{T}^d} e^{i \Phi_{\omega,k}^{\omega',k'} (x)} h_{\omega,\omega'}(x) \mathrm{d}x,
\end{equation}
where
\begin{equation*}
h_{\omega,\omega'}(x) = w(\kappa_{\omega'}^{-1}(x))\tilde{\chi}_{\omega}(F(\kappa_{\omega'}^{-1} x)) \chi_\omega(\kappa_{\omega'}^{-1}x)
\end{equation*}
and
\begin{equation*}
\Phi_{\omega,k}^{\omega',k'}(x) = 2 \pi k \cdot \kappa_{\omega} \circ F \circ \kappa_{\omega'}^{-1}(x) - 2 \pi k' \cdot x.
\end{equation*}
Notice that the function $h_{\omega,\omega'}$ is $\sigma$-Gevrey, with Gevrey norm controlled by $\n{w}_{\sigma,\epsilon}$. Consequently, $h_{\omega,\omega'}$ has a $\sigma$-Gevrey pseudo-analytic extension $\tilde{h}_{\omega,\omega'}$. This is a smooth function on the complex neighbourhood $\mathbb{C}^d / \mathbb{Z}^d$ of $\mathbb{T}^d$ that coincides with $h_{\omega,\omega'}$ on $\mathbb{T}^d$ and such that $\bar{\partial} \tilde{h}_{\omega,\omega'}$ vanishes at infinite order on $\mathbb{T}^d$. Moreover, since $h_{\omega,\omega'}$ is $\sigma$-Gevrey, we may choose $\tilde{h}_{\omega,\omega'}$ as a $\sigma$-Gevrey function, with Gevrey norm controlled by $\n{w}_{\sigma,\epsilon}$. It implies in particular that the moduli of the components of $\bar{\partial} \tilde{h}_{\omega,\omega'}$ at a point $z \in \mathbb{C}^d / \mathbb{Z}^d$ are less than
\begin{equation}\label{eq:bound_dbar}
C \n{w}_{\sigma,\epsilon} \exp\p{ - C^{-1} \va{\im z}^{- \frac{1}{\sigma - 1}}},
\end{equation}
where the constant $C$ may depend on $\sigma$ and $\epsilon$ but not on $w$. Moreover, the sup norm of $\tilde{h}_{\omega,\omega'}$ is also controlled by $\n{w}_{\sigma,\epsilon}$. We can also assume that the support of $\tilde{h}_{\omega,\omega'}$ is contained in a small neighbourhood of the support of $h_{\omega,\omega'}$. Similarly, construct a $\sigma$-Gevrey pseudo-analytic extension $\widetilde{\Phi}_{\omega,k}^{\omega',k'}$ by choosing a $\sigma$-Gevrey pseudo-analytic extension for $\kappa_\omega \circ F \circ \kappa_{\omega'}^{-1}$ near the support of $h_{\omega,\omega'}$. One may refer to \cite[\S 1.1.1.3]{BJ20} for the details of this construction.

Let us now study the phase $\Phi_{\omega,k}^{\omega',k'}$. For $x \in \mathbb{T}^d$, the gradient of this phase is given by the formula
\begin{equation*}
\begin{split}
&\nabla \Phi_{\omega,k}^{\omega',k'}(x) \\ & \qquad = 2 \pi {}^t D \kappa_{\omega'}^{-1} (x)\p{{}^t D F (\kappa_{\omega'}^{-1} x) \cdot {}^t D \kappa_{\omega} (F(\kappa_{\omega'}^{-1} x))\cdot k - {}^t D \kappa_{\omega'}(\kappa_{\omega'}^{-1}x) \cdot k'}.
\end{split}
\end{equation*}
This quantity is ${}^t D \kappa_{\omega'}^{-1} (x)$ applied to the difference of a point of a point in $\mathcal{W}_{\omega',k'}$ and a point in $\mathcal{F}(\mathcal{W}_{\omega,k})$. Since these two points are in the same fiber of $T^* M$, it follows from our hypothesis $(\omega,k) \not\hookrightarrow (\omega',k')$ that
\begin{equation}\label{eq:lower_bound_phase}
\va{\nabla \Phi_{\omega,k}^{\omega',k'}} \geq C^{-1} \max(\va{k},\va{k'}),
\end{equation}
for some constant $C$ that does not depend on $\omega,\omega',k$ nor $k'$. 

Let then $\delta = \delta(k,k') = \max(|k|,|k'|)^{\frac{1 - \sigma}{\sigma}}$, and shift contour in \eqref{eq:in_the_chart}, replacing $x$ by $x + i \delta \overline{\nabla \Phi_{\omega,k}^{\omega',k'}}/ \va{\nabla \Phi_{\omega,k}^{\omega',k'}}$. We find that
\begin{equation}\label{eq:post_shift}
\begin{split}
& \langle \mathcal{L}_{F,g} \tilde{e}_{k'}^{\omega'}, e_k^\omega \rangle \\ & \hspace{0.1cm} = \int_{\mathbb{T}^d} e^{i \widetilde{\Phi}_{\omega,k}^{\omega',k'} (x+ i\delta \overline{\nabla \Phi_{\omega,k}^{\omega',k'}} /\va{\nabla \Phi_{\omega,k}^{\omega',k'}})} \tilde{h}_{\omega,\omega'}(x + i\delta \overline{\nabla \Phi_{\omega,k}^{\omega',k'}}/\va{\nabla \Phi_{\omega,k}^{\omega',k'}}) J_{\delta,\omega,k}^{\omega',k'} \mathrm{d}x \\ & \qquad \qquad \qquad \qquad \qquad \qquad \qquad \qquad \qquad \qquad \qquad \qquad \qquad \qquad \hspace{0.8cm}+ R_{\omega,k}^{\omega',k'}.
\end{split}
\end{equation}
Here, $J_{\delta,\omega,k}^{\omega',k'}$ is the Jacobian that appears when parametrizing the new domain of integration by $\mathbb{T}^d$, it is uniformly bounded. The remainder $R_{\omega,k}^{\omega',k'}$ is an integral involving $\bar{\partial} \tilde{h}_{\omega,\omega'}$ and $\bar{\partial} \tilde{\Phi}_{\omega,k}^{\omega',k'}$. It follows from Taylor's formula that the phase that appears in $R_{\omega,k}^{\omega',k'}$  has a non-negative imaginary part, provided $\max(|k|,|k'|)$ is large enough. Hence, it follows from the bound \eqref{eq:bound_dbar} that for some $C > 0$, we have
\begin{equation*}
\va{R_{\omega,k}^{\omega',k'}} \leq C \n{w}_ {\sigma,\epsilon} \exp\p{ - C^{-1} \delta^{- \frac{1}{\sigma - 1}}} = C \n{w}_{\sigma,\epsilon} \exp\p{ - C^{-1} \max(|k|,|k'|)^{\frac{1}{\sigma}}}.
\end{equation*}
To estimate the integral in \eqref{eq:post_shift}, we start by noticing that, when $\max(|k|,|k'|)$ is large (and thus $\delta$ is small), Taylor's formula and \eqref{eq:lower_bound_phase} give that
\begin{equation*}
\begin{split}
\im \widetilde{\Phi}_{\omega,k}^{\omega',k'} (x+ i\delta \overline{\nabla \Phi_{\omega,k}^{\omega',k'}}/\va{\nabla \Phi_{\omega,k}^{\omega',k'}} ) \geq - C^{-1} \delta & \max(|k|,|k'|) \\ & = -C^{-1} \max(|k|,|k'|)^{\frac{1}{\sigma}},
\end{split}
\end{equation*}
for some constant $C > 0$. Here, we used the fact that the derivative of $\widetilde{\Phi}_{\omega,k}^{\omega',k'}$ at a real point is $\mathbb{C}$-linear. This estimate with the bound on $R_{\omega,k}^{\omega',k'}$ and \eqref{eq:post_shift} gives the announced result.
\end{proof}

From this estimate, we get immediately the analogue of Lemma \ref{lemma:individual_bound}. Notice that, since we do not consider complex perturbations of $F$, the analogue of \eqref{eq:weak_discard} is the straightforward bound $|\brac{\mathcal{L}_{F,w} \tilde{e}_{k}^\omega, e_{k'}^{\omega'}}| \leq \n{w}_\infty$.

\begin{lemma}\label{lemma:individual_bound_gevrey}
Let $\epsilon > 0$. There are constants $C,\tau,\gamma_0 > 0$ such that if $w \in \mathcal{G}_\epsilon^s$ and $0 < \gamma \leq \gamma_0$ then for every $(\omega,k) \in \Gamma$ we have
\begin{equation*}
\n{\mathcal{L}_{\Psi_X \circ F,w} \tilde{e}_{k}^\omega}_\gamma \leq C \n{w}_{\mathcal{O}_\epsilon} e^{- \tau \va{k}^{\frac{1}{\sigma}} - \gamma G_\omega(k)}.
\end{equation*}
\end{lemma}

As in \S \ref{section:construction}, we use this bound to define the operator $\widetilde{\mathcal{L}}_{F,w}$ by the formula
\begin{equation}\label{eq:def_L_tilde_gevrey}
\widetilde{\mathcal{L}}_{F,w} = \sum_{\omega \in \Omega,k \in \mathbb{Z}^d} \iota(\mathcal{L}_{F,w} \tilde{e}_{k}^\omega) \otimes l_{e_k^\omega}.
\end{equation}
It follows immediately from Lemmas \ref{lemma:concrete} and \ref{lemma:individual_bound_gevrey} that $\widetilde{\mathcal{L}}_{F,w}$ is an operator in the exponential class $1/(d \sigma)$. One can then check that these operators fulfill the conclusions of Theorem \ref{theorem:anisotropic_space_gevrey} (to prove the point \ref{item:integrals_gevrey}, one may follow the lines of the proof of Lemma \ref{lemma:integrals_and_operators}).

\subsection{Consequences}

As we mentioned, Theorem \ref{theorem:anisotropic_space_gevrey} implies Theorem \ref{theorem:gevrey_determinant} and Corollary \ref{corollary:upper_bound_resonances_gevrey}. Let us sketch the proof of another consequence of Theorem \ref{theorem:anisotropic_space_gevrey}. We consider a closed real-analytic manifold $M$. Let us consider the space $\mathcal{G}^{1+}$ of functions that are $\sigma$-Gevrey for every $\sigma >1$. Similarly, we say that a map $F: M \to M$ is $1+$-Gevrey if it is $\sigma$-Gevrey for every $\sigma > 1$.

We endow $\mathcal{G}^{1+}$ with a structure of Fréchet space (in particular, this is a Baire space) by endowing it with the semi-norms $\n{\cdot}_{\sigma,\epsilon}$ for $\sigma > 1$ and $\epsilon > 0$. To do so, we require that the charts used to define the norm $\n{\cdot}_{\sigma,\epsilon}$ are real-analytic rather than Gevrey. 

Thanks to this notion, we can state:

\begin{thm}\label{theorem:optimal_strange}
Let $M$ be a closed real-analytic manifold. Let $F: M \to M$ be a $1+$-Gevrey Anosov diffeomorphism. Then the set of $w \in \mathcal{G}^{1+}$ such that
\begin{equation}\label{eq:equality_optimal}
\limsup_{r \to 0} \frac{\log N_{F,w}(r)}{\log \va{\log r}} = d
\end{equation}
is a $G_\delta$ dense subset of $\mathcal{G}^{1+}$.
\end{thm}

In Theorem \ref{theorem:optimal_strange}, we can consider indifferently $\mathcal{G}^{1+}$ functions valued in $\mathbb{R}$ or $\mathbb{C}$.

\begin{proof}
Start by noticing that it follows from Theorem \ref{theorem:gevrey_determinant} that for every $w \in \mathcal{G}^{1+}$ we have
\begin{equation*}
\limsup_{r \to 0} \frac{\log N_{F,w}(r)}{\log \va{\log r}} \leq d.
\end{equation*}
Let $A$ denote the set of $w \in \mathcal{G}^{1+}$ such that \eqref{eq:equality_optimal} holds. It follows from \cite[Theorem 2.7]{gouezel_liverani_1} that $A$ is a $G_\delta$ in the $C^\infty$ topology (and thus in the $\mathcal{G}^{1+}$ topology which is finer).

It remains to prove that $A$ is dense. We will start by proving that $A$ is non-empty. Let $x_0 \in M$ be a periodic point for $F$. Thanks to the expansiveness of $F$ \cite[Corollary 6.4.10]{katok_hasselblatt}, there is an open neighbourhood $U$ of the orbit $\set{T^k x_0: k \in \mathbb{Z}}$ such that if $y \in M$ is such that the orbit of $y$ for $F$ is contained in $U$ then $y$ is a point of the orbit of $x_0$. Let then $w_0$ be a $\mathcal{G}^{1+}$ function supported in $U$ and such that $w_0(y) = 1$ for every $y$ in the orbit of $x_0$.

Now, if $n \geq 1$, we notice that in the sum
\begin{equation*}
\sum_{F^n x = x} \frac{\prod_{k = 0}^{n-1} w_0(F^k x)}{\va{\det\p{ I - D_x F^n}}},
\end{equation*}
the only $x$'s that will have a non-zero contributions are the elements of the orbit of $x_0$, and this only happens if $n$ is a multiple of the minimal period $m$ of $x_0$. Hence,
\begin{equation*}
\sum_{F^n x = x} \frac{\prod_{k = 0}^{n-1} w_0(F^k x)}{\va{\det\p{ I - D_x F^n}}} = 0
\end{equation*}
if $m$ does not divide $n$. Now, if $n = \ell m$ for some integer $\ell$, we find that
\begin{equation*}
\sum_{F^n x = x} \frac{\prod_{k = 0}^{n-1} w_0(F^k x)}{\va{\det\p{ I - D_x F^n}}} = \sum_{p = 0}^{m-1} \frac{1}{\va{\det\p{I - D_{F^p x_0} F^{m \ell}}}} = \frac{m}{\va{\det\p{I - D_{x_0} F^{m \ell}}}}.
\end{equation*}
Let us denote by $\lambda_1,\dots,\lambda_d$ the eigenvalues of $D_{x_0} F^m$. We order them in such a way that, for some $t \in \set{2,\dots,d-1}$, we have $\va{\lambda_j} > 1$ for $j = 1,\dots,t$ and $\va{\lambda_j} < 1$ for $j = t+1,\dots,d$. This way, we have
\begin{equation}\label{eq:explicit_trace}
\begin{split}
\frac{1}{\va{\det\p{I - D_{x_0} F^{m \ell}}}} & = \frac{1}{\prod_{j = 1}^t (\lambda_j^\ell - 1) \prod_{j = t+1}^{d} (1 - \lambda_j^{\ell})} \\
     & = \prod_{j = 1}^{t} \frac{\lambda_j^{- \ell} }{1 - \lambda_j^{- \ell}} \prod_{j = t+1}^{d} \frac{1}{1 - \lambda_j^t} \\
     & = \prod_{j = 1}^t \sum_{k \geq 1} \lambda_j^{- k \ell} \prod_{j = t+1}^{d} \sum_{k \geq 0} \lambda_j^{k \ell} \\
     & = \sum_{\lambda \in \mathcal{Z}} \lambda^{\ell},
\end{split}
\end{equation}
where $\mathcal{Z}$ is the set of complex numbers of the form $\lambda_1^{- k_1} \dots \lambda_{t}^{- k_t} \lambda_{t+1}^{k_{t+1}} \dots \lambda_d^{k_d}$, where $k_1,\dots,k_t$ are larger than or equal to $1$ and $k_{t+1},\dots,k_d$ are larger than or equal to $0$. It is understood that the multiplicity of an element of $\mathcal{Z}$ is the number of ways to write it in the form above.

From \eqref{eq:explicit_trace}, we get that
\begin{equation*}
d_{F,w_0}(z) = \exp\p{- \sum_{\ell = 1}^{+ \infty} \frac{z^{m \ell}}{\ell}\sum_{\lambda \in \mathcal{Z}} \lambda^{\ell}} = \prod_{\lambda \in \mathcal{Z}} (1 - \lambda z^m).
\end{equation*}
It follows that the resonances of $\mathcal{L}_{F,w_0}$ are the $m$th roots of the elements of $\mathcal{Z}$. Counting the elements of $\mathcal{Z}$, we find that there is a constant $C$ such that $N_{F,w_0}(r) \geq C^{-1} |\log r|^d$ for $r$ small enough, and it follows that $w_0 \in A$.

We prove now that $A$ is dense in $\mathcal{G}^{1+}$. Working as in the proof of Lemma \ref{lemma:optimal_order}, we see that if $w \in \mathcal{G}^{1+}$ then the order of growth of $z \mapsto d_{F,w}(e^z)$ is less than $d+1$ with equality if and only if $w \in A$. For $w$ an element of $\mathcal{G}^{1+}$, let us consider the function of two complex parameters $f(u,z) = d_{F, (1-u) w_0 + u w}(e^z)$.

For every $s > 1$, we can find $\epsilon > 0$ such that $w_0$ and $w$ belong to $\mathcal{G}_{\epsilon}^s$. Applying Theorem \ref{theorem:anisotropic_space_gevrey} and then working as in the proof of Proposition \ref{proposition:trace_formula}, we find that the for $u,z \in \mathbb{C}$ we have $f(u,z) = \det\p{I - e^z \widetilde{\mathcal{L}}_{F,(1-u) w_0 + u w}}$. In particular, $f$ is holomorphic in $\mathbb{C}^2$, and for every $u \in \mathbb{C}$, the order of growth of the function $z \mapsto f(u,z)$ is less than $\sigma d +1$. Since this is true for every $\sigma > 1$, the order of growth of this function is actually less than $d+1$. Moreover, we have equality for $u = 0$, since $w_0 \in A$.

Working as in the proof of Theorem \ref{theorem:optimal_dense}, we see that the set of $u \in \mathbb{C}$ such that $(1-u) w_0 + u w$ does not belong to $A$ is contained in a polar set, in particular it has Hausdorff dimension zero (and thus its intersection with $\mathbb{R}$ has empty interior). As a consequence, $w$ belongs to the closure of $A$, and following $A$ is dense in $\mathcal{G}^{1+}$.
\end{proof}

\section*{Acknowledgements}

I would like to thank Oscar Bandtlow and Julia Slipantschuk for explaining the construction from \cite{optimal_examples} to me. The author benefits from the support of the French government “Investissements d’Avenir” program integrated to France 2030, bearing the following reference ANR-11-LABX-0020-01. Most of this work was done while the author was working at Massachusetts Institute of Technology.

\bibliographystyle{alpha}
\bibliography{biblio_Anosov_analytic.bib}

\end{document}